\documentclass[sn-mathphys,Numbered]{sn-jnl}

\usepackage[T1]{fontenc}
\usepackage[%
  shortlabels,
  inline,
]{enumitem}
\usepackage{mathtools}
\usepackage[%
  scr=stixtwofancy, 
  cal=stixtwoplain, 
  bb=stixtwo, 
]{mathalpha}
\usepackage{amsthm}
\usepackage{amssymb}
\usepackage[Symbol]{upgreek}
\usepackage{mleftright}
\usepackage{interval} 
\usepackage{xfrac} 
\usepackage{fixdif}
\usepackage{leftindex}
\usepackage{accents}
\usepackage{aligned-overset}%
\usepackage{mathcommand}
\usepackage{tikz}
\usepackage{hyperref}
\usepackage[title]{appendix}%
\usepackage{manyfoot}
\usepackage[%
  hypcap=true,
  font=small,%
]{caption}
\usepackage[%
  capitalize,
  nameinlink,
]{cleveref}

\hypersetup{%
  linktoc=all, 
  colorlinks=true,
  citecolor=magenta,
}

\crefname{section}{\S}{\S}
\NewDocumentCommand\Crefnameitem { m m m O{\textup} O{(\roman*)}} {%
  \Crefname{#1enumi}{#2}{#3} 
  \AtBeginEnvironment{#1}{%
    \crefalias{enumi}{#1enumi}%
    \setlist[enumerate,1]{
        label={#4{#5}.},
        ref={#5}
    }%
  }  
}
\NewDocumentCommand\itemcref{mm}{%
  \labelcref{#1#2} of \cref{#1}%
}

\numberwithin{equation}{section}
\newcommand\numberthis{\addtocounter{equation}{1}\tag{\theequation}}
\theoremstyle{thmstyleone}%
\newtheorem{theorem}{Theorem}[section]
\newtheorem{proposition}[theorem]{Proposition}
\newtheorem{corollary}[theorem]{Corollary}
\newtheorem{lemma}[theorem]{Lemma}

\theoremstyle{thmstyletwo}%
\newtheorem{remark}[theorem]{Remark}
\newtheorem{warn}[theorem]{Warning}

\theoremstyle{thmstylethree}%

\newtheorem{definition}[theorem]{Definition}
\newtheorem{construction}[theorem]{Construction}
\Crefname{construction}{Construction}{Constructions}

\Crefnameitem{theorem}{Theorem}{Theorems}
\Crefnameitem{lemma}{Lemma}{Lemmas}

\usetikzlibrary{%
  cd, 
  shapes, 
  arrows, 
}
\NewDocumentCommand\placeholder{}{\:\cdot\:} 
\NewDocumentCommand\NewPairedDelimiterS{mmm}{%
  \DeclarePairedDelimiterX{#1}[1]{#2}{#3}%
    {\ifblank{##1}{\placeholder}{##1}}%
}
\NewDocumentCommand\NewPairedDelimiterSS{mmmO{,}}{%
  \DeclarePairedDelimiterX{#1}[2]{#2}{#3}%
    {\ifblank{##1}{\placeholder}{##1}%
    #4%
    \ifblank{##2}{\placeholder}{##2}}%
}
\NewPairedDelimiterS\abs\lvert\rvert 
\NewPairedDelimiterS\norm\lVert\rVert 
\NewPairedDelimiterS\lrangle\langle\longrightarrowngle 
\NewPairedDelimiterS\paren\lparen\rparen 
\NewPairedDelimiterS\lrbrack\lbrack\rbrack 
\NewPairedDelimiterS\lrbrace\lbrace\rbrace 
\NewPairedDelimiterS\ceil\lceil\rceil 
\NewPairedDelimiterS\floor\lfloor\rfloor 
\NewPairedDelimiterS\bra\langle\vert 
\NewPairedDelimiterS\ket\vert\rangle 
\NewPairedDelimiterS\normord{\mathopen{:}}{\mathclose{:}} 
\NewPairedDelimiterSS\braket\langle\rangle[%
  \,\delimsize\vert\,\mathopen{}%
] 
\NewPairedDelimiterSS\pairing\langle\rangle
\NewPairedDelimiterSS\inner\lparen\rparen
\NewPairedDelimiterSS\Liebracket\lbrack\rbrack
\newmathcommand\Lie\Liebracket%
\DeclarePairedDelimiterX\bracket[3]\langle\rangle%
  {\ifblank{#1}{\placeholder}{#1}%
  \,\delimsize\vert\,\mathopen{}%
  \ifblank{#2}{\placeholder}{#2}
  \,\delimsize\vert\,\mathopen{}%
  \ifblank{#3}{\placeholder}{#3}}%
\NewDocumentMathCommand\dparen{m}{\lparen\!\lparen{#1}\rparen\!\rparen}
\NewDocumentMathCommand\dbrack{m}{\lbrack\!\lbrack{#1}\rbrack\!\rbrack}
\providecommand\given{}
\NewDocumentCommand \SetSymbol {o}
  { \nonscript\:#1\vert\allowbreak\nonscript\:\mathopen{} }
\DeclarePairedDelimiterX\Set[1]\{\}
  { \renewcommand\given{\SetSymbol[\delimsize]}#1 }
\DeclarePairedDelimiterX\GSet[1]\langle\rangle
  { \renewcommand\given{\SetSymbol[\delimsize]}#1 }
\intervalconfig{soft open fences}

\RenewDocumentCommand \subset {} {\subseteq}

\renewmathcommand\le\leqslant
\renewmathcommand\ge\geqslant
\NewDocumentCommand \concept {m} {\textbf{#1}}

\DeclareRobustCommand\monomorphism
  {\lhook\joinrel\rightarrow} 
\DeclareRobustCommand\dashrightarrow
  {\relbar\rightarrow} 
\declaremathcommand\dashto{\dashrightarrow}
\NewDocumentCommand \txforall {O{\qquad}} {#1\text{for all}#1}%
\NewDocumentCommand \txand {O{\qquad}} {#1\text{and}#1}%
\NewDocumentCommand \txst {O{~}} {#1\text{s.t.}#1}%
\NewDocumentMathCommand \sequence { m O{1} O{n} }
  { \ensuremath{ {#1}_{#2}, \cdots, {#1}_{#3} } }
\NewDocumentMathCommand \supsequence { m O{1} O{n} }
  { \ensuremath{ {#1}^{#2}, \cdots, {#1}^{#3} } }
\NewDocumentCommand \fun { m e{^_} O{} }
  {%
    \operatorname{#1}%
    \IfValueT{#2}{\sp{#2}}%
    \IfValueT{#3}{\sb{#3}}%
    \ifblank{#4}{}{\mleft(#4\mright)}%
  }
\LoopCommands{%
  {Arg}%
  {Aut}%
  {Coker}%
  {coker}%
  {Cor}%
  {End}%
  {Ext}%
  {Gal}%
  {Gr}%
  {gl}%
  {gr}%
  {Ho}%
  {Hom}%
  {Id}%
  {id}%
  {Ind}%
  {Ker}%
  {Log}%
  {Mod}%
  {Mor}%
  {Ob}%
  {op}%
  {Pic}%
  {Proj}%
  {pr}%
  {pt}%
  {Rad}%
  {Res}%
  {Spec}%
  {Sq}%
  {supp}%
  {Tor}%
  {tr}%
  {wt}%
}[#1]
  {\declaremathcommand#2{\fun{#1}}}

\NewDocumentMathCommand\Dfrac{mm}{%
  \dfrac{\displaystyle #1}{\displaystyle #2}%
}
\NewDocumentMathCommand\odv{m}{\frac{\d}{\d{#1}}}
\NewDocumentMathCommand\pdv{m}{\frac{\partial}{\partial{#1}}}
\LoopCommands{GNZQRCFKTDWM}[#1]
  {\declaremathcommand#2{\mathbb#1}}

\LoopCommands{ADEFGHIJKMNOSUVWY}[c#1]
  {\declaremathcommand#2{\mathcal#1}}

\LoopCommands{mnpq}[#1#1]
  {\declaremathcommand#2{\mathfrak#1}}
\LoopCommands{hgxX}[#1]
  {\declaremathcommand#2{\mathfrak#1}}

\declaremathcommand\iu{\mathbb{i}}
\declaremathcommand\L{\mathcal{L}}
\declaremathcommand\U{\mathscr{U}}
\declaremathcommand\o{\otimes}
\declaremathcommand\spn{\fun{span}}
\declaremathcommand\Cc{\accentset{\circ}{C}}
\declaremathcommand\oC{\overline{C}}
\declaremathcommand\O{\mathcal{O}}
\declaremathcommand\PP{\mathbb{P}}
\declaremathcommand\la{\lambda}
\declaremathcommand\ds{\dots}
\declaremathcommand\al{\alpha}
\declaremathcommand\ra{\longrightarrow}
\declaremathcommand\tilX{\widetilde{\X}}
\declaremathcommand\Xc{\accentset{\circ}{\X}}

\declaremathcommand\Pc{\mathring{\mathbb{P}}^1}
\declaremathcommand\Cfb{\fun{\mathscr{C}}}
\declaremathcommand\Image{\fun{Im}}
\declaremathcommand\dQ{\mathfrak{Q}}

\NewDocumentMathCommand\fusion{O{M^1}O{M^2}O{M^3}}{\textstyle\binom{#3}{#1\;#2}}
\NewDocumentMathCommand\rfusion{O{M^1}O{M^2(0)}O{M^3(0)}}{\textstyle\binom{#3}{#1\;#2}}
\declaremathcommand\Fusion{\mathfrak{I}\fusion}
\declaremathcommand\Nusion{N\fusion}
\newmathcommand\vac{\mathbb{1}}
\NewDocumentMathCommand\delfun{mmO{{#2}^{-1}}}{#3\delta\left(\dfrac{#1}{#2}\right)}
\newmathcommand\del{\fun{\delta}}
\NewDocumentMathCommand\pfrac{mm}{\left(\dfrac{#1}{#2}\right)}
\NewDocumentMathCommand\vo{mm}{#1_{(#2)}}
\NewDocumentMathCommand\lo{mm}{#1({\textstyle #2})}

\declaremathcommand\ptseq{\sequence{\pp}}
\declaremathcommand\zseq{\sequence{z}}
\declaremathcommand\rseq{\sequence{r}}
\declaremathcommand\aseq{\supsequence{a}}
\NewDocumentMathCommand \cfbseq { O{1} O{n} }
  { \ensuremath{ (a^{#1},\pp_{#1})\cdots(a^{#2},\pp_{#2}) } }
\NewDocumentMathCommand \ainVseq { O{1} O{n} }
  { \ensuremath{ a^{#1}\in V^{r_{#1}}, \cdots, a^{#2}\in V^{r_{#2}} } }
\NewDocumentMathCommand \cfbseqb { O{1} O{p} }
  { \ensuremath{ (b^{#1},\pp_{+#1})\cdots(b^{#2},\pp_{+#2}) } }
\NewDocumentMathCommand \lobmseq { O{1} O{p} }
  { \ensuremath{ \lo{b^{#1}}{\frac{r_{#1}}{T}+m_{#1}}\cdots\lo{b^{#2}}{\frac{r_{#2}}{T}+m_{#2}} } }

\NewDocumentMathCommand\charge{m}{V_{L}^{#1}}
\NewDocumentMathCommand\charhalf{mm}{V_{#1+L}^{#2}}
\NewDocumentMathCommand\charlam{m}{V_{#1+L}}

\allowdisplaybreaks
\begin{document}

\title[Twisted restricted correlation functions and fusion rules]{Twisted restricted conformal blocks of vertex operator algebras I: $g$-twisted correlation functions and fusion rules}


\author[1]{\fnm{Xu} \sur{Gao}}\email{gausyu@tongji.edu.cn}
\equalcont{These authors contributed equally to this work.}

\author[2]{\fnm{Jianqi} \sur{Liu}}\email{jliu230@sas.upenn.edu}
\equalcont{These authors contributed equally to this work.}

\author*[3]{\fnm{Yiyi} \sur{Zhu}}\email{yzhu51@scut.edu.cn}
\equalcont{These authors contributed equally to this work.}

\affil[1]{\orgdiv{Department of Mathematics}, \orgname{Tongji University}, \orgaddress{\street{1239 Siping Road}, \city{Shanghai}, \postcode{200092}, \state{Shanghai}, \country{China}}}

\affil[2]{\orgdiv{Department of Mathematics}, \orgname{University of Pennsylvania}, \orgaddress{\street{209 South 33rd Street}, \city{Philadelphia}, \postcode{19104}, \state{PA}, \country{USA}}}

\affil*[3]{\orgdiv{Department of Mathematics}, \orgname{South China University of Technology}, \orgaddress{\street{381 Wushan Road}, \city{Guangzhou}, \postcode{510641}, \state{Guangdong}, \country{China}}}


\abstract{In this paper, we introduce a notion of $g$-twisted restricted conformal block on the three-pointed twisted projective line $\x\colon\oC\to\PP^1$ associated with an untwisted module $M^1$ and the bottom levels of two $g$-twisted modules $M^2$ and $M^3$ over a vertex operator algebra $V$. We show that the space of twisted restricted conformal blocks is isomorphic to the space of $g$-twisted (restricted) correlation functions defined by the same datum and to the space of intertwining operators among these twisted modules. As an application, we derive a twisted version of the Fusion Rules Theorem.}

\keywords{vertex operator algebra, fusion rule, twisted module, restricted conformal block}


\pacs[MSC Classification]{17B69, 81T40}

\maketitle
    
\tableofcontents
\section{Introduction}
This is the first paper in a series aiming to explore the general theory of twisted (restricted) conformal blocks of vertex operator algebras. In this paper, we focus on their $g$-twisted correlation functions and the fusion rules among $g$-twisted modules. 

The concept of \emph{twisted} representations of \emph{vertex operator algebras} (VOAs for short) originated in the realization of irreducible representations of twisted affine Lie algebras \cite{Lepowsky}, the construction of the celebrated moonshine module \cite{FLM}, and the study of orbifold models in conformal field theory \cite{DHVW86,DVVV89}. 
Over the past few decades, twisted representations have been extensively studied, e.g. \cite{DL,X95,DLM1,DLM00,DY02,BDM02}. 
One of the most notable applications is in the \emph{orbifold theory} of VOAs \cite{DLM00,DY02,CM16,DRX17}. 
The renowned \emph{orbifold conjecture} posits that every irreducible module over the fixed-point subVOA $V^G$ of $V$ under some finite automorphism group $G<\Aut(V)$ occurs in an irreducible $g$-twisted $V$-module for some $g\in G$, and if $V$ is strongly rational, then $V^G$ also follows suit. 
Recent breakthroughs have established the validity of this conjecture for cyclic groups \cite{M15,CM16,DRX17}. 
This has led to numerous new examples of strongly rational VOAs with irreducible modules emerging as direct summands of certain twisted modules.

In the landscape of VOA theory and the associated \emph{conformal field theory} (CFT for short), a crucial challenge is ascertaining the fusion algebra within the module category, entailing the computation of {fusion rules} among irreducible modules. 
By definition, the \emph{fusion rule} associated to $V$-modules $M^1$, $M^2$, and $M^3$ is the dimension of the space of \emph{intertwining operators} among them. 
In the context of certain rational orbifold CFTs, the application of the renowned \emph{Verlinde Formula} \cite{V88} has allowed the determination of fusion rules through a concrete description of the $S$-matrix in their modular transformations \cite{DVVV89}. 
On the VOA side, due to the intricate nature of twisted irreducible representations, the fusion rules were only established for certain $\Z/2\Z$ or $\Z/3\Z$-orbifold lattice VOAs. 
For instance, in the case of the $\theta$-cyclic orbifold VOAs $M(1)^+$ and $V_L^+$ introduced in \cite{FLM}, fusion rules were determined in \cite{ADL05} through explicit constructions of twisted intertwining operators for Heisenberg and lattice VOAs. 
In general, when dealing with an arbitrary strongly rational VOA $V$, the connection between fusion rules among ordinary modules over the orbifold VOA $V^G$ and fusion rules among twisted modules over $V$ remains elusive, and a unified method for computing fusion rules among twisted modules is currently lacking.

Let $V$ be a VOA, and let $g_1, g_2$, and $g_3$ be three finite-order automorphisms of $V$. The concept of \emph{twisted intertwining operators} among $g_1, g_2,$ and $g_3$-twisted $V$-modules $M^1$, $M^2$, and $M^3$ was initially introduced by Xu in \cite{X95}. 
Xu's definition generalizes the usual \emph{Jacobi identity} of untwisted intertwining operators in \cite{FHL} by incorporating factors involving rational powers of formal variables.
In addition, Huang has further extended the notion of twisted modules and twisted intertwining operators to arbitrary (not necessarily commuting or of finite order) automorphisms $g_1,g_2,$ and $g_3$ in \cite{H10,H18} by generalizing the duality properties of untwisted intertwining operators.  

One approach to fusion rules from the geometric side is by exploring \emph{conformal blocks} on algebraic curves associated to modules/sectors. 
Notably, the isomorphism between the space of \emph{correlation functions} of conformal blocks on the three-pointed complex projective line $(\PP^1, \infty,1,0)$ associated to irreducible $V$-modules $M^2$, $(M^3)'$, and $M^1$, and the space of intertwining operators of type $\fusion$ is well-known, as established in \cite{TUY89,MS89,Z94,Liu}. 
Furthermore, conformal blocks can be reconstructed from their restrictions on the \emph{bottom levels} $M^2(0)$ and $M^3(0)^{\ast}$, as established in \cite{Z,Liu}. 
Then, in this context, the fusion rule $\Nusion$ can be computed through the modules $M^2(0)$ and $M^3(0)^{\ast}$, and the bimodule $A(M^1)$ over \emph{Zhu's algebra} $A(V)$ for the VOA $V$.
This is the well-known \emph{Fusion Rules Theorem} claimed in {\cite{FZ}}. 
While the concept of conformal block for twisted modules has been formulated by Frenkel and Szczesny in \cite{FBZ04,FS04}, the twisted version of the aforementioned story remains unexplored.

In the present work, we address these questions in the simplest nontrivial scenario where $g_1=1$ and $g_2=g_3$, namely, when the $V$-module $M^1$ is untwisted, while $M^2$ and $M^3$ are both $g$-twisted for some automorphism $g$ of order $T<\infty$. We will refer to this scenario as the \emph{$g$-twisted} case. 
Let $I$ be a twisted intertwining operator among $M^1$, $M^2$, and $M^3$. 
In order to accommodate the rational powers $z^{1/T}$ and $w^{1/T}$ occurring in the twisted fields $Y_{M^2}(-,z)$, $Y_{M^3}(-,z)$, and $I(-,w)$ simultaneously, we introduce the \emph{$T$-twisted projective line} $\x\colon\oC\to\PP^1$. 
On this curve, we attach $M^2$ to $0$, $(M^3)'$ to $\infty$, and $M^1$ to a point $1\in \PP^1$ that is other than $0$ and $\infty$. 
In the spirit of \cite{TUY89,Z,Liu}, the space of $g$-twisted correlation functions associated to the datum 
$(\x\colon\oC\to\PP^1, \infty, 1, 0, (M^3)', M^1, M^2)$ can be defined by axiomizing the behaviors of the limit function (on $\oC$) of the \emph{Puiseux} series 
\begin{equation}\label{eq:TheSeries}
  \braket*{v'_3}{Y_{M^3}(a^{1},z_{1})\cdots Y_{M^3}(a^{k-1},z_{k-1})I(v,w)Y_{M^2}(a^{k},z_{k})\cdots Y_{M^2}(a^{n},z_{n})v_2}w^{h},
\end{equation} 
where $v'_3\in (M^3)'$, $v_2\in M^2$, $v\in M^1$, and $a^i\in V$.
Denote this space by $\Cor\fusion$. 
For the generality, we introduce a space $\Cor[\Sigma_{1}(N^3, M^1, M^2)]$ of \emph{$g$-twisted correlation functions} associated to the datum 
$\Sigma_{1}(N^3, M^1, M^2):=(\x\colon\oC\to\PP^1, \infty, 1, 0, N^3, M^1, M^2)$, where $N^3$ is an arbitrary $g^{-1}$-twisted $V$-module. See \cref{warn}.
Our first main theorem (\cref{thm:I=Cor}) establishes an isomorphism between $\Cor \fusion$ and the space $\Fusion$ of \emph{$g$-twisted intertwining operators}. 
To extend our construction to the general case where $M^1$ is also twisted, we have to introduce twisted curves of higher genus. 
For instance, the case when $g_3=g_1g_2=g_2g_1$ and $g_1^T=g_2^T=1$ involves the \emph{Fermat curve} of degree $T$, which has genus $\frac{(T-1)(T-2)}{2}$. This will be addressed in a subsequent work. 

To extend the \emph{Fusion Rules Theorem} to the $g$-twisted scenario, we introduce an auxiliary space $\Cor(\Sigma_{1}(U^3, M^1, U^2))$ of \emph{$g$-twisted restricted correlation functions} associated to the datum $\Sigma_{1}(U^3, M^1, U^2):=(\x\colon\oC\to\PP^1, \infty, 1, 0,$ $ U^3,$ $ M^1, U^2)$, where $U^2$ (resp. $U^3$) is an irreducible left (resp. right) module over the \emph{$g$-twisted Zhu's algebra} $A_g(V)$ introduced in \cite{DLM1}.
The axioms we impose on $\Cor (\Sigma_{1}(U^3, M^1, U^2))$ are based on behaviors of the limit function of the Puiseux series \labelcref{eq:TheSeries}, where $u_2\in M^2(0)$ and $u'_3\in M^3(0)^{\ast}$. 
A crucial difference between our axioms on $\Cor (\Sigma_{1}(U^3, M^1, U^2))$ and those in \cite{Liu} is the additional non-integer shifting of the coefficient functions $F_{n,i}(\pp,\qq)$ in the recursive formulas, due to the ramification of $\oC$ at the points $0$ and $\infty$. 
Our second main theorem (\cref{thm:CorBottom}) establishes an isomorphism between the space of $g$-twisted restricted correlation functions $\Cor (\Sigma_{1}(U^3, M^1, U^2))$ and the space of $g$-twisted correlation functions $\Cor[\Sigma_{1}(\overline{M}(U^3), M^1, \overline{M}(U^2))]$, where $\overline{M}(U^2)$ is the \emph{$g$-twisted generalized Verma module} associated to $U^2$, and $\overline{M}(U^3)$ is the \emph{$g^{-1}$-twisted generalized Verma module} associated to the right $A_g(V)$-module $U^3$ \cite{DLM1}. In the untwisted scenario, it was pointed out by Li in \cite{Li} that the \emph{Fusion Rules Theorem} does not hold for arbitrary $M^2$ and $M^3$.

The axioms of $\Cor (\Sigma_{1}(U^3, M^1, U^2))$ imply, in particular, that a system of correlation functions $S$ defines a linear functional $\varphi_S\colon u_3\otimes v\otimes u_2\mapsto S\bracket*{u_3}{(v,\qq)}{u_2}w^{\deg v}\in\C$ on the vector space $U^3\otimes M^1\otimes U^2$. 
Furthermore, the linear functional $\varphi_S$ vanishes on a subspace $J$ whose definition will be provided in \cref{def:resCfb}. The vanishing of $\varphi_S$ on $J$ can be interpreted as being invariant under the actions of the twisted chiral Lie algebras constrained at $\infty$ and $0$.
We call the space $(U^3\otimes M^1\otimes U^2)/J$ the space of \emph{$g$-twisted restricted coinvariants}, and the dual space $((U^3\otimes M^1\otimes U^2)/J)^\ast$ the space of \emph{$g$-twisted restricted conformal blocks}, denoted by $\Cfb[U^3, M^1, U^2]$. 
In \cref{sec4} and \cref{sec5}, we show that there is a one-to-one correspondence between the space of $g$-twisted restricted conformal blocks and the space of $g$-twisted restricted correlation functions (\cref{thm:iso-restrictcfb-bottomcorrelation}) by reconstructing a system of correlation functions from a given restricted conformal block $\varphi$ using the recursive formulas. 
In the twisted case, the occurrence of non-integer shifting of the coefficient functions $F_{n,i}(\pp,\qq)$ poses several challenges. Our main theorem in \cref{sec4} and \cref{sec5} can also be viewed as the $g$-twisted and restricted version of the ``propagation of vacua'' theorem in \cite{TUY89,Z94,FBZ04,NT}. 

The following diagram summarizes our main theorems in \cref{sec2}--\cref{sec5}, where we assume that $M^1$ is an untwisted $V$-module, and $M^2$ and $M^3$ are admissible $g$-twisted $V$-modules such that $M^2$ and $(M^3)'$ are generalized Verma modules, with bottom levels $M^2(0)=U^2$ and $M^3(0)^\ast=U^3$ being irreducible $A_g(V)$-modules:   
\[
  \begin{tikzcd}[column sep=1in, row sep=0.7in]
  \Cor\fusion\arrow[r, "{\text{Theorem } \labelcref{thm:I=Cor}}"]\arrow[d, "\text{Theorem }\labelcref{thm:CorBottom}"] & \Fusion\arrow[l]\arrow[d,dashed, "{\text{}}"] \\
  \Cor[\Sigma_{1}(U^3, M^1, U^2)]\arrow[r, "\text{Theorem }\labelcref{thm:iso-restrictcfb-bottomcorrelation}"]\arrow[u] &  \Cfb[U^3, M^1, U^2]\arrow[l]\arrow[u,dashed]
  \end{tikzcd}
\]
In particular, when $V$ is $g$-rational \cite{DLM1}, the space $\Fusion$ of intertwining operators is isomorphic to the space of $g$-twisted restricted conformal blocks $\Cfb[U^3, M^1, U^2]$, for arbitrary irreducible $g$-twisted $V$-modules $M^2$ and $M^3$. In a subsequent paper, we will introduce the notions of twisted conformal blocks $\Cfb[\Sigma_{1}(N^3, M^1, M^2)]$ and twisted restricted conformal blocks $\Cfb[\Sigma_{1}(U^3, M^1, U^2)]$ using the actions of (constrained) twisted chiral Lie algebra, and demonstrate the isomorphisms in the following diagram: 
\begin{equation*}
  \begin{tikzcd}[column sep=0.5in, row sep=0.3in]
  \Cor[\Sigma_{1}(N^3, M^1, M^2)]\arrow[r]\arrow[d] & \Cfb[\Sigma_{1}(N^3, M^1, M^2)]\arrow[l]\arrow[d] \\
  \Cor[\Sigma_{1}(U^3, M^1, U^2)]\arrow[r]\arrow[u] & \Cfb[\Sigma_{1}(U^3, M^1, U^2)].\arrow[l] \arrow[u]
  \end{tikzcd}
\end{equation*}

Furthermore, we will show in \cref{sec6} that the space of $g$-twisted restricted coinvariants $(U^3\otimes M^1\otimes U^2)/J$ is isomorphic to both $U^3\otimes_{A_g(V)} B_{g,h}(M^1)\otimes_{A_g(V)}U^2$ and $U^3\otimes_{A_g(V)} A_{g}(M^1)\otimes_{A_g(V)}U^2$, where $B_{g,h}(M^1)$ is an $A_g(V)$-bimodule generalizing $B_h(M^1)$ in \cite{Liu}, and $A_g(M^1)$ is an $A_g(V)$-bimodule constructed in \cite{JJ} that generalizes $A(M^1)$ in \cite{FZ}. Consequently, we have multiple methods to compute the fusion rules $\Nusion$ when $M^2$ and $(M^3)'$ are $g$-twisted generalized Verma modules.
Notably, the isomorphism \[\Fusion\cong (M^3(0)^{\ast}\otimes_{A_g(V)} A_{g}(M^1)\otimes_{A_g(V)}M^2(0))^{\ast}\] extends the renowned \emph{(untwisted) Fusion Rules Theorem} in \cite{FZ,Li,HY,Liu} to the $g$-twisted case. 
We also deduce several applications of the \emph{$g$-twisted Fusion Rules Theorem}. First, we establish the finiteness of $g$-twisted fusion rules under the assumption that $V$ is $C_2$-cofinite. Secondly, when $V$ is strongly rational, using the main theorem of \cite{CM16}, we find the relation between $g$-twisted fusion rules among irreducible $g$-twisted $V$-modules and the ordinary fusion rules among irreducible $V^0$-modules by decomposing $M^1$, $M^2$, and $M^3$ into direct sums of irreducible modules over $V^0$.

Lastly, in \cref{sec7}, we determine the fusion rules among irreducible $\theta$-twisted modules over the \emph{Heisenberg VOA} $M(1)$ and rank one \emph{lattice VOA} $V_{L}$ with $L=\Z\al$ and $(\al|\al)=2$. This is achieved through the calculation of $A_\theta(M(1,\la))$ and $A_{\theta}(V_{L+\frac{1}{2}\al})$, where $\theta$ is the standard involution of $M(1)$ and $V_L$ \cite{FLM}. In these examples, the $\theta$-twisted fusion rules encompass all possibilities of fusion rules among $\theta$-twisted modules, given that $\theta^{2}=1$.

This paper is structured as follows. In \cref{sec2}, we introduce the \emph{twisted projective line} $\x\colon\oC\to\PP^1$ and the space $\Cor[\Sigma_{1}(N^3, M^1, M^2)]$ of \emph{$g$-twsited correlation functions}. The key result in this section establishes the isomorphism between $\Cor\fusion$ and $\Fusion$. 
In \cref{sec3}, we introduce the space $\Cor(\Sigma_{1}(U^3, M^1, U^2))$ of \emph{$g$-twisted restricted correlation functions} and demonstrate its isomorphism to $\Cor[\Sigma_{1}(\overline{M}(U^3), M^1, \overline{M}(U^2))]$.
In \cref{sec4}, we reconstruct a system of correlation functions $S_\varphi$ from a $g$-twisted restricted conformal block $\varphi$ in $\Cfb[U^3, M^1, U^2]$ and establish the locality of $S_\varphi$. In \cref{sec5}, we demonstrate the associativity and other axioms of the reconstructed $S_\varphi$. In \cref{sec6}, we prove the \emph{$g$-twisted fusion rules theorem} and discuss its applications. Finally, in \cref{sec7}, we compute the fusion rules among $\theta$-twisted modules over the Heisenberg VOAs and the rank one lattice VOA using the \emph{$g$-twisted fusion rules theorem}. 

\subsection*{Convention} 
In this paper, we adopt a specific formatting convention to enhance clarity. Text that we want to emphasize, terms with clear contextual meanings, or results available in standard textbooks will be presented in \emph{italic font}. Whereas terminology introduced in the context will be in \concept{bold font}.

We adhere to the following mathematical notation: $\N$ denotes the set of natural numbers, including $0$; $\Z$ stands for the ring of integers; $\Q$ represents the field of rational numbers, and $\C$ denotes the field of complex numbers. 
All vector spaces are defined over $\C$. 
Tensor products are over $\C$ unless otherwise specified.



\section{Space of twisted correlation functions}\label{sec2}

\subsection{Preliminaries}  
  

\emph{Throughout this article, we fix a VOA $(V,Y,\vac,\upomega)$ and an automorphism $g\in \Aut(V)$ of order $T$.} 
The VOA $V$ is then decomposed into $g$-eigenspaces
\[
  V^r=\Set*{ a\in V \given g.a=e^{2\pi\iu\frac{r}{T}}a }.
\] 
Notably, $V^0$ forms a subVOA of $V$, and each $V^r$ serves as a module over $V^0$, utilizing the same vertex operator $Y$ (cf. \cite{DL,DLM1,FLM}). 
Throughout this article, we keep the convention $0\le r\le T-1$ for the superscript $r$. 
Unless otherwise specified, when we talk about (weak, admissible, etc.) \emph{module}, we mean (weak, admissible, etc.) \emph{$V$-modules}. 
We will consistently use notations like $\lo{a}{n}$ to denote the elements in the Lie algebra $\L_g(V)$ (cf. \labelcref{eq:def:LgV}) and notations like $\vo{a}{n}$ to denote the components of a vertex operator $Y(a,z)$ or an intertwining operator $I(a,z)$.

\subsubsection*{Twisted modules and the twisted Jacobi identity}
Recall the following definition: 
\begin{definition}[\cite{DLM1}]\label{def:twistedmodule}
  A \concept{weak $g$-twisted $V$-module} is a vector space $M$ equipped with a linear map  
  \[
    Y_{M}\colon V\longrightarrow \End(M)\{z\},\qquad 
    a\longmapsto Y_{M}(a,z)=\sum_{n\in \Q} \vo{a}{n} z^{-n-1},
  \]
  satisfying the following axioms for all $a\in V^r$, $b\in V$, and $u\in M$:
  \begin{itemize}
    \item \textbf{Index property}: $Y_M(a,z)=\sum_{n\in \frac{r}{T}+\Z} \vo{a}{n} z^{-n-1}$. 
    \item \textbf{Truncation property}: $\vo{a}{n}u=0$ for $n\gg 0$.
    \item \textbf{Vacuum property}: $Y_{M}(\vac,z)=\id_{M}$.
    \item \textbf{Twisted Jacobi identity}: 
    \begin{equation}\label{eq:Jac}
    \begin{aligned}
      \MoveEqLeft
      z_{0}^{-1}
      \del[\dfrac{z_{1}-z_{2}}{z_{0}}]
      Y_{M}(a,z_{1})Y_{M}(b,z_{2})u
      -
      z_{0}^{-1}
      \del[\dfrac{-z_{2}+z_{1}}{z_{0}}]
      Y_{M}(b,z_{2})Y_{M}(a,z_{1})u\\
      &=
      z_{2}^{-1}
      \pfrac{z_{1}-z_{0}}{z_{2}}^{-\frac{r}{T}}
      \del[\dfrac{z_{1}-z_{0}}{z_{2}}]
      Y_{M}(Y(a,z_{0})b,z_{2})u.
    \end{aligned}
    \end{equation}
  \end{itemize}  

  A weak $g$-twisted $V$-module $M$ is called an \concept{admissible $g$-twisted $V$-module} if it admits a subspace decomposition $M=\bigoplus_{n\in \frac{1}{T}\N}M(n)$ such that 
  \begin{equation}\label{eq:admissible}
    \vo{a}{m}M(n)\subset M(\wt a-m-1+n)
  \end{equation}
  for any homogeneous $a\in V$, any $m\in \Z$, and any $n\in \frac{1}{T}\N$.

  A weak $g$-twisted $V$-module $M$ is called a \concept{$g$-twisted $V$-module} if $\vo{L}{0}$ acts on it semi-simply with finite dimensional eigenspaces $M_\lambda$, and the following property holds: for each $\lambda\in \C$, the eigenspace $M_{\lambda+\frac{n}{T}}$ vanishes when $n\in \Z$ is sufficiently small. 
\end{definition}

For a \emph{formal Puiseux series} $f(z)\in \C\dbrack{z^{\frac{1}{T}}}$, we employ the symbol $\Res_{z} f(z)$ for the coefficient of $z^{-1}$ in $f(z)$. 
Multiplying the \emph{twisted Jacobi identity} \labelcref{eq:Jac} with $z^{m+\frac{r}{T}}_1z_2^{n+\frac{s}{T}}z_0^l$ and then applying $\Res_{z_0}\Res_{z_1}\Res_{z_2}$, we obtain its component form as follows: 
\begin{lemma}\label{lem:Jac-comp}
  Let $M$ be a weak $g$-twisted module. 
  Then, for any $a\in V^r$, $b\in V^s$,
  \begin{equation}\label{eq:Jac'}
  \begin{aligned}
    \MoveEqLeft
    \sum_{i\ge 0} 
      \binom{l}{i} (-1)^i 
      \vo{a}{\frac{r}{T}+m+l-i}
      \vo{b}{\frac{s}{T}+n+i}
    -\sum_{i\ge 0} 
      \binom{l}{i} (-1)^{l+i} 
      \vo{b}{\frac{s}{T}+n+l-i}
      \vo{a}{\frac{r}{T}+m+i}\\
    &=\sum_{j\ge 0}
      \binom{m+\frac{r}{T}}{j} 
      \vo{(\vo{a}{j+l}b)}{\frac{r+s}{T}+m+n-j}
  \end{aligned}
  \end{equation}
  holds for all $m,n,l\in \Z$, where $\vo{a}{j+l}b:=\Res_{z}z^{j+l}Y(a,z)b$. 
\end{lemma}

By \cref{lem:Jac-comp} and taking into account \cite[(3.5)]{DLM1}, we can readily establish a twisted version of the \emph{duality property}:
\begin{proposition}\label{prop:Jac-fun}
  Let $(M,Y_M)$ be an admissible $g$-twisted module and $M'$ be its graded dual space. 
  Then, for any $a\in V^r$, $b\in V^s$, and any $u\in M$, $u'\in M'$, there exists a rational function $f(z_1,z_2)$ with possible poles only at $z_1=0,z_2=0$, and $z_1=z_2$,
  such that the following identities of formal Laurent\footnote{Note that there are no fractional powers involved.} series hold:
  \begin{align*}
    \braket*{u'}{Y_M(a,z_1)Y_M(b,z_2)u}
    z^{\frac{r}{T}}_1z^{\frac{s}{T}}_2
    &=\iota_{z_1,z_2} f(z_1,z_2),
    \\ 
    \braket*{u'}{Y_M(b,z_2)Y_M(a,z_1)u}
    z^{\frac{r}{T}}_1z^{\frac{s}{T}}_2
    &=\iota_{z_2,z_1}f(z_1,z_2),
    \\
    \braket*{u'}{Y_M(Y(a,z_1-z_2)b,z_2)u}
    (z_2+z_1-z_2)^{\frac{r}{T}}z^{\frac{s}{T}}_2
    &=\iota_{z_2,z_1-z_2}f(z_1,z_2),
  \end{align*}
  where $\iota_{z_1,z_2},\iota_{z_2,z_1}$, and $\iota_{z_2,z_1-z_2}$ send a rational function $f(z_1,z_2)$ to its Laurent series expansions in the domains $|z_1|>|z_2|$, $|z_2|>|z_1|$, and $|z_2|>|z_1-z_2|$ respectively. 
  Furthermore, the component form of the twisted Jacobi identity \labelcref{eq:Jac'} is equivalent to the existence of such a rational function.
\end{proposition}

\subsubsection*{The associated Lie algebra and lowest-weight modules} 
Let's recall the Lie algebra $\mathcal{L}_g(V)$ associated to 
a VOA $V$ and an automorphism $g$ of $V$ with order $T$, 
as introduced in \cite{DLM1}. 
The automorphism $g$ can be extended to the vertex algebra $V\otimes\C[t^{\pm\sfrac{1}{T}}]$ by
\begin{equation*}
  g(a\otimes t^\frac{m}{T}):=e^{-2\pi\iu\frac{m}{T}}(ga\otimes t^\frac{m}{T}).
\end{equation*}
Denote the $g$-invariant subspace of $V\otimes\C[t^{\pm\sfrac{1}{T}}]$ by $\mathcal{L}(V, g)$.
It is clear that $\mathcal{L}(V, g)$ is a sub-vertex algebra of $V\otimes\C[t^{\pm\sfrac{1}{T}}]$, with the translation operator $\nabla:= \vo{L}{-1}\otimes\id + \id\otimes\pdv{t}$. 
Then, $\mathcal{L}_g(V)$ is the quotient
\begin{equation}\label{eq:def:LgV}
  \mathcal{L}_g(V):=\mathcal{L}(V, g)/\nabla\mathcal{L}(V, g).
\end{equation}
For any $m\in \Z$ and $a\in V$, we denote the equivalent class of $a\otimes t^\frac{m}{T}$ in $\mathcal{L}_g(V)$ by $\lo{a}{\frac{m}{T}}$. 
Then, $\mathcal{L}_g(V)$ is a Lie algebra, with the Lie bracket given by
\[
  \Lie*{\lo{a}{m+\frac{r}{T}}}{\lo{b}{n+\frac{s}{T}}}
  =\sum_{j\ge 0}\binom{m+\frac{r}{T}}{j} \lo{(\vo{a}{j}b)}{m+n+\frac{r+s}{T}-j},
\]
for any $a\in V^r,b\in V^s$, and $m,n\in \Z$. 
Moreover, $\L_g(V)$ has a natural gradation given by $\deg \lo{a}{\frac{m}{T}}=\wt a-\frac{m}{T}-1$, 
where $m\in\Z$ and $a$ is a homogeneous element of $V$. 
Let $\L_g(V)_n$ be the subspace of $\L_g(V)$ spanned by elements of degree $n\in \frac{1}{T}\Z$. 
Then, we have a triangular decomposition:
\begin{equation*}
  \L_g(V)=\L_g(V)_-\oplus \L_g(V)_0\oplus \L_g(V)_+,
\end{equation*}
where $\L_g(V)_{\pm}=\bigoplus_{n\in \frac{1}{T}\mathbb{Z}_{>0}}\L_g(V)_{\pm n}$.
Recall the following result in \cite{DLM1,Li}:
\begin{proposition}\label{prop:repnOfLgV}
  Let $M$ be a weak $g$-twisted module. Then, the linear map
  \begin{equation*}
    \L_g(V)\longrightarrow \End(M), \quad 
    \lo{a}{\frac{m}{T}}\longmapsto 
    \Res_{z} Y_{M}(a,z)z^{\frac{m}{T}}
  \end{equation*} 
  defines a representation of the Lie algebra $\L_g(V)$ on $M$. 
  Furthermore, if $M$ is equipped with a $\frac{1}{T}\N$-gradation, then $M$ is an admissible $g$-twisted module if and only if $M$ is a graded module for the graded Lie algebra $\L_g(V)$.
\end{proposition}
Recall the following definition in \cite{Li}: 
\begin{definition} \label{def:lowest-weight}
  A weak $g$-twisted module $M$ is called a \concept{lowest-weight module} if there exists $h\in\C$ such that the $\vo{L}{0}$-eigenspace $M_h$ with eigenvalue $h$ is an irreducible $\L_g(V)_0$-module, and $M=\U(\L_g(V)_+)M_h$. 
\end{definition}
If this is the case, then $\vo{L}{0}$ acts on $M$ semi-simply with eigenvalues in $h+\frac{1}{T}\Z$. 
We denote the eigenspace $M_{h+\frac{n}{T}}$ by $M(\tfrac{n}{T})$, and write
$\deg u=\frac{n}{T}$ and $\wt u=h+\frac{n}{T}$ for any homogeneous element $u\in M(\tfrac{n}{T})$. 
Then, $M=\oplus_{n\in \N}M(\frac{n}{T})$ is an admissible $g$-twisted module. 
\begin{definition}
  An admissible $g$-twisted module $M=\bigoplus_{n\in\N}M(\frac{n}{T})$ is said to be \concept{of conformal weight $h\in\C$}, if  $\vo{L}{0}$ acts on $M$ semi-simply and each eigenspace $M_{h+\frac{n}{T}}$ is precisely $M(\frac{n}{T})$.
\end{definition}

\begin{definition}
  Let $M=\bigoplus_{n\in \mathbb{N}}M(\frac{n}{T})$ be an admissible $g$-twisted module. 
  Then, its \emph{graded dual space}
  $M'=\bigoplus_{n\in \mathbb{N}}M(\frac{n}{T})^{\ast}$ 
  naturally carries \emph{right} $g$-twisted vertex operators given by compositions $u'\circ Y_{M}$ ($u'\in M'$). 
  Such a structure induces an usual admissible $g^{-1}$-twisted module structure $Y_{M'}$ defined as
  \begin{equation}\label{eq:def:contragredient}
    Y_{M'}(a, z)u' := u'\circ Y_{M}(e^{z\vo{L}{1}}(-z^{-2})^{\vo{L}{0}}a, z^{-1}), 
  \end{equation}
  where $a\in V$ and $u'\in M'$. This module is called the \textbf{contragredient module} of $M$; refer to \cite{FHL,X95} for more details. 
\end{definition}
\begin{warn}\label{warn}
  For an admissible $g$-twisted module $M$, its components $M(\frac{n}{T})$ need NOT be finite-dimensional. Consequently, its \emph{double} contragredient module $M''$ is not necessarily equal to $M$ itself.
  Hence, in general, given a $g^{-1}$-twisted module $N$, there is no guarantee that there exists an admissible $g$-twisted module $M$ such that $M' = N$.
\end{warn}

Note that \labelcref{eq:def:contragredient} implies
\begin{equation}\label{eq:contragredient'}
  \braket*{Y_{M'}(e^{z\vo{L}{1}}(-z^{-2})^{\vo{L}{0}}a, z^{-1})u'}{u} = 
  \braket*{u'}{Y_{M}(a,z)u}.
\end{equation}
This allows us to spell out a translation between the \emph{right} action of $\L_g(V)$ and the \emph{left} action of $\L_{g^{-1}}(V)$ induced from the contragredient vertex operator $Y_{M'}$.

Indeed, for any $a\in V$ and $m\in\Z$, define 
\begin{equation}\label{eq:def:theta}
  \theta(\lo{a}{\frac{m}{T}}):=\sum_{j\ge 0} \frac{(-1)^{\wt a}}{j!}
  \lo{(\vo{L}{1}^ja)}{\wt a-j-1 + \deg\lo{a}{\frac{m}{T}}}.
\end{equation}
Then, $\theta$ is an anti-isomorphism between $\L_g(V)$ and $\L_{g^{-1}}(V)$ and we have
\begin{equation}\label{eq:contragredient'-comp}
  \theta(\lo{a}{\frac{m}{T}})u'=
  u'\circ\vo{a}{\frac{m}{T}},\qquad
  a\in V, u'\in M'.
\end{equation}
In particular, we have
\begin{equation*}
  \theta(\lo{a}{\frac{m}{T}})M'(n)\subset M'(n-\deg{\lo{a}{\frac{m}{T}}})
\end{equation*}

\subsubsection*{$g$-twisted intertwining operators}

The following definition can be found in \cite{X95}:
\begin{definition}
  Let $M^1$ (resp. $M^2$ and $M^3$) be a weak \emph{untwisted} (resp. \emph{$g$-twisted}) module. 
  An \concept{$g$-twisted intertwining operator of type $\fusion$} is a linear map 
  \[
    I(\cdot, w)\colon 
    M^1 \longrightarrow \Hom(M^2, M^3)\{w\},\quad v\longmapsto I(v,w)=\sum_{m\in \C} v_{m} w^{-m-1},
  \] 
  satisfying the following axioms:
  \begin{itemize}  
    \item \textbf{Truncation property}: For any $v\in M^1$, $ v_2\in M^2$, and a fixed $\lambda\in \C$, we have 
    \[{v}_{\lambda+n}v_2=0\] whenever $n\in \Q$ and $n\gg0$.  
    \item \textbf{Twisted Jacobi identity}: For any $a\in V^{r}$ and $v \in M^1$,  we have
    \begin{equation}\label{eq:twistedJac}
    \begin{aligned}
      \MoveEqLeft
      z_1^{-1}\del[\dfrac{z_2-w}{z_1}]
        Y_{M^3}(a,z_2)I(v,w)
        -z_1^{-1}\del[\dfrac{w-z_2}{-z_1}]
        I(v,w)Y_{M^2}(a,z_2)\\
      &=z_2^{-1}\pfrac{w+z_1}{z_2}^{r/T}
        \del[\dfrac{w+z_1}{z_2}]
        I(Y_{M^1}(a,z_1)v,w).
    \end{aligned} 
    \end{equation}
    \item \textbf{$\vo{L}{-1}$-derivative property}: For any $v\in M^1$, 
    \begin{equation*}
     I(\vo{L}{-1}v, w)=\odv{w}I(v, w).    
    \end{equation*}
  \end{itemize}
  Denote the space of $g$-twisted intertwining operators of type $\fusion$ by $\Fusion$ and set 
  \begin{equation*}
    \Nusion:=\dim\Fusion.
  \end{equation*}
  These numbers are called the \textbf{fusion rules} associated to the above data. 
\end{definition}

The following proposition is a straightforward consequence of the \emph{twisted Jacobi identity} and the \emph{$\vo{L}{-1}$-derivative property}. See \cite{FHL} and \cite{FZ} for more details.
\begin{proposition}\label{prop:I-degree}
  Suppose $M^1$, $M^2$, and $M^3$ are of conformal weights $h_1$, $h_2$, and $h_3$ respectively. 
  For any $v\in M^1$, we can write the intertwining operator $I(v,w)$ as
  \begin{equation*}
    w^{h_1+h_2-h_3}I(v,w)=
    \sum_{m\in \Z}
    \vo{v}{\frac{m}{T}}
    z^{-\frac{m}{T}-1}
    \in\Hom(M^2, M^3)\dbrack{z^{\pm\frac{1}{T}}},
  \end{equation*}
  where $\vo{v}{\frac{m}{T}}={v}_{h_1+h_2-h_3+\frac{m}{T}}$.
  Furthermore, $\vo{v}{\frac{m}{T}}M^2(\frac{n}{T})\subset M^3(\frac{n}{T}+\deg v-\frac{m}{T}-1)$ for any homogeneous $v\in M^1$ and any $m,n\in \N$. 
\end{proposition}

Multiplying the \emph{twisted Jacobi identity} \labelcref{eq:twistedJac} with $z^{m+\frac{r}{T}}_1z_2^{h+\frac{n}{T}}z_0^l$, where $m,n,l\in \Z$, then take $\Res_{z_0}\Res_{z_1}\Res_{z_2}$, we obtain its component form:
\begin{lemma} 
  Let $a\in V^r$ and $v\in M^1$, we have
  \begin{equation}\label{eq:twistedJac-comp}
  \begin{aligned}
    \MoveEqLeft
    \sum_{i\ge 0} 
    \binom{l}{i}(-1)^i
    \vo{a}{\frac{r}{T}+m+l-i}\vo{v}{\frac{n}{T}+i}
    -\sum_{i\ge 0}
    \binom{l}{i}(-1)^{l+i}
    \vo{v}{\frac{n}{T}+l-i}\vo{a}{m+\frac{r}{T}+i}\\
    &=\sum_{j\ge 0}
    \binom{m+\frac{r}{T}}{j}
    \vo{(\vo{a}{j+l}v)}{m-j+\frac{r+n}{T}},
  \end{aligned}
  \end{equation}
  for any $m,n,l\in \Z$. 
\end{lemma}

\subsection{Functions on the twisted projective line}\label{sec:TheCurve}
Now we introduce the algebraic curve $\bar{C}$.
For the general theory of algebraic curves, we refer to various algebraic geometry textbooks such as \cite[Ch.~7]{liu2002algebraic} and \cite[Ch.~IV]{Hartshorne}, or consult \cite{fulton} for a traditional treatment and \cite{Schlag} for an analytic approach to Riemann surfaces.

Throughout this paper and its subsequences, we adopt the following conventions on an integral scheme $(\X,\O_{\X})$ that is proper over $\C$:
\begin{itemize}[wide]
  \item We use the same notation for the analytification of $\X$. According to \emph{GAGA}, the categories of coherent \emph{algebraic} sheaves on $\X$ and coherent \emph{analytic} sheaves on $\X$ are equivalent, thus the terminology of \concept{$\O_{\X}$-module} is unambiguous.
  A special case provides the equivalence between \concept{rational functions} and \concept{meromorphic functions} on $\X$, which allows us to use these terminologies interchangeably.
  We use $\cK_{\X}$ to denote the constant sheaf of the \concept{field of rational functions}. 
  \item The \concept{de Rham complex} of $\X$ is denoted as $(\Omega^{\bullet}_{\X},\d)$. When the $1$-forms $\d{x_1},\cdots,\d{x_n}$ form a basis of the first de Rham cohomology space $H^1(\X,\Omega^{\bullet}_\X)$, we will use $\pdv{x_1},\cdots,\pdv{x_n}$ to denote its dual basis in the module of derivatives.
  \item Unless otherwise specified, a \concept{point} of $\X$, we mean a closed point of $\X$. We use Fraktur letters, such as $\pp$ and $\qq$, to denote such points. By abuse of notation, we do not distinguish a \emph{skyscraper sheaf} supported at a point $\pp$ with its stalk at $\pp$. 
  For a point $\pp$, we use $\cI_\pp$ to denote its ideal sheaf and $\kappa_{\pp}$ the residue field $\O_{\X}/\cI_\pp$.
  \item Following \cite{FBZ04}, we use $\O_{\pp}$ to denote the \concept{complete local ring at $\pp$}. That is the \emph{$\cI_\pp$-adic completion} of $\O_{\X}$, namely the limit $\varprojlim\O_{\X}/\cI_\pp^n$. 
  It is also the completion of the stalk of $\O_\X$ at $\pp$. 
\end{itemize}
Now, we assume that $\X$ is a curve, i.e., $\dim \X=1$. 
\begin{itemize}[wide]
  \item Each complete local ring $\O_{\pp}$ is a DVR. We use $v_{\pp}$ to denote the normalized valuation (i.e. $v_{\pp}(\O_{\pp}\setminus\Set{0})=\N$). 
  This valuation extends to the function field $\cK_\X$.
  \item A \concept{divisor} on $\X$ is a linear combination of points of $\X$. 
  The \concept{support} $\supp\Delta$ of a divisor $\Delta$ is the set of points involved (i.e. has nonzero coefficient) in $\Delta$. 
  Any rational function $f$ on $\X$ defines a divisor $(f):=\sum_{\pp\in\X}v_{\pp}(f)\pp$.
  Given a divisor $\Delta$, we use $\O(\infty\Delta)$ to denote the sheaf of meromorphic functions with possible poles along $\Delta$.
  \item By the \emph{Cohen structure theorem}, $\O_{\pp}\cong\C\dbrack{t}$ for some topological generator $t$ of $\O_\X$. 
  Such an element $t$ is called a \concept{local coordinate at $\pp$}. 
  The choice of $\pp$ and $t$ provides a morphism $\iota_t\colon\O_{\X}\monomorphism\O_{\pp}\cong\C\dbrack{t}$, images under which are called \concept{formal expansions}. 
  The pair $(\pp,t)$ is thus called a \concept{local chart}. 
\end{itemize}

\subsubsection*{Formal expansions of rational functions}
First, we recall the general formal expansions. 

\begin{lemma}\label{lem:formalEXP}
  Let $U$ be an open neighborhood of a point $\pp$ of $\X$. 
  Then any rational function $f$ on $U$ admits an expansion 
  \[
    f = \sum_{n=v_{\pp}(f)}^{\infty}a_nt^n,
  \]  
  where $a_n\in\C$ and $t$ is a fixed local coordinate at $\pp$. 
  This equality is understood in the sense that the series on the right-hand side converges to $f$ under the $\cI_\pp$-adic topology. 
\end{lemma}
\begin{proof}
  When $f$ is regular on $U$, this is just a concrete way to spell out the embedding $\iota_{t}$, where each $\sum_{n=0}^{m}a_nt^n$ serves as a representative of the class $f+\cI_{\pp}^{m+1}$. 
  The rational case follows by taking the fractional sections on both sides of the canonical embedding.
\end{proof}

To connect this lemma with its analytic counterpart, we need the following notions.
\begin{definition}
  For a point $\pp$ of $\X$, a \concept{(germ of) $1$-cycle around $\pp$} is an element in the \emph{costalk} of the \emph{cosheaf} of punctured singular homology $U\mapsto H_1(U-\pp,\Z)$, which can be presented as a $1$-cycle on a sufficiently small punctured neighborhood of $\pp$.
  Such a $1$-cycle is called \concept{simple} if its winding number is $1$.
\end{definition}
\begin{definition}
  For any meromorphic $1$-form $\alpha$ on $\X$, its \concept{residue} $\Res_{\pp}\alpha$ at a point $\pp\in \X$ is the value of the integration $\frac{1}{2\pi\iu}\int_{\gamma}\alpha$, where $\gamma$ is a simple $1$-cycle around $\pp$. 
\end{definition}

By this definition, given any rational map $f\colon \X\to\PP^1$ and any meromorphic $1$-form $\alpha$ on $\PP^1$, we have
\begin{equation}\label{eq:InvRes}
  \Res_{\pp}f^{\ast}\alpha = 
  \fun{mult}_{f}(\pp)\cdot\Res_{f(\pp)}\alpha.
\end{equation}
Here $\fun{mult}_{f}(\pp)$ denotes the \concept{multiplicity} of $f$ at $\pp$, which is the $\kappa_{f(\pp)}$-dimension of the fiber of the direct image $f_{\ast}\O_{\X}$ at $f(\pp)$.

\begin{lemma}\label{lem:formalEXPcoef}
  The coefficient $a_n$ in \cref{lem:formalEXP} can be computed by the formula 
  \[
    a_n = \Res_{\pp}t^{-n-1}f\d{t}.
  \]
\end{lemma}
\begin{proof}
  Direct computation shows 
  \[
    \int_{\gamma}t^n\d{t}=
    \begin{cases*}
      2\pi\iu & if $n=-1$,\\
      0 & otherwise,
    \end{cases*}
  \]
  where $\gamma$ is a simple $1$-cycle around $\pp$. 
  Then, the statement follows.
\end{proof}
\begin{remark}
  Let $t$ be a local coordinate at $\pp$. The lemma shows the following: for any meromorphic function $f$, we have $\Res_{\pp}f\d{t} = \Res_t(\iota_{t}f)$.
\end{remark}
As a corollary of this lemma, the series in \cref{lem:formalEXP} converges absolutely on $U-\pp$ and defines a meromorphic function which can be expressed as the rational function $f$.

The following is a special case of \cite[Remark 9.2.10]{FBZ04}, originally credited to \cite{Tate}.
\begin{lemma}[Strong Residue Theorem]\label{lem:SRT}
  Let $\ptseq$ be distinct points. 
  Then a collection of formal series $\Set*{f_i\in\O_{\pp_i}}_{i=1,\cdots,n}$ has the property that 
  \[
    \sum_{i=1}^{n}\Res_{\pp_i}f\alpha = 0,\txforall\alpha\in\Gamma(\X-\Set{\ptseq},\Omega^1),
  \]
  if and only if $f_i$ can be extended to the \emph{same} regular function on $\X-\Set{\ptseq}$.
\end{lemma}

\subsubsection*{The twisted projective line}
The algebraic curve we are concerned with is a \emph{$T$-twisted} version of the projective line. Abstractly, it is an \emph{stacky curve} obtained from $\PP^1$ modulo an action of a cyclic group of order $T$. 
We represent it as a smooth projective curve $\oC$ equipped with a ramified covering $\x\colon\oC\to\PP^1$, whose \emph{Galois group} is cyclic of order $T$. 

To make our expressions more explicit, we give the following \emph{ad-hoc} construction.
First, let $C$ be the smooth curve over $\C$ defined by the polynomial 
\[Y^T-X.\] 
For a point $\pp\in C$, we use $(X(\pp),Y(\pp))$ to denote its (global) coordinates in $\C^2$. 
We refer to the point with coordinates $(X,Y)=(0,0)$ as $0$.
Then, we have an isomorphism 
\begin{equation}\label{eq:isoC}
  \Gamma(C,\O_C)=\C[X,Y]/(Y^T-X)\cong\C[z^{\frac{1}{T}}]\colon X\mapsto z, Y\mapsto z^{\frac{1}{T}}.
\end{equation}
With the identification \labelcref{eq:isoC}, we have the following correspondences:
\begin{itemize}
  \item regular functions on $C$ $\longleftrightarrow$ $\C[z^{\frac{1}{T}}]$; and 
  \item regular functions on $C-\Set{0}$ $\longleftrightarrow$ $\C[z^{\pm\frac{1}{T}}]$.
\end{itemize}
Next, we introduce $C'$ as another copy of $C$, with coordinates written as $X'$ and $Y'$. 
We refer to the point with $(X',Y')=(0,0)$ as $\infty$. 
Then, we can identify $C-\Set{0}$ with $C'-\Set{\infty}$ through the isomorphism provided by $X^{-1}=X'$ and $Y^{-1}=Y'$. 
We call this domain $\Cc$. 
Gluing $C$ and $C'$ along $\Cc$ results a compactification of $C$, denoted by $\oC$. 
Then, \labelcref{eq:isoC} extends to an isomorphism $\oC\cong\Proj\C[z^{\frac{1}{T}}]$ and gives the following correspondence:
\begin{itemize}
  \item rational functions on $\oC$ $\longleftrightarrow$ $\C(z^{\frac{1}{T}})$.
\end{itemize}
Finally, the coordinate $X$ extends to a rational function $\x\colon\oC\to\PP^1$ provides the desired \emph{$T$-fold ramified covering} of the projective line $\PP^1$, with \emph{branch points} $0,\infty$ and \emph{unramified locus} $\Pc:=\PP^1\setminus\Set{0,\infty}$.  
It is straightforward to verify that its group of \emph{Deck transformations} is cyclic of order $T$. 

The inverse image $\x^{\ast}\O_{\PP^1}$ can be interpreted as the subsheaf of $\O_{\oC}$ consisting of regular functions factoring through $\x$. 
Then, the following correspondence is straightforward:
\begin{itemize}
  \item rational functions on $\oC$ that factor through $\x\colon\oC\to\PP^1$ $\longleftrightarrow$ $\C(z)\subset\C(z^{\frac{1}{T}})$.
\end{itemize}
On the other hand, the direct image $\x_{\ast}\O_{\oC}$ can be interpreted as an extension of $\O_{\PP^1}$.
Spelling out the stalks of $\x_{\ast}\O_{\oC}$, we see that its sections are multi-valued functions on $\PP^1$.

Our construction leads to a natural choice of local coordinates at $0$, $\infty$, and $\pp\in\Cc$. Refer to \cref{tab:EXP} for these coordinates, along with the established formal expansions provided by \cref{lem:formalEXP,lem:formalEXPcoef}.
\begin{table}[!h]
  \centering
    \renewcommand{\arraystretch}{2.5}
    \begin{tabular}{c|c|c}
     Local chart & Rational expansion & Coefficients in the expansion \\
     \hline
     $(0,Y)$ & 
     $\displaystyle f = \sum_{n=v_{0}(f)}^{\infty}a_{\frac{n}{T}}Y^n$ & 
     $a_{\frac{n}{T}}=\frac{1}{T}\Res_{0}Y^{-n}X^{-1}f\d{X}$ \\
     \hline
     $(\infty,Y^{-1})$ & 
     $\displaystyle f = \sum_{n=-\infty}^{-v_{\infty}(f)}a_{\frac{n}{T}}Y^n$ & 
     $a_{\frac{n}{T}}=-\frac{1}{T}\Res_{\infty}Y^{-n}X^{-1}f\d{X}$ \\
     \hline
     $(\pp,X-X(\pp))$ & 
     $\displaystyle f = \sum_{n=v_{\pp}(f)}^{\infty}a_n(X-X(\pp))^n$ & 
     $a_n=\Res_{\pp}(X-X(\pp))^{-n-1}f\d{X}$ \\
     \hline
    \end{tabular}
    \caption{Rational expansions at $\pp\in\oC$.}\label{tab:EXP}
\end{table}

The following corollary of \cref{lem:SRT} plays a crucial role in the present work. 
\begin{lemma}[Residue Sum Formula]\label{lem:ResidueSum}
  For any meromorphic $1$-form $\alpha$ on $\PP^1$, we have
  \begin{equation}\label{eq:ResidueSum}
    \tfrac{1}{T}\Res_{\pp=0}\x^{\ast}\alpha + \tfrac{1}{T}\Res_{\pp=\infty}\x^{\ast}\alpha + 
    \sum_{\qq}\Res_{\pp=\qq}\x^{\ast}\alpha = 0,
  \end{equation}
  where $\qq$ ranges over a branch of $\oC$. That is to say, for each singularity $\bar\qq\in\PP^1$ of $\alpha$, we only take one representative $\qq$ from $\x^{-1}\bar\qq$.
\end{lemma}
Note that the summation is finite since $\alpha$ only has finitely many poles.
\begin{proof}
The \emph{residue sum formula} on $\PP^1$ indicates that the sum of residues of a meromorphic $1$-form $\alpha$ on $\PP^1$ is zero. 
  By \labelcref{eq:InvRes}, we have 
  \[
    0= \sum_{\bar\qq\in\PP^1}\Res_{\bar\qq}\alpha = \sum_{\qq}\fun{mult}_{\x}(\qq)^{-1}\Res_{\qq}\x^{\ast}\alpha,
  \]
  where $\qq$ ranges over a branch of $\oC$. 
  Then \labelcref{eq:ResidueSum} follows from the observation that $\fun{mult}_{\x}(0)=T$, $\fun{mult}_{\x}(\infty)=T$, and $\fun{mult}_{\x}(\qq)=1$ for $\qq\in\Cc$.
\end{proof}

In the rest of this paper, we will frequently express a rational function on $\oC$ in terms of the local coordinate at the given point. 
For simplicity, we introduce the following shorthand notations.
\begin{table}[!h]
  \centering
    \renewcommand{\arraystretch}{1.5}
    \begin{tabular}{c|c|c}
     variable of points & coordinate $X$ & coordinate $Y$ \\
     \hline
     $\pp$ & 
     $z$ & 
     $z^{\frac{1}{T}}$ \\
     \hline
     $\pp_i$ & 
     $z_i$ & 
     $z_i^{\frac{1}{T}}$ \\
     \hline
     $\qq$ & 
     $w$ & 
     $w^{\frac{1}{T}}$ \\
     \hline
    \end{tabular}
    \caption{Shorthand notations for coordinates.}\label{tab:SHnotation}
\end{table}

\subsubsection*{Expansions of two-variable functions}
Let $f(\pp,\qq)$ be a meromorphic function on $\oC\times\oC$ with possible poles at $0$, $\infty$, and the divisor $(z-w)$. 
We can write $f(\pp,\qq)$ in the form
\begin{equation*}
  f(\pp,\qq) = 
  \frac{g(z^{\frac{1}{T}},w^{\frac{1}{T}})}{z^{\frac{m}{T}}w^{\frac{n}{T}}(z-w)^l},
\end{equation*}
where $g(z^{\frac{1}{T}},w^{\frac{1}{T}})\in\C[z^{\frac{1}{T}},w^{\frac{1}{T}}]$. 
Fixing $\qq$ and varying $\pp$, we can assign the following three expansions to $f$:
\begin{itemize}[wide]
  \item $\iota_{\pp=\infty}f$, the formal expansion at the point $\infty$. This corresponds to the \emph{$\iota_{z,w}$-expansion} in formal calculus, reflecting that the series converges in the domain $\abs{z}>\abs{w}$. 
  By \cref{tab:EXP}, 
  \begin{equation}\label{eq:iota_zw}
    \iota_{\pp=\infty}f = \frac{g(z^{\frac{1}{T}},w^{\frac{1}{T}})}{z^{\frac{m}{T}}w^{\frac{n}{T}}}\sum_{i\ge 0}\binom{l+i-1}{i}z^{-l-i}w^i \in\C\dparen{z^{-\frac{1}{T}}}\dparen{w^{\frac{1}{T}}}.
  \end{equation}
  In particular, $\Res_z(\iota_{\pp=\infty}f) = - \frac{1}{T}\Res_{\pp=\infty}f\d z \in\C[w^{\pm\frac{1}{T}}]$.
  \item $\iota_{\pp=0}f$, the formal expansion at the point $0$. This corresponds to the \emph{$\iota_{w,z}$-expansion} in formal calculus, reflecting that the series converges in the domain $\abs{w}>\abs{z}$. 
  By \cref{tab:EXP}, 
  \begin{equation}\label{eq:iota_wz}
   \iota_{\pp=0}f = \frac{g(z^{\frac{1}{T}},w^{\frac{1}{T}})}{z^{\frac{m}{T}}w^{\frac{n}{T}}}\sum_{i\ge 0}\binom{-l}{i}(-w)^{-l-i}z^i \in\C\dparen{w^{-\frac{1}{T}}}\dparen{z^{\frac{1}{T}}}.
  \end{equation}
  In particular, $\Res_z(\iota_{\pp=0}f) = \frac{1}{T}\Res_{\pp=0}f\d z \in\C[w^{\pm\frac{1}{T}}]$.
  \item $\iota_{\pp=\qq}f$, the formal expansion at the point $\qq$. This corresponds to the \emph{$\iota_{w,z-w}$-expansion} in formal calculus, reflecting that the series converges in the domain $\abs{w}>\abs{z-w}$ with the argument range\footnote{Here $\Arg(a:b)\in\linterval{-\pi}{\pi}$ denotes the argument of the ray $[a:b]\in\PP^1$. This condition guarantees that $\pp$ and $\qq$ are in the same branch of $\oC$.} $\abs{\Arg(z^{\frac{1}{T}}:w^{\frac{1}{T}})}<\frac{\pi}{2T}$.  
  By \cref{tab:EXP}, 
  \begin{equation}\label{eq:iota_wz-w}
    \iota_{\pp=\qq}z^{\frac{m}{T}} = \sum_{i\ge 0}\binom{\frac{m}{T}}{i}w^{\frac{m}{T}-i}(z-w)^i\in\C\dparen{w^{-\frac{1}{T}}}\dparen{z-w}.
  \end{equation}
  In particular, $\Res_{z-w}(\iota_{\pp=\qq}f) = \Res_{\pp=\qq}f\d z \in\C[w^{\pm\frac{1}{T}}]$.
\end{itemize}

By \cref{prop:I-degree} and \labelcref{eq:twistedJac}, we have the following \emph{duality property}.
\begin{proposition}\label{prop:Jac-fun-tw}
  Let $M^1$ be an admissible untwisted module, and let $M^2$ and $M^3$ be admissible $g$-twisted modules. 
  Suppose $M^1$, $M^2$, and $M^3$ are of conformal weights $h_1$, $h_2$, and $h_3$ respectively. 
  For any $a\in V^r$, $v\in M^1$, $v_2\in M^2$, and $v'_3\in (M^3)'$, there exists a meromorphic function $f$ on $\oC\times\oC$ of the form (where $m,n,l\in\N$)
  \begin{equation}\label{eq:2ptFun}
    f(\pp,\qq) = 
    \frac{g(z,w^{\frac{1}{T}})}{z^{\frac{r}{T}}z^{m}w^{\frac{n}{T}}(z-w)^l}\qquad(g(z,w^{\frac{1}{T}})\in\C[z,w^{\frac{1}{T}}]),
  \end{equation}
  such that the following identities of formal Puiseux series hold:
  \begin{align*}
    &\braket*{v'_3}{Y_{M^3}(a,z)I(v,w)v_2} w^h 
    = \iota_{\pp=\infty}f,
    \\
    &\braket*{v'_3}{I(v,w)Y_{M^2}(a,z)v_2} w^h 
    = \iota_{\pp=0}f,
    \\
    &\braket*{v'_3}{I(Y_{M^1}(a,z-w)v,w)v_2} w^h 
    = \iota_{\pp=\qq}f.
  \end{align*}
\end{proposition}
\begin{proof}
  With the formulas \labelcref{eq:iota_zw,eq:iota_wz,eq:iota_wz-w}, 
  the statement follows from the the \emph{twisted Jacobi identity} \labelcref{eq:twistedJac} by the \emph{\nameref{lem:SRT}}.
\end{proof}

Conversely, we have 
\begin{proposition}
  For any meromorphic function on $\oC\times\oC$ of the form \labelcref{eq:2ptFun}, we have
  \begin{equation*}
    \Res_z\left(\iota_{\pp=\infty}z^{\frac{r}{T}}f\right)-\Res_z \left(\iota_{\pp=0}z^{\frac{r}{T}}f\right)=\Res_{z-w}\left(\iota_{\pp=\qq}z^{\frac{r}{T}}f\right).
  \end{equation*} 
  In particular, the twisted Jacobi identity \labelcref{eq:twistedJac} of the intertwining operator $I$ among an admissible untwisted module $M^1$ and admissible $g$-twisted modules $M^2$ and $M^3$ is equivalent to the existence of a meromorphic function $f(\pp,\qq)$ satisfying \cref{prop:Jac-fun-tw}.
\end{proposition}
\begin{proof}
  Fixing $\qq\in\oC$, the meromorphic $1$-form $z^{\frac{r}{T}}f\d{z}$ factors through the covering map $x\colon\oC\to\PP^1$. Hence, the \emph{\nameref{lem:ResidueSum}} applies and gives us
  \[
    \tfrac{1}{T}\Res_{\pp=0}z^{\frac{r}{T}}f\d z + 
    \tfrac{1}{T}\Res_{\pp=\infty}z^{\frac{r}{T}}f\d z + 
    \Res_{\pp=\qq}z^{\frac{r}{T}}f\d z = 0.
  \]
  Applying formulas \labelcref{eq:iota_zw,eq:iota_wz,eq:iota_wz-w} to the above, we obtain the desired identity.
\end{proof}

\subsection{Space of \texorpdfstring{$g$}{g}-twisted correlation functions}
In this subsection, we define the space of \emph{$g$-twisted correlation functions} for the \emph{$g$-twisted conformal blocks}.

Let $M^1$ be an admissible untwisted module, and let $M^2$ and $M^3$ be admissible $g$-twisted modules. Suppose $M^1$, $M^2$, and $M^3$ are of conformal weights $h_1$, $h_2$, and $h_3$, respectively. Suppose $I$ is a \emph{$g$-twisted intertwining operator} in $\Fusion$. 
For $v\in M^1$, $v_2\in M^2$, $v_3'\in (M^3)'$, and $\ainVseq$, consider the following $(n+1)$-variable formal Puiseux series: 
\begin{equation}\label{eq:npt-Fun}
  \braket*{v_3'}{
    Y_{M^3}(a^{1},z_{1})\cdots Y_{M^3}(a^{k-1},z_{k-1})
    I(v,w)
    Y_{M^2}(a^{k},z_{k})\cdots Y_{M^2}(a^{n},z_{n})v_2
    }w^{h},
\end{equation}  
Using a similar method as the proof of \cite[Proposition 3.5.1]{FHL}, by \cref{prop:Jac-fun,prop:Jac-fun-tw}, applying the \emph{\nameref{lem:SRT}}, we see that there is a rational function on $\oC^{n+1}$ of the form:
\begin{equation}\label{eq:F(rseq)}
  f(\ptseq,\qq)=
  \Dfrac{g(\zseq,w^{\frac{1}{T}})}
  {
    w^{\frac{n}{T}}
    \prod_{i=1}^n z_i^{\frac{r_i}{T}}z_i^{m_i}
    \prod_{j<k}(z_j-z_k)^{l_{jk}}
    \prod_{p=1}^n (z_p-w)^{l_p}},
\end{equation}
where $m_i,n,l_{kj},l_p\in\N$, and $g(\zseq,w^{\frac{1}{T}})\in \C[\zseq,w^{\frac{1}{T}}]$ such that the series \labelcref{eq:npt-Fun} is the formal expansion of $f(\ptseq,\qq)$ in the domain
\[ 
  \Set*{(\ptseq,\qq)\in\oC^{n+1}\given\infty>|z_1|>\cdots>|z_{k-1}|>|w|>|z_k|>\cdots>|z_n|>0}.
\]
We denote the space of functions of the form \labelcref{eq:F(rseq)} by $\cF(\rseq)$. 

Let $\Delta_n$ denote the \emph{divisor} 
\[
  \Delta_n:=
  (w\prod_{i=1}^nz_i
  \prod_{j<k}(z_j-z_k)
  \prod_{p=1}^n (z_p-w)).
\]
Then the space $\O(\infty\Delta_n)$ of meromorphic functions on $\oC^{n+1}$ with possible poles along the divisor $\Delta_n$, namely at the points where either $z_i=0$, $z_i=\infty$, $w=0$, $w=\infty$, $z_j=z_k$, or $z_p=w$, admits a decomposition
\[
  \O(\infty\Delta_n) = 
  \bigoplus_{0\le \rseq \le T-1}\cF(\rseq),
\]
Note that $\cF(\emptyset)=\O(\infty\Delta_0)=\Gamma(\Cc,\O_{\oC})=\C[w^{\pm\frac{1}{T}}]$.

Following \cite{Liu,Z} with slight modifications, we use the following notation to denote the function $f(\pp_1,\cdots,\pp_n,\qq)$ in \labelcref{eq:F(rseq)}:
\begin{equation*}
  S_I\bracket*{v_3'}{\cfbseq[1][k-1](v,\qq)\cfbseq[k][n]}{v_2}.
\end{equation*}
Then we obtain a system of linear maps $S_I=\Set*{(S_I)^n_{V\cdots M^1\cdots V}}_{n\in\N}$, where 
\begin{equation}\label{eq:SI}
  \begin{aligned}
    (S_I)_{V\cdots M^1\cdots V}^n\colon& (M^3)'\otimes V\otimes\cdots\otimes V\otimes M^{1}\otimes V\otimes \cdots V\otimes M^{2}\longrightarrow\O(\infty\Delta_n),\\
    & v_3'\otimes a^{1}\otimes\cdots\otimes a^{k-1}\otimes v\otimes a^{k}\otimes\cdots\otimes a^{n}\otimes v_2\\
    & \longmapsto
    S_I\bracket*{v_3'}{\cfbseq[1][k-1](v,\qq)\cfbseq[k][n]}{v_2}.
  \end{aligned}
\end{equation} 
By \cref{prop:Jac-fun,prop:Jac-fun-tw}, we have: 
\begin{itemize}
  \item each map $(S_I)_{V\cdots M^1\cdots V}^n$ factors through $(M^3)'\otimes \fun{Sym}[V,\cdots,V,M^1] \otimes M^{2}$; and 
  \item $(S_I)^n_{M^1V\cdots V}=(S_I)^n_{VM^1\cdots V}=\cdots =(S_I)^n_{V\cdots VM^1}$.
\end{itemize}
Hence we can always put the terms in the order $\cfbseq(v,\qq)$ and omit these terms unless we want to emphasize some of them.

Note that for homogeneous $\ainVseq$, the function 
\[
  S_I\bracket*{v_3'}{\cdots}{v_2}
  z_1^{\frac{r_1}{T}}\cdots z_n^{\frac{r_n}{T}}
\] 
factors through $(x,\cdots,x,\id)\colon\oC^{n+1}\to(\PP^1)^n\times\oC$, and thus can be viewed as a meromorphic function on $(\PP^1)^n\times\oC$. 
The following definition generalizes \cite[Definition 2.1]{Liu}:
\begin{definition}\label{def:genus-zero}
  Let $V$ be a VOA with an order $T$ automorphism $g$, and let $M^1$ (resp. $M^2$ and $N^3$) be an admissible \emph{untwisted} (resp. \emph{$g$-twisted} and \emph{$g^{-1}$-twisted}) module of conformal weight $h_1$ (resp. $h_2$ and $h_3$). 
  Put $h=h_1+h_2-h_3$. 
  A system of linear maps 
  $S=\Set*{S^n_{V\cdots M^1\cdots V}}_{n\in\N}$, 
  \begin{align*}
    S_{V\cdots M^1\cdots V}^n\colon& N^3\otimes V\otimes\cdots\otimes V\otimes M^{1}\otimes V\otimes \cdots V\otimes M^{2}\longrightarrow\O(\infty\Delta_n),\\
    & v_3\otimes a^{1}\otimes\cdots\otimes a^{k-1}\otimes v\otimes a^{k}\otimes\cdots\otimes a^{n}\otimes v_2\\
    & \longmapsto
    S\bracket*{v_3}{\cfbseq[1][k-1](v,\qq)\cfbseq[k][n]}{v_2},
  \end{align*}
  is said to satisfy the \concept{twisted genus-zero property associated to the datum} 
  \begin{equation}\label{eq:def:Sigma1}
      \Sigma_{1}(N^3, M^1, M^2):=(\x\colon\oC\to\PP^1, \infty, 1, 0, N^3, M^1, M^2)
  \end{equation}
  if it satisfies the following axioms for all $v_3\in N^3,v_2\in M^{2}$:
\begin{enumerate}
  \item \textbf{Truncation property}: For any fixed $v\in M^{1}$ and $v_2\in M^{2}$, there exists $N\in \N$ depending only on $v,v_2$, such that 
  $S\bracket*{v_3}{(v,\qq)}{v_2}w^{\frac{N}{T}}\in\C[w^{\frac{1}{T}}]$ for all $v_3\in N^3$. 
  \item \textbf{Locality}: The terms $\cfbseq(v,\qq)$ can be arbitrarily permuted:
  \begin{itemize}
    \item each map $S_{V\cdots M^1\cdots V}^n$ factors through $N^3\otimes \fun{Sym}[V,\cdots,V,M^1] \otimes M^{2}$; and 
    \item $S^n_{M^1V\cdots V}=S^n_{VM^1\cdots V}=\cdots =S^n_{V\cdots VM^1}$.
  \end{itemize}
  Hence we can always put the terms in the order $\cfbseq(v,\qq)$ and omit these terms unless we want to emphasize some of them.
  \item \textbf{Homogeneous property}: For homogeneous $\ainVseq$, we have 
  \begin{equation}\label{eq:homogeneous}
    S\bracket*{v_3}{\cdots}{v_2}\in\cF(\rseq).  
  \end{equation}
  \item \textbf{Vacuum property}:
  \begin{equation}\label{eq:vacuum}
    S\bracket*{v_3}{(\vac,\pp)\cdots}{v_2} =
    S\bracket*{v_3}{\cdots}{v_2}.
  \end{equation}
  \item \textbf{$\vo{L}{-1}$-derivative property}: 
  \begin{align}
    \pdv{z_{1}}S\bracket*{v_3}{(a^{1},\pp_1)\cdots}{v_2} &
     = S\bracket*{v_3}{(\vo{L}{-1}a^{1},\pp_{1})\cdots}{v_2}, \label{eq:L-derivative-z}\\
    \pdv{w}\left(S\bracket*{v_3}{\cdots(v,\qq)}{v_2}w^{-h}\right) &
     = S\bracket*{v_3}{\cdots(\vo{L}{-1}v,\qq)}{v_2}w^{-h}. \label{eq:L-derivative-w}
  \end{align}
  \item \textbf{Associativity}:
  \begin{align*}
    \Res_{\pp_1=\qq}S\bracket*{v_3}{(a^{1},\pp_1)\cdots(v,\qq)}{v_2}(z_{1}-w)^{k}\d{z_1}
    &= S\bracket*{v_3}{\cdots(\vo{a^{1}}{k}v,\qq)}{v_2}, \numberthis\label{eq:associativity_p1q} \\
    \Res_{\pp_1=\pp_2}S\bracket*{v_3}{(a^{1},\pp_1)(a^{2},\pp_2)\cdots}{v_2}(z_{1}-z_{2})^{k}\d{z_1}
    &= S\bracket*{v_3}{(\vo{a^{1}}{k}a^{2},\pp_{2})\cdots}{v_2}. \numberthis\label{eq:associativity_p1p2}
  \end{align*}
  \item \textbf{Generating property for $M^2$}:
  For any $a\in V^r$ and $m\in \Z$, we have:
  \begin{equation}\label{eq:generating_M2}
    \Res_{\pp=0}S\bracket*{v_3}{(a,\pp)\cdots}{v_2}z^{m+\frac{r}{T}}\d{z} = 
    TS\bracket*{v_3}{\cdots}{\vo{a}{m+\frac{r}{T}}v_2}.
  \end{equation}
  \item \textbf{Generating property for $N^3$}:
  For any $a\in V^r$ and $m\in \Z$, we have
  \begin{equation}\label{eq:generating_M3}
    \Res_{\pp=\infty}S\bracket*{v_3}{(a,\pp)\cdots}{v_2}z^{m+\frac{r}{T}}\d{z} =
    -TS\bracket*{\theta(\lo{a}{m+\frac{r}{T}})v_3}{\cdots}{v_2},
  \end{equation}
  where $\theta\colon\L_g(V)\to\L_{g^{-1}}(V)$ is the anti-isomorphism defined in \labelcref{eq:def:theta} and $\L_{g^{-1}}(V)$ acts on the admissible $g^{-1}$-twisted module $N^3$ via \cref{prop:repnOfLgV}.
\end{enumerate}
  The vector space consists of systems of linear maps $S=\Set*{S^n_{V\cdots M^1\cdots V}}_{n\in\N}$ satisfying the above axioms is called the \concept{space of $g$-twisted correlation functions associated to the datum $\Sigma_{1}(N^3, M^1, M^2)$}, we denote it by $\Cor[\Sigma_{1}(N^3, M^1, M^2)]$. 
  When $N^3$ is the contragridient module of an admissible $g$-twisted module $M^3$, we call this space \concept{the space of $g$-twisted correlation functions of type $\fusion$} and denote it by $\Cor\fusion$. 
\end{definition}
\begin{remark}
  Note that we do not initially require $N^3$ to be the contragridient module of an admissible $g$-twisted module $M^3$ in the definition above.  
\end{remark}

\begin{proposition}
  The system $S_I$ given by \labelcref{eq:SI} belongs to $\Cor\fusion$.
\end{proposition}
\begin{proof}
  We have already proven the \emph{locality} (2) and the \emph{homogeneous property} (3). 
  The properties (1), (4), and (5) follow from the \emph{truncation property}, the \emph{vacuum property}, and the \emph{$\vo{L}{-1}$-derivative property} of $I(\cdot,w)$ and $Y_{M^{i}}(\cdot,z)$.
  
  For the \emph{associativity} (6): by \cref{prop:Jac-fun,prop:Jac-fun-tw}, we have 
  \begin{align*}
    \MoveEqLeft
    \Res_{\pp_1=\qq}S_I\bracket*{v_3'}{(a^{1},\pp_1)\cdots(v,\qq)}{v_2}(z_{1}-w)^{k}\d{z_1} \\
    \overset{\iota_{\pp_1=\qq}}&{=} 
    \Res_{z_1-w}S_I\bracket*{v_3'}{\cdots(Y_{M^1}(a^{1},z_1)v,\qq)}{v_2}(z_{1}-w)^{k}
    =S_I\bracket*{v_3'}{\cdots(\vo{a^{1}}{k}v,\qq)}{v_2}, \\
    \MoveEqLeft
    \Res_{\pp_1=\pp_2}S_I\bracket*{v_3'}{(a^{1},\pp_1)(a^{2},\pp_2)\cdots}{v_2}(z_{1}-z_{2})^{k}\d{z_1} \\
    \overset{\iota_{\pp_1=\pp_2}}&{=} 
    \Res_{z_1-z_2}S_I\bracket*{v_3'}{(Y_{M^1}(a^{1},z_1)a^{2},\pp_2)\cdots}{v_2}(z_{1}-z_{2})^{k} 
    =S_I\bracket*{v_3'}{(\vo{a^{1}}{k}a^{2},\pp_{2})\cdots}{v_2}. 
  \end{align*}
  
  For the \emph{generating property} (7):
  \begin{align*}
    \MoveEqLeft
    \Res_{\pp=0}S_I\bracket*{v_3'}{(a,\pp)\cdots}{v_2}z^{m+\frac{r}{T}}\d{z} = 
    \Res_{\pp=0}S_I\bracket*{v_3'}{\cdots(a,\pp)}{v_2}z^{m+\frac{r}{T}}\d{z}\\
    &
    \overset{\iota_{\pp=0}}{=} 
    T\Res_{z}S_I\bracket*{v_3'}{\cdots}{Y_{M^2}(a,z)v_2}z^{m+\frac{r}{T}} = 
    TS_I\bracket*{v_3'}{\cdots}{\vo{a}{m+\frac{r}{T}}v_2}
  \end{align*}

  For the \emph{generating property} (8): using \labelcref{eq:contragredient'},
  \begin{align*}
    \MoveEqLeft
    \Res_{\pp=\infty}S_I\bracket*{v_3'}{(a,\pp)\cdots}{v_2}z^{m+\frac{r}{T}}\d{z} \\
    \overset{\iota_{\pp=\infty}}&{=} 
    -T\Res_{z}S_I\bracket*{Y_{(M^3)'}(e^{z\vo{L}{1}}(-z^{-2})^{\vo{L}{0}}a, z^{-1})v_3'}{\cdots}{v_2}z^{m+\frac{r}{T}}\d{z}\\
    \overset{\labelcref{eq:def:theta}}&{=}
    -TS_I\bracket*{\theta(\lo{a}{m+\frac{r}{T}})v_3'}{\cdots}{v_2}.
  \end{align*}
  Hence $S_I$ satisfies the twisted genus-zero property associated to $\fusion$.
\end{proof}

Now we have our first main theorem of this paper, which generalizes \cite[Theorem 2.5 and Corollary 2.6]{Liu}.
\begin{theorem}\label{thm:I=Cor}
  Let $M^1$ (resp. $M^2$ and $M^3$) be an admissible untwisted (resp. $g$-twisted) module of conformal weight $h_1$ (resp. $h_2$ and $h_3$). 
  Put $h=h_1+h_2-h_3$. 
  Then we have the following isomorphism of vector spaces:
  \begin{equation*}
    \Fusion\cong \Cor\fusion, \qquad I\longmapsto S_I.
  \end{equation*}
\end{theorem}
\begin{proof}
  Given any $S\in\Cor\fusion$, we define 
  \[
    I_S(\cdot,w)\colon M^1\to\Hom(M^2,M^3)\dbrack{w^{\pm\frac{1}{T}}}w^{-h},\quad 
    v\mapsto
    I_S(v,w)=\sum_{n\in\Z}\vo{v}{\frac{n}{T}}w^{-\frac{n}{T}-1-h},
  \]
  where $\vo{v}{\frac{n}{T}}$ is determined by
  \begin{equation}\label{eq:IfromS}
    \braket*{v_3'}{\vo{v}{\frac{n}{T}}v_2} = 
    \tfrac{1}{T}\Res_{\qq=0}S\bracket*{v_3'}{(v,\qq)}{v_2}w^{\frac{n}{T}}\d{w} \in\C[w^{\pm\frac{1}{T}}],
  \end{equation}
  where $v_3'\in(M^3)'$ and $v_2\in M^2$.
  Then the \emph{truncation property} and the \emph{$\vo{L}{-1}$-derivative property} of $I_S$ follow from the axioms (1) and (5).
  It remains to show the following \emph{twisted Jacobi identity} of the twisted intertwining operator $I_S(\cdot,w)$:
  \[
  \begin{multlined}
    \sum_{i\ge 0} 
    \binom{l}{i}(-1)^i
    \vo{a}{\frac{r}{T}+m+l-i}\vo{v}{\frac{n}{T}+i}v_2
    -\sum_{i\ge 0} 
    \binom{l}{i}(-1)^{l+i}
    \vo{v}{\frac{n}{T}+l-i}\vo{a}{m+\frac{r}{T}+i}v_2\\
    =\sum_{j\ge 0}
    \binom{m+\frac{r}{T}}{j}
    \vo{(\vo{a}{j+l}v)}{m-j+\frac{r+n}{T}}v_2.
  \end{multlined}
  \]
  Note that the involved summations are finite due to the \emph{truncation property}.
  
  Indeed, by the \emph{generating properties} \labelcref{eq:generating_M3}, \labelcref{eq:IfromS}, and \labelcref{eq:contragredient'-comp}, we have
  \begin{align*}
    \MoveEqLeft
    \braket*{v_3'}{\sum_{i\ge 0} 
    \binom{l}{i}(-1)^i
    \vo{a}{\frac{r}{T}+m+l-i}\vo{v}{\frac{n}{T}+i}v_2} \\
    &= 
    \sum_{i\ge 0} 
    \binom{l}{i}(-1)^i
    \braket*{\theta(\lo{a}{\frac{r}{T}+m+l-i})v_3'}{\vo{v}{\frac{n}{T}+i}v_2}  \\
    &=-
    \sum_{i\ge 0} 
    \binom{l}{i}(-1)^i
    \tfrac{1}{T^2}
    \Res_{\qq=0}\Res_{\pp=\infty}
    S\bracket*{v_3'}{(a,\pp)(v,\qq)}{v_2}z^{m+l-i-\frac{r}{T}}w^{\frac{n}{T}+i}\d{z}\d{w}\\
    &=-
    \tfrac{1}{T^2}
    \Res_{\qq=0}\Res_{\pp=\infty}
    S\bracket*{v_3'}{(a,\pp)(v,\qq)}{v_2}
    (z-w)^lz^{m+\frac{r}{T}}w^{\frac{n}{T}}\d{z}\d{w}.
  \end{align*}
  On the other hand, by \labelcref{eq:generating_M2} and \labelcref{eq:IfromS}, we have
  \begin{align*}
    \MoveEqLeft
    \braket*{v_3'}{\sum_{i\ge 0} 
    \binom{l}{i}(-1)^{l+i}
    \vo{v}{\frac{n}{T}+l-i}\vo{a}{m+\frac{r}{T}+i}v_2} \\
    &= 
    \sum_{i\ge 0} 
    \binom{l}{i}(-1)^{l+i}
    \tfrac{1}{T^2}
    \Res_{\qq=0}\Res_{\pp=0}
    S\bracket*{v_3'}{(a,\pp)(v,\qq)}{v_2}z^{m+\frac{r}{T}+i}w^{\frac{n}{T}+l-i}\d{z}\d{w}\\
    &=
    \tfrac{1}{T^2}
    \Res_{\qq=0}\Res_{\pp=0}
    S\bracket*{v_3'}{(a,\pp)(v,\qq)}{v_2}
    (z-w)^lz^{m+\frac{r}{T}}w^{\frac{n}{T}}\d{z}\d{w}.
  \end{align*}
  Therefore, we have 
  \begin{align*}
    \MoveEqLeft
    \braket*{v_3'}{\sum_{i\ge 0} 
    \binom{l}{i}(-1)^i
    \vo{a}{\frac{r}{T}+m+l-i}\vo{v}{\frac{n}{T}+i}v_2
    -\sum_{i\ge 0} 
    \binom{l}{i}(-1)^{l+i}
    \vo{v}{\frac{n}{T}+l-i}\vo{a}{m+\frac{r}{T}+i}
    v_2} \\
    &=
    \tfrac{1}{T^2}\Res_{\qq=0}(-\Res_{\pp=\infty}-\Res_{\pp=0})
    S\bracket*{v_3'}{(a,\pp)(v,\qq)}{v_2}
    (z-w)^lz^{m+\frac{r}{T}}w^{\frac{n}{T}}\d{z}\d{w}\\
    \overset{\ast}&{=}
    \tfrac{1}{T}
    \Res_{\qq=0}\Res_{\pp=\qq}
    S\bracket*{v_3'}{(a,\pp)(v,\qq)}{v_2}
    (z-w)^lz^{m+\frac{r}{T}}w^{\frac{n}{T}}\d{z}\d{w}\\
    \overset{**}&{=}
    \tfrac{1}{T}
    \Res_{\qq=0}\sum_{j\ge  0}\binom{m+\frac{r}{T}}{j}
    \Res_{\pp=\qq}S\bracket*{v_3'}{(a,\pp)(v,\qq)}{v_2}
    (z-w)^{l+j}w^{m-j+\frac{r+n}{T}}\d{z}\d{w}\\
    \overset{\labelcref{eq:associativity_p1q}}&{=}
    \sum_{j\ge 0}\binom{m+\frac{r}{T}}{j}
    \tfrac{1}{T}
    \Res_{\qq=0}S\bracket*{v_3'}{(\vo{a}{l+j}v,\qq)}{v_2}w^{m-j+\frac{r+n}{T}}\d{w}\\
    \overset{\labelcref{eq:IfromS}}&{=}
    \sum_{j\ge  0}\binom{m+\frac{r}{T}}{j}
    \braket{v_3'}{\vo{(\vo{a}{l+j}v,\qq)}{m-j+\frac{r+n}{T}}v_2},
  \end{align*}
  where $*$ follows from the \emph{residue sum formula} and $**$ follows from applying $\iota_{\pp=\qq}$ to $z^{m+\frac{r}{T}}$. 
  Indeed, we have $S\bracket*{v_3'}{(a,\pp)(v,\qq)}{v_2}\in\cF(r)$ by the \emph{homogeneous property} of $S$. 
  Hence the meromorphic $1$-form $S\bracket*{v_3'}{(a,\pp)(v,\qq)}{v_2}(z-w)^lz^{m+\frac{r}{T}}w^{\frac{n}{T}}\d{z}$ factors through $\x$ with possible poles $0$, $\infty$, and $w$ on $\PP^1$. 
  Then the \emph{\nameref{lem:ResidueSum}} applies. 
  For the expansion $\iota_{\pp=\qq}$, note that the summation involved is finite, 
so it commutes with the integrals. Thus the equality follows.
\end{proof}

\section{Reconstructing \texorpdfstring{$g$}{g}-twisted correlation functions from restricted correlation functions}\label{sec3}
In this section, we introduce the space of \emph{correlation functions} associated to the datum $(\x\colon\oC\to\PP^1, \infty, 1, 0, U^3, M^1, U^2)$ where $U^2$ (resp. $U^3$) is a left (resp. right) $A_g(V)$-module. 
In our application, $U^2$ and $U^3$ are the \emph{lowest-weight} subspaces of $M^2$ and $(M^3)'$ respectively. 
In general, they are only considered as irreducible modules over the \emph{$g$-twisted Zhu's algebra} $A_g(V)$. 

Recall the definition of $A_g(V)$ in \cite{DLM1}. It is the quotient of $V$ modulo the subspace $O_g(V)$, which is spanned by  
\[
  a\circ _g b:=\Res_z \frac{(1+z)^{\wt a-1+\delta(r)+\frac{r}{T}}}{z^{1+\delta(r)}} Y(a,z)b,\quad a\in V^r,b\in V,
\] 
where $\delta(r)$ is defined by
$
  \delta(r)=
  \begin{cases*}
    1 & if $r=0$,\\
    0 & otherwise.
  \end{cases*}
$

By \cite[Lemma 2.1]{DLM1}, $V^r\subseteq O_g(V)$ for $r\neq 0$.
Define 
\begin{equation}\label{def:g-star-product}
  a\ast _g b:=
  \begin{cases*}
    \Res_z Y(a,z)b\frac{(1+z)^{\wt a}}{z}& 
    if $a\in V^0$,\\
    0& otherwise.
  \end{cases*}
\end{equation}
Denote the image of $a\in V$ in $A_g(V)$ by $[a]$. We have the following result:
\begin{lemma}[\cite{Z,DLM1}]
  The operation $\ast_g$ induces an associative algebra structure on $A_g(V)$ with $[\vac]$ as the identity, and $[\upomega]$ lying in its center. 
\end{lemma}

Let $M=\bigoplus_{n\in \N} M(\frac{n}{T})$ be an admissible $g$-twisted module. By \labelcref{eq:admissible}, its bottom level $M(0)$ is preserved by the \emph{zero mode operators} $o(a):=\vo{a}{\wt a-1}$ ($a\in V$). 
Note that the assignment $o\colon a\mapsto o(a)$ vanishes outside $V^0$.
Furthermore, we have
\begin{lemma}[\cite{Z,DLM1}]
  The bottom level $M(0)$ is an $A_g(V)$-module with the action given by $[a]\cdot v=o(a)v$ for $a\in V$ and $v\in M(0)$. More specifically, we have 
  \begin{align}
    o(a\circ_g b)v&=0, \label{eq:oab}\\
    o(a)o(b)v&=o(a\ast b)v, \label{eq:oaob}\\
    o(a)o(b)v-o(b)o(a)v&=\sum_{j\ge 0} \binom{\wt a-1}{j} o(\vo{a}{j}b)v
    \qquad\text{for } a\in V^0.\label{eq:oaob-oboa}
  \end{align}
\end{lemma}
Given an $A_g(V)$-module $U$, the dual space $U^{\ast}$ is a right module over $A_g(V)$, 
where $[a]$ acts on $u'\in U^{\ast}$ on the right by $\braket*{u'\cdot [a]}{u}=\braket*{u'}{[a]\cdot u}$ for $u\in U$. 
When $U=M(0)$ for some admissible $g$-twisted module $M$, we have the following formula that is dual to \labelcref{eq:oaob-oboa}: 
\begin{equation*}
  v'o(a)o(b)-v'o(b)o(a)=\sum_{j\ge 0} \binom{\wt a-1}{j} v'o(\vo{a}{j}b),\quad a,b\in V^0, v\in U,\label{eq:oaob-oboa_'}
\end{equation*}

\subsection{Space of \texorpdfstring{$g$}{g}-twisted restricted correlation functions}
To define the auxiliary space $\Cor[\Sigma_{1}(U^3, M^1, U^2)]$ of $g$-twisted restricted correlation functions, we need the following two-variable functions on $\oC$
\begin{equation}\label{Fpq}
  F_{n,i}(\pp,\qq):=
  \frac{z^{-n}}{i!}\left(\pdv{w}\right)^{i}\frac{w^{n}}{z-w},
\end{equation}
for all $n\in\tfrac{1}{T}\Z$. 
The following lemmas are evident.
\begin{lemma}\label{lem:ExpOfFpq}
  The expansions of $F_{n,i}(\pp,\qq)$ at $\pp=0$, $\pp=\infty$, and $\pp=\qq$ are
  \begin{align*}
    \iota_{\pp=0}F_{n,i} &= -\sum_{j\ge 0}\binom{n-j-1}{i}z^{j-n}w^{n-j-i-1},\\
    \iota_{\pp=\infty}F_{n,i} &= \sum_{j\ge 0}\binom{n+j}{i}z^{-n-j-1}w^{n+j-i},\\
    \iota_{\pp=\qq}F_{n, i} &= \sum_{l=0}^i \sum_{p\ge 0} \binom{n}{i-l} \binom{-n}{p} w^{-i+l-p}(z-w)^{p-l-1}.
  \end{align*}
\end{lemma}
\begin{lemma}\label{lem:IndFpq}
  The functions $F_{n,i}(\pp,\qq)$ for successive $n$ have the following relation:
  \[
    F_{n,i}(\pp,\qq) - F_{n+1,i}(\pp,\qq) = \binom{n}{i}z^{-n-1}w^{n-i}.
  \]
\end{lemma}

We now give a definition that generalizes \cite[Definition 3.1]{Liu} to the $g$-twisted case. Note that there is an additional shifting in each recursive formula.

\begin{definition}\label{def:genus-zero-bottom}
  Let $M^1$ be an admissible untwisted module of conformal weight $h_1$, and $U^2$ (resp. $U^3$) a left (resp. right) $A_g(V)$-module where $[\upomega]$ acts as $h_2\id$ (resp. $h_3\id$). Put $h=h_1+h_2-h_3$. 
  A system of linear maps $S=\Set*{S^n_{V\cdots M^1\cdots V}}_{n\in\N}$, where
  \begin{align*}
    S_{V\cdots M^1\cdots V}^n\colon& U^3\otimes V\otimes\cdots\otimes V\otimes M^{1}\otimes V\otimes \cdots V\otimes U^2\longrightarrow\O(\infty\Delta_n),\\
    & u_3\otimes a^{1}\otimes\cdots\otimes a^{k-1}\otimes v\otimes a^{k}\otimes\cdots\otimes a^{n}\otimes u_2\\
    & \longmapsto
    S\bracket*{u_3}{\cfbseq[1][k-1](v,\qq)\cfbseq[k][n]}{u_2},
  \end{align*}
  is said to satisfy the \concept{twisted genus-zero property associated to the datum} 
  \begin{equation}\label{eq:def:Sigma1:bottom}
      \Sigma_{1}(U^3, M^1, U^2):=(\x\colon\oC\to\PP^1, \infty, 1, 0, U^3, M^1, U^2)
  \end{equation}
  if it satisfies the following axioms for all $u_2\in U^2$ and $u_3\in U^3$: 
  \begin{enumerate}
    \item Properties (2)--(6) in \cref{def:genus-zero}, with $u_3\in U^3$ and $u_2\in U^2$.
    \item \textbf{Monomial property}: There is a linear functional $\varphi\in (U^3\otimes M^1\o U^2)^{\ast}$ such that
     \begin{equation}\label{eq:monomial}
       S\bracket*{u_3}{(v,\qq)}{u_2}=
       \braket*{\varphi}{u_3\otimes w^{-L(0)+h_1}v\o u_2}.
     \end{equation}
    \item \textbf{Recursive formula about $U^3$ and $V$}: For any $a\in V^r$, we have
    \begin{equation}\label{eq:recursive_U3}
      \begin{aligned} 
        S\bracket*{u_3}{(a,\pp)\cdots}{u_2}
        &=  
        S\bracket*{u_3\cdot [a]}{\cdots}{u_2}z^{-\wt a}\\
        &\phantom{=}  +\sum_{k=1}^{n}\sum_{i\ge 0} 
        F_{\wt a-1+\delta(r)+\frac{r}{T},i}(\pp,\pp_{k}) 
        S\bracket*{u_3}{\cdots(\vo{a}{i}a^{k},\pp_{k})\cdots}{u_2} \\
        &\phantom{=}  +\sum_{i\ge 0} 
        F_{\wt a-1+\delta(r)+\frac{r}{T},i}(\pp,\qq) 
        S\bracket*{u_3}{\cdots(\vo{a}{i}v,\qq)}{u_2},
      \end{aligned}
    \end{equation}
  \item \textbf{Recursive formula about $U^2$ and $V$}: For any $a\in V^r$, we have 
  \begin{equation}\label{eq:recursive_U2}
    \begin{aligned}
      S\bracket*{u_3}{\cdots(a,\pp)}{u_2}
      &=  
      S\bracket*{u_3}{\cdots}{[a]\cdot u_2}z^{-\wt a}\\
      &\phantom{=}  +\sum_{k=1}^{n}\sum_{i\ge 0} 
      F_{\wt a-1+\frac{r}{T},i}(\pp,\pp_{k}) 
      S\bracket*{u_3}{\cdots(\vo{a}{i}a^{k},\pp_{k})\cdots}{u_2} \\
      &\phantom{=}  +\sum_{i\ge 0} 
      F_{\wt a-1+\frac{r}{T},i}(\pp,\qq) 
      S\bracket*{u_3}{\cdots(\vo{a}{i}v,\qq)}{u_2},
    \end{aligned}
  \end{equation}
  \end{enumerate}
  The vector space consists of systems of linear maps $S=\Set*{S^n_{V\cdots M^1\cdots V}}_{n\in\N}$ satisfying the above axioms is called the \concept{space of $g$-twisted restricted correlation functions associated to the datum $\Sigma_{1}(U^3, M^1, U^2)$} and is denoted by $\Cor[\Sigma_{1}(U^3, M^1, U^2)]$. 
\end{definition}

\begin{proposition}\label{prop:cor-restriction}
  Let $S$ be a system of $g$-twisted correlation functions associated to the datum $\Sigma_{1}(N^3, M^1, M^2)$ in \labelcref{eq:def:Sigma1}, then its restriction to the bottom levels of $N^3$ and $M^2$ gives a system  of $g$-twisted restricted correlation functions associated to the datum $\Sigma_{1}(N^3(0), M^1, M^2(0))$. 
\end{proposition}
\begin{proof}
  We first show the \emph{monomial property}. 
  Since we have $S\bracket*{u_3}{(v,\qq)}{u_2}\in\C[w^{\pm\frac{1}{T}}]$ by \emph{truncation property}, it suffics to show the $w$-derivative of $S\bracket*{u_3}{(w^{L(0)-h_1}v,\qq)}{u_2}$ vanishes for all $u_3\in N^3(0)$, $v\in M^1$, and $u_2\in M^2(0)$.
  Indeed, for homogeneous $v\in M^1$, we have 
  \begin{align*}
    \MoveEqLeft
    \pdv{w}\left(S\bracket*{u_3}{(v,\qq)}{u_2}w^{\deg v}\right) =\pdv{w}\left(S\bracket*{u_3}{(v,\qq)}{u_2}w^{-h}w^{\deg{v}+h}\right)\\
    &= \pdv{w}\left(S\bracket*{u_3}{(v,\qq)}{u_2}w^{-h}\right)w^{\deg{v}+h} + S\bracket*{u_3}{(v,\qq)}{u_2}w^{-h}\pdv{w}\left(w^{\deg{v}+h}\right), \\
    \shortintertext{by the \emph{$\vo{L}{-1}$-property} \labelcref{eq:L-derivative-w},}
    &=
    S\bracket*{u_3}{(\vo{L}{-1}v,\qq)}{u_2}w^{\deg{v}} + (\deg{v}+h)S\bracket*{u_3}{(v,\qq)}{u_2}w^{\deg{v}-1}, \\
    \shortintertext{by the \emph{associativity} \labelcref{eq:associativity_p1q},}
    &=
    \Res_{\pp=\qq}S\bracket*{u_3}{(\upomega,\pp)(v,\qq)}{u_2}w^{\deg{v}}\d{z} + (\deg{v}+h)S\bracket*{u_3}{(v,\qq)}{u_2}w^{\deg{v}-1} \\
    &=\Res_{\pp=\qq}S\bracket*{u_3}{(\upomega,\pp)(v,\qq)}{u_2}zw^{\deg{v}-1}\d{z} \\
    &\qquad - \Res_{\pp=\qq}S\bracket*{u_3}{(\upomega,\pp)(v,\qq)}{u_2}(z-w)w^{\deg{v}-1}\d{z} + \cdots,\\
    \shortintertext{applying the \emph{\nameref{eq:ResidueSum}} to $S\bracket*{u_3}{(\upomega,\pp)(v,\qq)}{u_2}zw^{\deg{v}-1}\d{z}$, and noticing that its possible poles on $\PP^1$ are $0$, $\infty$, and $\qq$,}
    &= -(\tfrac{1}{T}\Res_{\pp=0}+\tfrac{1}{T}\Res_{\pp=\infty})S\bracket*{u_3}{(\upomega,\pp)(v,\qq)}{u_2}zw^{\deg{v}-1}\d{z} \\
    &\qquad - \Res_{\pp=\qq}S\bracket*{u_3}{(\upomega,\pp)(v,\qq)}{u_2}(z-w)w^{\deg{v}-1}\d{z} + \cdots,\\
    \shortintertext{by the \emph{generating properties} \labelcref{eq:generating_M2,eq:generating_M3}, and the \emph{associativity} \labelcref{eq:associativity_p1q} again,}
    &= - S\bracket*{u_3}{(v,\qq)}{\vo{L}{0}u_2}w^{\deg{v}-1} + S\bracket*{\vo{L}{0}u_3}{(v,\qq)}{u_2}w^{\deg{v}-1} \\
    &\qquad - S\bracket*{u_3}{(\vo{L}{0}v,\qq)}{u_2}w^{\deg{v}-1} + (\deg{v}+h)S\bracket*{u_3}{(v,\qq)}{u_2}w^{\deg{v}-1} \\
    &= (-h_2+h_3-(\deg{v}+h_1)+(\deg{v}+h))S\bracket*{u_3}{(v,\qq)}{u_2}w^{\deg{v}-1} =0.
  \end{align*}
  
  It remains to prove the \emph{recursive formulas} \labelcref{eq:recursive_U3,eq:recursive_U2}.

  For any homogeneous $a\in V^r$, the function $S\bracket*{u_3}{(a,\pp)\cdots}{u_2}z^{m+\frac{r}{T}}$ factors through $\PP^1$ and has only $n+3$ possible poles: $0$, $
  \infty$, $\zseq$, and $w$. 
  Expanding it at $\pp=0$ and applying the \emph{\nameref{lem:ResidueSum}}, we obtain
  \begin{equation}\label{eq:ResSumForS(ap)}
    0=
    \left(\tfrac{1}{T}\Res_{\pp=0}+\tfrac{1}{T}\Res_{\pp=\infty}+\sum_{k=1}^{n}\Res_{\pp=\pp_k}+\Res_{\pp=\qq}\right)S\bracket*{u_3}{(a,\pp)\cdots}{u_2}z^{m+\frac{r}{T}}\d{z}.
  \end{equation}
  By \labelcref{eq:associativity_p1q,eq:associativity_p1p2,eq:generating_M2,eq:generating_M3}, and \labelcref{eq:iota_wz-w}, we have 
  \begin{align*}
    \MoveEqLeft
    \Res_{\pp=0}S\bracket*{u_3}{(a,\pp)\cdots}{u_2}z^{m+\frac{r}{T}}\d{z} =TS\bracket*{u_3}{\cdots}{\vo{a}{m+\frac{r}{T}}u_2}, \\
    \MoveEqLeft
    \Res_{\pp=\infty}S\bracket*{u_3}{(a,\pp)\cdots}{u_2}z^{m+\frac{r}{T}}\d{z} =-TS\bracket*{\theta(\lo{a}{m+\frac{r}{T}})u_3}{\cdots}{u_2}, \\
    \MoveEqLeft
    \Res_{\pp=\pp_k}S\bracket*{u_3}{(a,\pp)\cdots(a^{k},\pp_{k})\cdots}{u_2}z^{m+\frac{r}{T}}\d{z} \\
    \overset{\iota_{\pp=\pp_k}}&{=}\sum_{i\ge 0}\binom{m+\frac{r}{T}}{i}
    \Res_{\pp=\pp_k}S\bracket*{u_3}{(a,\pp)\cdots(a^{k},\pp_{k})\cdots}{u_2}z_k^{m+\frac{r}{T}-i}(z-z_k)^i\d{z}\\
    &=\sum_{i\ge 0}\binom{m+\frac{r}{T}}{i}S\bracket*{u_3}{\cdots(\vo{a}{i}a^{k},\pp_{k})\cdots}{u_2}z_k^{m+\frac{r}{T}-i}, \\
    \MoveEqLeft
    \Res_{\pp=\qq}S\bracket*{u_3}{(a,\pp)\cdots(v,\qq)}{u_2}z^{m+\frac{r}{T}}\d{z} \\
    \overset{\iota_{\pp=\qq}}&{=}\sum_{i\ge 0}\binom{m+\frac{r}{T}}{i}
    \Res_{\pp=\qq}S\bracket*{u_3}{(a,\pp)\cdots(v,\qq)}{u_2}w^{m+\frac{r}{T}-i}(z-w)^i\d{z}\\
    &=\sum_{i\ge 0}\binom{m+\frac{r}{T}}{i}S\bracket*{u_3}{\cdots(\vo{a}{i}v,\qq)}{u_2}w^{m+\frac{r}{T}-i},
  \end{align*}
  where the involved summation is finite since $\vo{a}{i}a^k=0$ and $\vo{a}{i}v=0$ for $i\gg 0$.
  Note that 
  \begin{align*}
      S\bracket*{u_3}{\cdots}{\vo{a}{m+\frac{r}{T}}u_2} = 0
      &\qquad\text{if}\qquad
      m+\tfrac{r}{T}>\wt{a}-1,\\
      S\bracket*{\theta(\lo{a}{m+\frac{r}{T}})u_3}{\cdots}{u_2} = 0
      &\qquad\text{if}\qquad
      m+\tfrac{r}{T}<\wt{a}-1.
  \end{align*}
Then it follows from \labelcref{eq:ResSumForS(ap)} and \cref{tab:EXP} that   
    \begin{align*}
      \MoveEqLeft\iota_{\pp=0}S\bracket*{u_3}{(a,\pp)\cdots}{u_2}\\
      &= 
      \sum_{m\le\floor{\wt{a}-1-\frac{r}{T}}}z^{-1-m-\frac{r}{T}}\bigg(S\bracket*{\theta(\lo{a}{m+\frac{r}{T}})u_3}{\cdots}{u_2}\\
      &\qquad -\sum_{k=1}^{n}\sum_{i\ge 0}\binom{m+\frac{r}{T}}{i}z_k^{m+\frac{r}{T}-i}S\bracket*{u_3}{\cdots(\vo{a}{i}a^{k},\pp_{k})\cdots}{u_2} \\
      &\qquad -\sum_{i\ge 0}\binom{m+\frac{r}{T}}{i}w^{m+\frac{r}{T}-i}S\bracket*{u_3}{\cdots(\vo{a}{i}v,\qq)}{u_2}\bigg) \\
      &= 
      S\bracket*{u_3o(a)}{\cdots}{u_2}z^{-\wt{a}} \\
      &\qquad - 
      \sum_{k=1}^{n}\sum_{i\ge 0}\sum_{m+\frac{r}{T}\le\wt{a}-1}\binom{m+\frac{r}{T}}{i}z^{-m-\frac{r}{T}-1}z_k^{m+\frac{r}{T}-i}S\bracket*{u_3}{\cdots(\vo{a}{i}a^{k},\pp_{k})\cdots}{u_2}\\
      &\qquad - \sum_{i\ge 0}\sum_{m+\frac{r}{T}\le\wt{a}-1}\binom{m+\frac{r}{T}}{i}z^{-m-\frac{r}{T}-1}w^{m+\frac{r}{T}-i}S\bracket*{u_3}{\cdots(\vo{a}{i}v,\qq)}{u_2}\\
      &=  
      S\bracket*{u_3o(a)}{\cdots}{u_2}z^{-\wt a} \\
      &\phantom{=}  +\sum_{k=1}^{n}\sum_{i\ge 0} 
      \iota_{\pp=0}F_{\wt a-1+\delta(r)+\frac{r}{T},i}(\pp,\pp_{k}) 
      S\bracket*{u_3}{\cdots(\vo{a}{i}a^{k},\pp_{k})\cdots}{u_2} \\
      &\phantom{=}  +\sum_{i\ge 0} 
      \iota_{\pp=0}F_{\wt a-1+\delta(r)+\frac{r}{T},i}(\pp,\qq) 
      S\bracket*{u_3}{\cdots(\vo{a}{i}v,\qq)}{u_2} \qquad\text{(by \cref{lem:ExpOfFpq})},
    \end{align*}
  which implies \labelcref{eq:recursive_U3} by the injectivity of $\iota_{\pp=0}$. 
  Finally, expanding $S\bracket*{u_3}{(a,\pp)\cdots}{u_2}$ at $\pp=\infty$ yields a proof of \labelcref{eq:recursive_U2}.
\end{proof}
\begin{remark}
  When $N^3$ is the contragridient module of $M^3$, the \emph{monomial property} follows from the \emph{truncation property} of intertwining operator $I_S$. However, in general we cannot assume that $N^3$ is a contragredient module of some $M^3$. See \cref{warn}.
\end{remark}

By \cref{prop:cor-restriction}, there exists a linear map \[\uppi: \Cor[\Sigma_{1}(N^3, M^1, M^2)]\longrightarrow \Cor[\Sigma_{1}(N^3(0), M^1, M^2(0)].\]
\begin{proposition}\label{prop:injectivity-from-cor-to-corbot}
  If $M^2$ and $N^3$ are lowest-weight modules, then $\uppi$ is injective.  
\end{proposition}
\begin{proof}
    Let $S\in \Cor[\Sigma_{1}(N^3, M^1, M^2)]$ such that $\uppi(S)=0$. Then, $S\bracket*{u_3}{(v,\qq)}{u_2}=0$ for all $u_3\in N^3(0), v\in M^1$ and $u_2\in M^2(0)$. 
    Let $M$ be the subspace 
    \[
      M:=\Set*{v_2\in M^2\given S\bracket*{u_3}{(v,\qq)}{v_2}=0 \txforall[\ ] u_3\in N^3(0), v\in M^1 }
    \]
    Then $M^2(0)\subseteq M$. 
    For any $v_2\in M$, homogeneous $a\in V^r$, and $m\in \Z$, by \labelcref{eq:generating_M2,eq:recursive_U3}, we have 
    \begin{align*}
        \MoveEqLeft
        S\bracket*{u_3}{(v,\qq)}{\vo{a}{m+\frac{r}{T}}v_2}=\Res_{\pp=0}\uppi(S)\bracket*{u_3}{(a,\pp)(v,\qq)}{v_2}z^{m+\frac{r}{T}}\d{z}\\
        &=\Res_{\pp=0}\uppi(S)\bracket*{u_3o(a)}{(v,\qq)}{v_2}z^{m+\frac{r}{T}-\wt a}\d{z}\\
        &\quad +\sum_{i\ge 0}F_{\wt a-1+\delta(r)+\frac{r}{T},i}(\pp,\qq)\Res_{\pp=0}\uppi(S)\bracket*{u_3}{(\vo{a}{i}v,\qq)}{v_2}z^{m+\frac{r}{T}}\d{z}\\
        &=0.
    \end{align*}
    This implies that $M$ is a submodule of $M^2$ containing $M^2(0)$, hence $M=M^2$. Therefore
    \[S\bracket*{u_3}{(v,\qq)}{v_2}=0, \txforall[\qquad] u_3\in N^3(0), v\in M^1, v_2\in M^2.\]
    Similarly, using \labelcref{eq:generating_M3,eq:recursive_U2}, we can show
    \[S\bracket*{v_3}{(v,\qq)}{v_2}=0,\txforall[\qquad] v_3\in N^3, v\in M^1, v_2\in M^2.\]
    Suppose all the $(n+3)$-point functions in $S$ vanish. 
    Then, for any $a\in V$, by \labelcref{eq:associativity_p1q},
    \begin{align*}
        \MoveEqLeft
        \iota_{\pp=\qq}S\bracket*{v_3}{(a, \pp)\cdots(v,\qq)}{v_2}\\
        &=\sum_{k\in \Z}\left(\Res_{\pp=\qq}S\bracket*{v_3}{\cdots(a,\pp)(v,\qq)}{v_2}(z-w)^k\d{z}\right)(z-w)^{-k-1}\\
        &=\sum_{k\in \Z}S\bracket*{v_3}{\cdots(\vo{a}{k}v,\qq)}{v_2}(z-w)^{-k-1}=0,
    \end{align*}
    where the last equality follows from the inductive hypothesis. 
    Since $\iota_{\pp=\qq}$ is injective,   
    \[S\bracket*{v_3}{(a,\pp)\cdots(v,\qq)}{v_2}=0,\txforall[\qquad] v_3\in N^3, v\in M^1, v_2\in M^2.\]
    This shows the system of functions $S$ vanishes.
\end{proof}

\subsection{Extending restricted correlation functions from the bottom levels}
In this subsection, our objective is to establish an isomorphism between the spaces of correlation functions associated to the datum $\Sigma_{1}(U^3, M^1, U^2)$ and $\Sigma_{1}(\overline{M}(U^3), M^1, \overline{M}(U^2))$ respectively. Here $U^2$ (resp. $U^3$) is an irreducible left (resp. right) $A_g(V)$-module, and $\overline{M}(-)$ assigns an $A_g(V)$-module (or $A_{g^{-1}}(V)$-module) to the associated \emph{generalized Verma module}, as defined in \cite{DLM1}.

Recall that there is an epimorphism from $\L_g(V)_0$ to $A_g(V)$ as Lie algebras. Hence any $A_g(V)$-module $U$ can be regarded as an $\L_g(V)_0$-module, where $\L_g(V)$ is the twisted Lie algebra $\L_g(V)$ in \labelcref{eq:def:LgV}. Let $U$ be an $\L_g(V)_{-}+\L_g(V)_{0}$-module by letting $\L_g(V)_{-}$ act trivially and consider the following induced module 
\begin{equation*}
    \U\big(\L_g(V)\big)\otimes_{\U\big(\L_g(V)_{-}+\L_g(V)_{0}\big)}U.
\end{equation*}
Then the \concept{generalized Verma module} $\overline{M}(U)$ generated by $U$ is defined to be the quotient space of the above module modulo the submodule generated by all the coefficients of the \emph{twisted Jacobi identity}. 

\cref{prop:cor-restriction} shows that any system of correlation functions associated to the datum $\Sigma_{1}(\overline{M}(U^3), M^1, \overline{M}(U^2))$ restricts to one for the datum $\Sigma_{1}(U^3, M^1, U^2)$. 
In the rest of \cref{sec3}, we prove that the converse is also true by adopting a similar method as in \cite{Z,Liu}. 

\subsubsection*{Extending $U^2$}
We denote the tensor product space $T(\L_g(V))\otimes U$ by $M(U)$, where $T(\L_g(V))$ is the tensor algebra of the twisted Lie algebra $\L_g(V)$.
The space $M(U)$ is spanned by
\begin{equation}
  \lo{b^1}{\frac{r_1}{T}+m_1}\otimes\cdots\otimes\lo{b^p}{\frac{r_p}{T}+m_p} \otimes u,\quad b^i\in V^{r_i}, m_i\in \Z, u\in U, p\in \N.
\end{equation}

Given a system of correlation functions $S\colon U^3\otimes\fun{Sym}[V,\cdots,V,M^1]\otimes U^2\to\O(\infty\Delta_n)$, we first extend it to $U^3\otimes\fun{Sym}[V,\cdots,V,M^1]\otimes M(U^2)$ by
\begin{equation*}
  \begin{aligned}
    \MoveEqLeft
    S\bracket*{u_3}{\cdots}{\lo{b^1}{\frac{r_1}{T}+m_1}\otimes\cdots\otimes\lo{b^p}{\frac{r_p}{T}+m_p} \otimes u_2}\\
    &:=\tfrac{1}{T}\Res_{\pp_{+1}=0}\cdots\tfrac{1}{T}\Res_{\pp_{+p}=0}\\ 
    &\qquad\qquad 
    S\bracket*{u_3}{\cdots\cfbseqb}{u_2}
    z_{+1}^{\frac{r_1}{T}+m_1}\cdots z_{+p}^{\frac{r_p}{T}+m_p}
    \d{z_{+p}}\cdots\d{z_{+1}}.
  \end{aligned}
\end{equation*}
It is easy to show that such a family of functions $S$ is well-defined using the \emph{$\vo{L}{-1}$-derivative property} \labelcref{eq:L-derivative-z}.
Define the radical of the family $S$ by 
\begin{equation}\label{3.12}
  \begin{aligned}
    \Rad(S):=\Set*{v_2\in M(U^2)\given S\bracket*{u_3}{\cdots}{v_2}=0\text{ for all }u_3\in U^3}. 
  \end{aligned}
\end{equation}
Then define $\Rad(U^2):=\bigcap_S\Rad(S)$, where $S$ ranges over $\Cor[\Sigma_{1}(U^3, M^1, U^2)]$. 
It is easy to see that for any $v_2\in M(U^2)$, we have
\begin{equation}\label{eq:SfromU2toMU2}
  \begin{aligned}
    \MoveEqLeft
    S\bracket*{u_3}{\cdots}{\lo{b^1}{\frac{r_1}{T}+m_1}\otimes\cdots\otimes\lo{b^p}{\frac{r_p}{T}+m_p}\otimes v_2}\\
    &=\tfrac{1}{T}\Res_{\pp_{+1}=0}\cdots\tfrac{1}{T}\Res_{\pp_{+p}=0}\\ 
    &\qquad\qquad 
    S\bracket*{u_3}{\cdots\cfbseqb}{v_2}
    z_{+1}^{\frac{r_1}{T}+m_1}\cdots z_{+p}^{\frac{r_p}{T}+m_p}
    \d{z_{+p}}\cdots\d{z_{+1}}.
  \end{aligned}
\end{equation}
\begin{lemma}\label{lm3.5}
  Let $S\in \Cor[\Sigma_{1}(U^3, M^1, U^2)]$, and let $\varphi$ be as in \labelcref{eq:monomial}. If $\varphi=0$ then $S=0$. 
\end{lemma}
\begin{proof}
  Clearly, $S\bracket*{u_3}{(v, \qq)}{u_2}=0$ for any $u_3\in U^3$ and $u_2\in U^2$. Suppose all the $(n+3)$-point functions $S$ vanish. For any homogeneous $a\in V^r$, by \labelcref{eq:recursive_U3},
  \begin{align*}
    S\bracket*{u_3}{(a,\pp)\cdots}{u_2}
    &=  
    S\bracket*{u_3\cdot [a]}{\cdots}{u_2}z^{-\wt a}\\
    &\phantom{=}  +\sum_{k=1}^{n}\sum_{i\ge 0} 
    F_{\wt a-1+\delta(r)+\tfrac{r}{T},i}(\pp,\pp_{k}) 
    S\bracket*{u_3}{\cdots(\vo{a}{i}a^{k},\pp_{k})\cdots}{u_2} \\
    &\phantom{=}  +\sum_{i\ge 0} 
    F_{\wt a-1+\delta(r)+\tfrac{r}{T},i}(\pp,\qq) 
    S\bracket*{u_3}{\cdots(\vo{a}{i}v,\qq)}{u_2},
  \end{align*}
  and the right-hand side is $0$ by the induction hypothesis. 
\end{proof}
\begin{lemma}\label{lem:Rad-prop}
  The following properties hold for $\Rad(U^2)$:
  \begin{enumerate}
    \item\label{lem:Rad-prop(1)} $\lo{b}{\frac{r}{T}+m}\otimes \Rad(U^2)\subset\Rad(U^2)$ for all $b\in V^r$ and $m\in \Z$.
    \item\label{lem:Rad-prop(2)} $U^2\cap\Rad(U^2)=0$, where $U^2$ is viewed as $\C\otimes U^2\subset M(U^2)$.
    \item\label{lem:Rad-prop(3)} $\lo{b}{\frac{r}{T}+m}\otimes u_2\in \Rad(U^2)$ for all $b\in V^r$ and $m\in \Z$ such that $\deg \lo{b}{\frac{r}{T}+m}<0$. 
    \item\label{lem:Rad-prop(4)} $\lo{b}{\wt b-1}\otimes u_2-1\otimes [b]\cdot u_2\in \Rad(U^2)$ for $b\in V^0$.
  \end{enumerate}
\end{lemma}
\begin{proof}
  \labelcref{lem:Rad-prop(1)} is clear. 
  For property \labelcref{lem:Rad-prop(2)}, suppose there exists a nonzero $u_2\in U^2\cap \Rad(U^2)$. Then for any system of correlation functions $S$ with a linear functional $\varphi$ as in \labelcref{eq:monomial} and any homogeneous $a\in V^r$, $u_3\in U^3$, $v\in M^1$, and $u_2\in U^2$, by \labelcref{eq:monomial,eq:recursive_U2}, we have 
  \begin{align*}
      \braket*{\varphi}{u_3\otimes v\otimes([a]\cdot u_2)}
      &=S\bracket*{u_3}{(v,\qq)}{[a]\cdot u_2}w^{\deg v}\\
      &=S\bracket*{u_3}{(a,\pp)(v,\qq)}{u_2}z^{\wt a}w^{\deg v}\\
      &\phantom{=}  -\sum_{i\ge 0} 
      F_{\wt a-1+\frac{r}{T},i}(\pp,\qq) 
      S\bracket*{u_3}{(\vo{a}{i}v,\qq)}{u_2}z^{\wt a}w^{\deg v}\\
      &=0.
  \end{align*} 
  On the other hand, by \cite[Lemma 2.1]{DLM1}, $A_g(V)$ is a quotient algebra of $A(V^0)$. Hence $U^2=A(V^0)u_2$, and $\varphi$ vanishes on all the entire $U^3\otimes M^1\otimes U^2$. This implies $S=0$ by \cref{lm3.5}, which is a contradiction.
  
  For property \labelcref{lem:Rad-prop(3)}, given any $u_2\in U^2$ and any homogeneous $b\in V^r$ with $\deg \lo{b}{\frac{r}{T}+m}<0$, by \labelcref{eq:recursive_U3,eq:SfromU2toMU2}, we have
  \begin{align*}
      \MoveEqLeft
      S\bracket*{u_3}{\cdots}{\lo{b}{\frac{r}{T}+m}\otimes u_2}\\
      &=
      \tfrac{1}{T}\Res_{\pp=0}S\bracket*{u_3}{\cdots(b,\pp)}{u_2}z^{\frac{r}{T}+m}\d{z}\\
      &=
      S\bracket*{u_3\cdot b}{\cdots}{u_2}\tfrac{1}{T}\Res_{\pp=0}z^{-\wt b+\frac{r}{T}+m}\d{z}\\
      &\phantom{=}  +\sum_{k=1}^{n}\sum_{i\ge 0} 
      S\bracket*{u_3}{\cdots(\vo{b}{i}a^{k},\pp_{k})\cdots}{u_2} 
      \tfrac{1}{T}\Res_{\pp=0}F_{\wt{b}-1+\delta(r)+\frac{r}{T},i}(\pp,\pp_{k})z^{\frac{r}{T}+m}\d{z} \\
      &\phantom{=}  +\sum_{i\ge 0} 
      S\bracket*{u_3}{\cdots(\vo{b}{i}v,\qq)}{u_2}
      \tfrac{1}{T}\Res_{\pp=0}F_{\wt{b}-1+\delta(r)+\frac{r}{T},i}(\pp,\qq)z^{\frac{r}{T}+m}\d{z}\\
      &=0,
  \end{align*}
  where the last equality follows from the fact that $F_{n,i}(\pp,\qq)z^n$ is holomorphic at $\pp=0$. 

For property \labelcref{lem:Rad-prop(4)}, given any $u_2\in U^2$ and any homogeneous $b\in V^0$, by \labelcref{eq:recursive_U2} we have
  \begin{align*}
      \MoveEqLeft
      S\bracket*{u_3}{\cdots}{\lo{b}{\wt b-1}\otimes u_2}\\
      &=
      \tfrac{1}{T}\Res_{\pp=0}S\bracket*{u_3}{\cdots(b,\pp)}{u_2}z^{\wt b-1}\d{z}\\
      &=
      S\bracket*{u_3}{\cdots}{[b]\cdot u_2}\tfrac{1}{T}\Res_{\pp=0}z^{-1}\d{z}\\
      &\phantom{=}  +\sum_{k=1}^{n}\sum_{i\ge 0} 
      S\bracket*{u_3}{\cdots(\vo{b}{i}a^{k},\pp_{k})\cdots}{u_2} 
      \tfrac{1}{T}\Res_{\pp=0}F_{\wt{b}-1,i}(\pp,\pp_{k})z^{\wt b-1}\d{z} \\
      &\phantom{=}  +\sum_{i\ge 0} 
      S\bracket*{u_3}{\cdots(\vo{b}{i}v,\qq)}{u_2}
      \tfrac{1}{T}\Res_{\pp=0}F_{\wt{b}-1,i}(\pp,\qq)z^{\wt b-1}\d{z}\\
      &=S\bracket*{u_3}{\cdots}{[b]\cdot u_2}.&\qedhere
  \end{align*}
\end{proof}
In the following lemma, we use the same notation for the elements in $M(U^2)$ and their images in $M(U^2)/\Rad(U^2)$.
\begin{lemma}\label{lmeq:oab}
  For any $a\in V^r$, $b\in V^s$, $v_2\in M(U^2)$, and $m,n,l\in \Z$, the following element of $M(U^2)/\Rad(U^2)$ vanishes:
  
    \begin{align*}
      &-\sum_{i\ge 0} 
      \binom{l}{i}(-1)^i
      \lo{a}{\frac{r}{T}+m+l-i}\otimes \lo{b}{\frac{s}{T}+n+i}\otimes v_2\\
      &+\sum_{i\ge 0}
      \binom{l}{i}(-1)^{l+i}
      \lo{b}{\frac{s}{T}+n+l-i}\otimes \lo{a}{\frac{r}{T}+m+i}\otimes v_2 \numberthis\label{eq:compJacU2} \\
      &+\sum_{j\ge 0}
      \binom{m+\frac{r}{T}}{j}
      \lo{(\vo{a}{j+l}b)}{\frac{r+s}{T}+m+n-j}\otimes v_2.
    \end{align*}
\end{lemma}
Note that the summations in \labelcref{eq:compJacU2} are finite by property \itemcref{lem:Rad-prop}{(3)}.
\begin{proof}
  Indeed, for any system of correlation functions $S$, we have
  \begin{align*}
    &
    S\bracket*{u_3}{\cdots}{\sum_{i\ge 0} 
    \binom{l}{i}(-1)^i
    \lo{a}{\frac{r}{T}+m+l-i}\otimes\lo{b}{\frac{s}{T}+n+i}\otimes v_2}\\
    &=\tfrac{1}{T}\Res_{\pp_{+1}=0}\tfrac{1}{T}\Res_{\pp_{+2}=0}\\
    &\qquad\qquad
    \sum_{i\ge 0}\binom{l}{i}(-1)^i 
    S\bracket*{u_3}{\cdots(a,\pp_{+1})(b,\pp_{+2})}{v_2}z_{+1}^{\frac{r}{T}+m+l-i}z_{+2}^{\frac{s}{T}+n+i}\d{z_{+2}}\d{z_{+1}}\\
    &=\tfrac{1}{T}\Res_{\pp_{+1}=0}\tfrac{1}{T}\Res_{\pp_{+2}=0}\\
    &\qquad\qquad
    S\bracket*{u_3}{\cdots(a,\pp_{+1})(b,\pp_{+2})}{v_2}(z_{+1}-z_{+2})^{l}z_{+1}^{\frac{r}{T}+m}z_{+2}^{\frac{s}{T}+n}\d{z_{+2}}\d{z_{+1}}\\
    \overset{\ast}&{=}
    \tfrac{1}{T}\Res_{\pp_{+2}=0}(\tfrac{1}{T}\Res_{\pp_{+1}=0}+\Res_{\pp_{+1}=\pp_{+2}})\\
    &\qquad\qquad
    S\bracket*{u_3}{\cdots(a,\pp_{+1})(b,\pp_{+2})}{v_2}(z_{+1}-z_{+2})^{l}z_{+1}^{\frac{r}{T}+m}z_{+2}^{\frac{s}{T}+n}\d{z_{+1}}\d{z_{+2}}\\
    &{=}
    \tfrac{1}{T}\Res_{\pp_{+2}=0}\tfrac{1}{T}\Res_{\pp_{+1}=0}\\
    &\qquad\qquad
    S\bracket*{u_3}{\cdots(b,\pp_{+2})(a,\pp_{+1})}{v_2}(z_{+1}-z_{+2})^{l}z_{+1}^{\frac{r}{T}+m}z_{+2}^{\frac{s}{T}+n}\d{z_{+1}}\d{z_{+2}}\\
    &\phantom{=} +
    \tfrac{1}{T}\Res_{\pp_{+2}=0}\Res_{\pp_{+1}=\pp_{+2}}\\
    &\qquad\qquad
    S\bracket*{u_3}{\cdots(a,\pp_{+1})(b,\pp_{+2})}{v_2}(z_{+1}-z_{+2})^{l}z_{+1}^{\frac{r}{T}+m}z_{+2}^{\frac{s}{T}+n}\d{z_{+1}}\d{z_{+2}}\\
    \overset{\ast\ast}&{=}
    \tfrac{1}{T}\Res_{\pp_{+2}=0}\tfrac{1}{T}\Res_{\pp_{+1}=0}\\
    &\qquad\qquad
    \sum_{i\ge 0}\binom{l}{i}(-1)^{l+i}
    S\bracket*{u_3}{\cdots(b,\pp_{+2})(a,\pp_{+1})}{v_2}z_{+2}^{\frac{s}{T}+n+l-i}z_{+1}^{\frac{r}{T}+m+i}\d{z_{+1}}\d{z_{+2}}\\
    &\phantom{=} +
    \tfrac{1}{T}\Res_{\pp_{+2}=0}\Res_{\pp_{+1}=\pp_{+2}}\\
    &\qquad\qquad
    \sum_{j\ge  0}\binom{m+\frac{r}{T}}{j}
    S\bracket*{u_3}{\cdots(a,\pp_{+1})(b,\pp_{+2})}{v_2}(z_{+1}-z_{+2})^{l+j}z_{+2}^{\frac{r+s}{T}+m+n-j}\d{z_{+1}}\d{z_{+2}}\\
    &=
    \sum_{i\ge 0}\binom{l}{i}(-1)^i
    S\bracket*{u_3}{\cdots}{\lo{b}{\frac{s}{T}+n+l-i}\otimes\lo{a}{\frac{r}{T}+m+i}\otimes v_2}\\ 
    &\phantom{=} +
    \tfrac{1}{T}\Res_{\pp_{+2}=0}
    S\bracket*{u_3}{\cdots(\vo{a}{l+j}b,\pp_{+2})}{v_2}z_{+2}^{\frac{r+s}{T}+m+n-j}\d{z_{+2}}\\
    &= 
    S\bracket*{u_3}{\cdots}{\sum_{i\ge 0}\binom{l}{i}(-1)^i\lo{b}{\frac{s}{T}+n+l-i}\otimes\lo{a}{\frac{r}{T}+m+i}\otimes v_2}\\ 
    &\phantom{=} +
    S\bracket*{u_3}{\cdots}{\sum_{j\ge 0}
    \binom{m+\frac{r}{T}}{j}
    \lo{(\vo{a}{j+l}b)}{\frac{r+s}{T}+m+n-j}\otimes v_2}.
  \end{align*}

Equality $\ast$ follows from the residue sum formula since both of the following functions
  \begin{align*}
    \pp_{+1}&\mapsto\tfrac{1}{T}\Res_{\pp_{+2}=0}S\bracket*{u_3}{\cdots(a,\pp_{+1})(b,\pp_{+2})}{v_2}(z_{+1}-z_{+2})^{l}z_{+1}^{\frac{r}{T}+m}z_{+2}^{\frac{s}{T}+n}\d{z_{+2}}\\
    \pp_{+1}&\mapsto S\bracket*{u_3}{\cdots(a,\pp_{+1})(b,\pp_{+2})}{v_2}(z_{+1}-z_{+2})^{l}z_{+1}^{\frac{r}{T}+m}z_{+2}^{\frac{s}{T}+n}
  \end{align*}
  factor through $\PP^1$, and the second fucntion has one extra possible pole at $z_{+1}=z_{+2}$. Furthermore, at any common pole $\pp_{\ast}\neq0$ of these functions, we have 
  \[
    \Res_{\pp_{+1}=\pp_{\ast}}\Res_{\pp_{+2}=0}\cdots = \Res_{\pp_{+2}=0}\Res_{\pp_{+1}=\pp_{\ast}}\cdots
  \]
  since $\pp_{\ast}$ is away from the divisor $\pp_{+1}=\pp_{+2}$. Equality $\ast\ast$ follows from expanding $(z_{+1}-z_{+2})^{l}$ at $\pp_{+1}=0$ and $z_{+1}^{\frac{r}{T}+m}$ at $\pp_{+1}=\pp_{+2}$, respectively.
\end{proof}
\begin{remark}
  In particular, taking $l=0$ in \cref{lmeq:oab}, we have 
  \begin{equation}
  \begin{aligned}
      \MoveEqLeft
      \lo{a}{m+\frac{r}{T}}\otimes \lo{b}{n+\frac{s}{T}}\otimes v_2-\lo{b}{n+\frac{s}{T}}\otimes \lo{a}{m+\frac{r}{T}}\otimes v_2\\
      &\qquad \equiv\sum_{j\ge 0}
      \binom{m+\frac{r}{T}}{j}
      \lo{(\vo{a}{j}b)}{\frac{r+s}{T}+m+n-j}\otimes v_2\pmod{\Rad(U^2)}.
  \end{aligned}
  \end{equation}
  Therefore, $M(U^2)/\Rad(U^2)$ is a $\L_g(V)$-module. 
  Furthermore, \itemcref{lem:Rad-prop}{(4)} allows us to unambiguously write $\lobmseq u_2$ for the image of $\lo{b^1}{\frac{r_1}{T}+m_1} \otimes\cdots\otimes\lo{b^p}{\frac{r_p}{T}+m_p}\otimes u_2$ in $M(U^2)/\Rad(U^2)$. Moreover, it is clear that $M(U^2)/\Rad(U^2)$ is spanned by the following elements: 
  \begin{equation}\label{spann}
  \lobmseq u_2,
  \end{equation}
  where $u_2\in U^2$, $b^i\in V^{r_i}$, $m_i\in \Z$ for all $i$, and $\deg \lo{b^1}{\frac{r_1}{T}+m_1}\ge \cdots\ge \deg \lo{b^p}{\frac{r_p}{T}+m_p}$.

\end{remark} 
\begin{definition}\label{df3.7}
  Denote $M(U^2)/\Rad(U^2)$ by $\widetilde{M}(U^2)$. 
  Define a vertex operator by
  \begin{equation}\label{3.27'}
 Y_{\widetilde{M}(U^2)}: V\rightarrow \End(\widetilde{M}(U^2))[[z,z^{-1}]],\quad    Y_{\widetilde{M}(U^2)}(a,z) = \sum_{n\in \Z} \lo{a}{\frac{n}{T}} z^{-\frac{n}{T}-1},
  \end{equation}
  where $a\in V$ and $a(\frac{n}{T})\in \L_g(V)$ for all $n\in \Z$.
\end{definition}
Furthermore, we introduce a gradation on $\widetilde{M}(U^2)$ by 
  \begin{equation}\label{3.24'}
    \deg \left(\lobmseq u_2\right) := 
    \sum_{i=1}^p \deg\lo{b^i}{\frac{r_i}{T}+m_i}, 
  \end{equation}
  where $b_i\in V^{r_i}$, $m_i\in \Z$, and $u_2\in U^2$. 
  Then $\widetilde{M}(U^2)=\bigoplus_{m\in \N}\widetilde{M}(U^2)(\frac{m}{T})$ by the type of its spanning elements \eqref{spann} and \itemcref{lem:Rad-prop}{(3)}.

\begin{proposition}\label{prop3.8}
  The pair $(\widetilde{M}(U^2),Y_{\widetilde{M}(U^2)})$ in \cref{df3.7} defines an admissible $g$-twisted module of conformal weight $h_2$. 
\end{proposition}
\begin{proof}
  By \itemcref{lem:Rad-prop}{(1)}, the vertex operator $Y_{\widetilde{M}(U^2)}(-,z)$ is well-defined. 
  Given $a\in V^r$, by the definition of gradation \labelcref{3.24'}, we have 
  \[
    \lo{a}{\frac{r}{T}+n}\widetilde{M}(U^2)(\tfrac{m}{T})\subset\widetilde{M}(U^2)(\tfrac{m}{T}+\deg\lo{a}{\frac{r}{T}+n}).
  \] 
 Hence $\widetilde{M}(U^2)(\tfrac{m}{T})=0$ for $m<0$ by \itemcref{lem:Rad-prop}{(3)}. 
  This shows $\lo{a}{\frac{r}{T}+n}v_2=0$ for $n\gg 0$. The Jacobi identity of $Y_{\widetilde{M}(U^2)}$ follows from \cref{lmeq:oab}. 

  By adopting a similar argument as \cite[Proposition 3.6 and 3.7]{Liu}, together with the assumption that $[\upomega]$ acts as $h_2 \id$ on $U^2$, we can show that $Y_{\widetilde{M}(U^2)}$ satisfies the vacuum property, and $\widetilde{M}(U^2)=\bigoplus_{n\in\N}\widetilde{M}(U^2))(\frac{n}{T})$ with the bottom level $\widetilde{M}(U^2)(0)=U^2$. 
  Moreover, for any $a\in V^0$, its action on $\widetilde{M}(U^2)(0)=U^2$ agrees with the operator $o(a)=\vo{a}{\wt a-1}:=\Res_z z^{\wt a-1}Y_{\widetilde{M}(U^2)}(a,z)$, and each $\widetilde{M}(U^2)(\frac{n}{T})$ is an eigenspace of $\vo{L}{0}=o(\upomega)$ with the eigenvalue $\frac{n}{T}+h_2$. 
\end{proof}

So far, we have extended $S$ to $U^3\otimes\fun{Sym}[V,\cdots,V,M^1]\otimes\widetilde{M}(U^2)$. 
The last factor can be further extended to $\overline{M}(U^2)$ since $\widetilde{M}(U^2)$ is a quotient module of the generalized Verma module $\overline{M}(U^2)$.

Now let $M^{\op}(U^3)=U^3\otimes T(\L_g(V))$. By adopting a slight modification of the argument in this subsection, we can extend $S$ to $M^{\op}(U^3)\otimes\fun{Sym}[V,\cdots,V,M^1]\otimes \widetilde{M}(U^2)$ by
\begin{equation*}
  \begin{aligned}
    \MoveEqLeft
    S\bracket*{u_3\otimes\lo{b^p}{\frac{r_p}{T}+m_p}\otimes\cdots\otimes\lo{b^1}{\frac{r_1}{T}+m_1}}{\cdots}{v_2}\\
    &:=(-\tfrac{1}{T}\Res_{\pp_{+1}=\infty})\cdots(-\tfrac{1}{T}\Res_{\pp_{+p}=\infty})\\ 
    &\qquad\qquad 
    S\bracket*{u_3}{\cfbseqb\cdots}{u_2}
    z_{+1}^{\frac{r_1}{T}+m_1}\cdots z_{+p}^{\frac{r_p}{T}+m_p}
    \d{z_{+p}}\cdots\d{z_{+1}}.
  \end{aligned}
\end{equation*}
Define $\Rad^{\op}(U^3)$ as the intersection of all $\Rad^{\op}(S)$, where 
\begin{equation}
  \begin{aligned}
    \Rad^{\op}(S):=\Set*{v'_3\in M^{\op}(U^3)\given S\bracket*{v'_3}{\cdots}{v_2}=0\text{ for all }v_2\in\widetilde{M}(U^2)}. 
  \end{aligned}
\end{equation}
Then $\widetilde{M}^{\op}(U^3):=M^{\op}(U^3)/\Rad^{\op}(U^3)$ is a right $\L_g(V)$-module, and so a left $\L_{g^{-1}}(V)$-module via the pushout along $\theta\colon\L_{g}(V)\to\L_{g^{-1}}(V)$. Namely, $\lo{a}{\frac{m}{T}} v_3:=v_3\cdot\theta(\lo{a}{\frac{m}{T}})$, where the anti-isomorphism $\theta$ is defined in \labelcref{eq:def:theta}. 
Furthermore, $\widetilde{M}^{\op}(U^3)$ is an admissible $g^{-1}$-twisted module of conformal weight $h_3$ whose vertex operator is given by
\begin{equation}
  Y_{\widetilde{M}^{\op}(U^3)}(a,z)v'_3 = 
  \sum_{n\in \Z} \lo{a}{\frac{n}{T}}v'_3 z^{-\frac{n}{T}-1} = 
  \sum_{n\in \Z} v'_3\cdot\theta(\lo{a}{\frac{n}{T}}) z^{-\frac{n}{T}-1},
\end{equation}
where $a\in V$ and $v'_3\in\widetilde{M}^{\op}(U^3)$. 
Then $S$ is extended to $\widetilde{M}^{\op}(U^3)\otimes\fun{Sym}[V,\cdots,V,M^1]\otimes \widetilde{M}(U^2)$. 
The first factor can be further extended to \emph{generalized Verma module} $\overline{M}(U^3)$.

\begin{proposition}\label{prop:cor-extension}
    The resulting system of correlation functions $S$ satisfies the {twisted genus-zero property} associated to the datum $\Sigma_{1}(\widetilde{M}^{\op}(U^3), M^1, \widetilde{M}(U^2))$.
\end{proposition}
\begin{proof}
    It suffices to show the \emph{truncation property}.
    We first show that $S\bracket*{u_3}{(v,\qq)}{v_2}w^{n}$ is holomorphic at $\qq=0$ for all $u_3\in U^3$, when $n\ge \deg v + \deg v_2$. 
   Since $\widetilde{M}(U^2)$ is a $V$-module by \cref{prop3.8}, we may assume $v_2=\lo{a}{\frac{r}{T}+m}u_2$ for some homogeneous $a\in V^{r}$, $m\in\Z$, $u_2\in U^2$, and $\deg v_2\ge 0$. Then 
    \begin{align*}
      \MoveEqLeft
      S\bracket*{u_3}{(v,\qq)}{\lo{a}{\frac{r}{T}+m}u_2}w^n\\
      &=\tfrac{1}{T}\Res_{\pp=0}S\bracket*{u_3}{(a,\pp)(v,\qq)}{u_2}z^{\frac{r}{T}+m}w^n\d{z}\\
      &=\tfrac{1}{T}S\bracket*{u_3\cdot [a]}{(v,\qq)}{u_2}w^n
      \Res_{\pp=0}z^{\frac{r}{T}+m-\wt a}\d{z}\\
      &\phantom{=}  + 
      \tfrac{1}{T}\sum_{i\ge 0}S\bracket*{u_3}{(\vo{a}{i}v,\qq)}{u_2}w^n
      \Res_{\pp=0}F_{\wt a-1+\delta(r)+\tfrac{r}{T}, i}(\pp, \qq)z^{\frac{r}{T}+m}\d{z}\\
      \overset{\labelcref{lem:ExpOfFpq}}&{=}
      \tfrac{1}{T}S\bracket*{u_3\cdot [a]}{(v,\qq)}{u_2}w^n\Res_{\pp=0}z^{\frac{r}{T}+m-\wt a}\d{z}\\
      &\phantom{=}  - 
      \tfrac{1}{T}\sum_{i\ge 0}\binom{\frac{r}{T}+m}{i}S\bracket*{u_3}{(\vo{a}{i}v,\qq)}{u_2}w^{\frac{r}{T}+m+n-i}\\
      &=\tfrac{1}{T}\braket*{\varphi}{(u_3\cdot [a])\otimes v\otimes u_2}w^{n-\deg v}\Res_{\pp=0}z^{\frac{r}{T}+m-\wt a}\d{z}\\
      &\phantom{=}  - 
      \tfrac{1}{T}\sum_{i\ge 0}\binom{\frac{r}{T}+m}{i}\braket*{\varphi}{u_3\otimes(\vo{a}{i}v)\otimes u_2}w^{n-\deg v_2 -\deg v},
    \end{align*}
    which is holomorphic at $\pp=0$ if $n\ge\deg v+\deg v_2$.

    It remains to show $S\bracket*{v'_3}{(v,\qq)}{v_2}w^{n}$ is holomorphic at $\qq=0$ for all $v'_3\in\widetilde{M}^{\op}(U^3)$, when $n\ge \deg v + \deg v_2$.
    We may assume $v'_3=\theta(\lo{a}{\frac{r}{T}+m})u_3$ for some homogeneous $a\in V^{r}$, $m\in\Z$, $u_3\in U^3$, and $-\wt a+\frac{r}{T}+m+1\ge 0$. Then 
    \begin{align*}
      \MoveEqLeft
      S\bracket*{\theta(\lo{a}{\frac{r}{T}+m})u_3}{(v,\qq)}{v_2}w^n\\
      &=-\tfrac{1}{T}\Res_{\pp=\infty}S\bracket*{u_3}{(a,\pp)(v,\qq)}{v_2}z^{\frac{r}{T}+m}w^n\d{z}\\
      &=-\tfrac{1}{T}S\bracket*{u_3\cdot [a]}{(v,\qq)}{v_2}w^n
      \Res_{\pp=\infty}z^{\frac{r}{T}+m-\wt a}\d{z}\\
      &\phantom{=}  - 
      \tfrac{1}{T}\sum_{i\ge 0}S\bracket*{u_3}{(\vo{a}{i}v,\qq)}{v_2}w^n
      \Res_{\pp=\infty}F_{\wt a-1+\delta(r)+\tfrac{r}{T}, i}(\pp, \qq)z^{\frac{r}{T}+m}\d{z}\\
      \overset{\labelcref{lem:ExpOfFpq}}&{=}
      -\tfrac{1}{T}S\bracket*{u_3\cdot [a]}{(v,\qq)}{v_2}w^n\Res_{\pp=\infty}z^{\frac{r}{T}+m-\wt a}\d{z}\\
      &\phantom{=}  + 
      \tfrac{1}{T}\sum_{i\ge 0}\binom{\frac{r}{T}+m}{i}S\bracket*{u_3}{(\vo{a}{i}v,\qq)}{v_2}w^{\frac{r}{T}+m+n-i}.
    \end{align*}
  The first term is holomorphic at $\qq=0$ since $u_3\cdot [a]\in U^3$ and $n\ge\deg v+\deg v_2$. The second term is holomorphic at $\qq=0$ since $\deg \vo{a}{i}v+\deg v_2 = \wt a + \deg v -i -1 +\deg v_2 \le \frac{r}{T}+m + n - i$.
\end{proof}
By \cref{prop:cor-extension}, the extended system of $g$-twisted correlation functions $S\in \overline{M}(U^3)$ $\otimes\fun{Sym}[V,\cdots,V,M^1]\otimes \overline{M}(U^2)$ also satisfies the \emph{twisted genus-zero property} associated to the datum $\Sigma_{1}(\overline{M}(U^3), M^1, \overline{M}(U^2))$.

Now we have our main theorem in this section: 
\begin{theorem}\label{thm:CorBottom}
  Let $M^1$ be an admissible untwisted module of conformal weight $h_1$, and let $U^2$ (resp. $U^3$) be a left (resp. right) $A_g(V)$-module with $[\upomega]$ acting as $h_2$ (resp. $h_3$). Put $h=h_1+h_2-h_3$. 
  Then any system of $g$-twisted restricted correlation functions associated to the datum $\Sigma_{1}(U^3, M^1, U^2)$ can be extended to one associated to the datum $\Sigma_{1}(\widetilde{M}^{\op}(U^3), M^1, \widetilde{M}(U^2))$ and one associated to the datum $\Sigma_{1}(\overline{M}(U^3), M^1, \overline{M}(U^2))$. Moreover, we have 
  \begin{equation*}
      \Cor[\Sigma_{1}(U^3, M^1, U^2)]
      \cong\Cor[\Sigma_{1}(\widetilde{M}^{\op}(U^3), M^1, \widetilde{M}(U^2))]
      \cong\Cor[\Sigma_{1}(\overline{M}(U^3), M^1, \overline{M}(U^2))].
  \end{equation*}
\end{theorem}
\begin{proof}
    It follows from \cref{prop:cor-restriction} and \cref{prop:cor-extension}.
\end{proof}
\begin{corollary}\label{coro:isomorphism-between-correlations}
    Let $M^2$ and $M^3$ be admissible $g$-twisted $V$-modules of conformal weight $h_2$ and $h_3$, and assume $M^2$ and $(M^3)'$ are generalized Verma modules. 
    Then
    \[\Cor[\Sigma_{1}(M^3(0)^\ast, M^1, M^2(0))]\cong\Cor\fusion\cong\Fusion.\]
\end{corollary}
\begin{proof}
    It follows from \cref{thm:CorBottom} and \cref{thm:I=Cor}.
\end{proof}

\section{Reconstructing \texorpdfstring{$g$}{g}-twisted restricted correlation functions}\label{sec4}

\subsection{\texorpdfstring{$g$}{g}-twisted restricted conformal blocks (first definition)}
We introduce the following notion of space of restricted coinvariants and conformal blocks: 
\begin{definition}\label{def:resCfb}
  Let $M^1$ be an admissible untwisted module of conformal weight $h_1$, and $U^2$ (resp. $U^3$) a left (resp. right) $A_g(V)$-module on which $[\upomega]$ acts as $h_2 \id$ (resp. $h_3 \id$). Put $h=h_1+h_2-h_3$. 
  Let $J$ be the subspace of $U^3\otimes M^1\otimes U^2$ spanned by the elements 
 \begin{align}
    &u_3\otimes (\vo{L}{-1}+\vo{L}{0}-h_1+h)v\otimes u_2,\label{f-relation1}\\
    &u_3\cdot [a]\otimes v\otimes u_2-\sum_{j\ge 0}\binom{\wt a}{j} u_3\otimes \vo{a}{j-1}v\otimes u_2,\quad a\in V^0,\label{f-relation2}\\
    &u_3\otimes v\otimes [a]\cdot u_2-\sum_{j\ge 0}\binom{\wt a-1}{j}u_3\otimes \vo{a}{j-1}v\otimes u_2,\quad a\in V^0,\label{f-relation3}\\
    &\sum_{j\ge 0} \binom{\wt a-1+\frac{r}{T}}{j} u_3\otimes \vo{a}{j-1}v\otimes u_2,\quad a\in V^r, r\neq 0,\label{f-relation4}
  \end{align}
  where $u_3\in U^3$, $v\in M^1$, and $u_2\in U^2$. 
  We call the quotient space $(U^3\otimes M^1\otimes U^2)/J$ the \concept{space of $g$-twisted restricted coinvariants} and a linear functional $\varphi$ vanishing on $J$ a \concept{$g$-twisted restricted conformal block} associated to $U^3, M^1,$ and $ U^2$. We denote the vector space of $g$-twisted restricted conformal blocks by $\Cfb[U^3, M^1, U^2]$.
\end{definition}
\begin{remark}
The relations \labelcref{f-relation1,f-relation2,f-relation3,f-relation4} are obtained from our later calculation of the twisted correlation functions. In fact, these relations are also compatible with the definitions of the usual space of (twisted) coinvariants and conformal blocks of VOAs associated to the datum $(\mathbb{P}^1_{\C}, 0,1,\infty, M^2, M^1,M^3)$ in \cite{Z94,FBZ04,FS04,NT}. We can obtain relations \labelcref{f-relation1,f-relation2,f-relation3,f-relation4} by restricting $M^2$ and $M^3$ to their bottom levels. We will give a more general definition of \emph{twisted (restricted) conformal block} and discuss it in more detail in a subsequent paper \cite{PartII}. 
\end{remark}

\begin{remark}
Observe that $\sum_{j\ge 0}\binom{\wt a}{j} \vo{a}{j-1}v=a\ast v$ and $\sum_{j\ge 0} \binom{\wt a-1}{j} \vo{a}{j-1}v=v\ast a$ for $a\in V^0$ and $v\in M^1$, where $a\ast v$ and $v\ast a$ are the $A(V^0)$-bimodule actions defined in \cite{FZ}. 
Later on, we will show that the vector space $\Cfb[U^3, M^1, U^2]$ is indeed dual to $U^3\otimes_{A_g(V)} B_{g, \lambda}(M^1)\otimes_{A_g(V)}U^2$, where $B_{g, \lambda}(M^1)$ is a quotient of $A_g(M^1)$ constructed in \cite{JJ} that generalizes $B_{\lambda}(M^1)$ in \cite{Liu} and $\lambda=h_2-h_3$. 
\end{remark}

\begin{proposition}\label{prop:rCor->rCfb}
  Any systme of correlation functions $S\in \Cor[\Sigma_{1}(U^3, M^1, U^2)]$ gives rise to a meromorphic family of linear functionas $\Set{\varphi_S(\qq)}_{\qq\in\Cc}$ in $\Cfb[U^3, M^1, U^2]$.
\end{proposition}
\begin{proof}
    Given any $S\in \Cor[\Sigma_{1}(U^3, M^1, U^2)]$, we define a meromorphic family of  linear functionas $\varphi_S(\qq)$ on $U^3\otimes M^1\otimes U^2$ by
    \[
    \varphi_S(\qq)\colon u_3\o v\o u_2\in U^3\otimes M^1\otimes U^2
    \longmapsto
    S\bracket*{u_3}{(w^{L(0)-h_1}v,\qq)}{u_2}.
    \]
    We show $\varphi_S(\qq)$ vanishes on $J$. Vanishing of $\varphi_S(\qq)$ on \labelcref{f-relation1} follows from the \emph{$\vo{L}{-1}$-derivative property} \labelcref{eq:L-derivative-w}. 
    For homogeneous $a\in V^r$ in $J$, we have 
    \begin{align*}
        &\sum_{j\ge 0}\binom{\wt a-1+\delta(r)+\frac{r}{T}}{j}\braket*{\varphi_S(\qq)}{u_3\o \vo{a}{j-1}v\o u_2}\\
        ={}&\sum_{j\ge 0}\binom{\wt a-1+\delta(r)+\frac{r}{T}}{j}S\bracket*{u_3}{(\vo{a}{j-1}v,\qq)}{u_2}w^{\deg v+\wt a-j}\\
        ={}&\sum_{j\ge 0}\binom{\wt a-1+\delta(r)+\frac{r}{T}}{j}\Res_{\pp=\qq}S\bracket*{u_3}{(a,\pp)(v,\qq)}{u_2}w^{\deg v+\wt a-j}(z-w)^{j-1}\d{z}\\
        ={}&\sum_{j\ge 0}\binom{\wt a-1+\delta(r)+\frac{r}{T}}{j}\Res_{\pp=\qq}S\bracket*{u_3\cdot[a]}{(v,\qq)}{u_2}w^{\deg v+\wt a-j}(z-w)^{j-1}z^{-\wt a}\d{z}\\
        &+\sum_{j\ge 0}\binom{\wt a-1+\delta(r)+\frac{r}{T}}{j}\Res_{\pp=\qq}\sum_{i\ge 0}F_{\wt a-1+\delta(r)+\frac{r}{T},i}S\bracket*{u_3}{(\vo{a}{i}v,\qq)}{u_2}\\
        &\phantom{=}  \cdot w^{\deg v+\wt a-j}(z-w)^{j-1}z^{-\wt a}\d{z}\\
        ={}&\braket*{\varphi_S(\qq)}{u_3\cdot [a]\o v\o u_2}\\
        &+\sum_{j\ge 0}\sum_{i\ge 0}\sum_{l=0}^{i}\binom{\wt a-1+\delta(r)+\frac{r}{T}}{j}\binom{\wt a-1+\delta(r)+\frac{r}{T}}{i-l}\binom{-\wt a+1-\delta(r)-\frac{r}{T}}{l+1-j}\\
        &\phantom{=}  \cdot \braket*{\varphi_S(\qq)}{u_3\o\vo{a}{i}v\o u_2}\\
        ={}&\braket*{\varphi_S(\qq)}{u_3\cdot [a]\o v\o u_2}.
    \end{align*}
    The last equality follows from the identity $\sum_{j\ge 0}\binom{\wt a-1+\delta(r)+\frac{r}{T}}{j}\binom{-\wt a+1-\delta(r)-\frac{r}{T}}{l+1-j}=0$ for $l\ge 0$. 
    This shows that $\varphi_S(\qq)$ vanishes on \labelcref{f-relation2} (resp. \labelcref{f-relation4}) when $r=0$ (resp. $r\neq 0$). The vanishing of $\varphi_S(\qq)$ on \labelcref{f-relation3} can be proved by a similar method using the other recursive formula. 
    Hence $\varphi_S(\qq)$ is in $\Cfb[U^3, M^1, U^2]$ for all $\qq\in \Cc$. 

    By the \emph{monomial property} of $S$, the family $\Set{\varphi_S(\qq)}_{\qq\in\Cc}$ is constant. In particular, it is meromorphic.
\end{proof}

\begin{theorem}\label{thm:iso-restrictcfb-bottomcorrelation}
    $\Cfb[U^3, M^1, U^2]\cong \Cor[\Sigma_{1}(U^3, M^1, U^2)].$
\end{theorem}
\begin{proof}
    Given any $S\in \Cor[\Sigma_{1}(U^3, M^1, U^2)]$, by \cref{prop:rCor->rCfb}, we have a constant family $\Set{\varphi_S(\qq)}_{\qq\in\Cc}$ of elements of $\Cfb[U^3, M^1, U^2]$.

    The rest of \cref{sec4} and the whole \cref{sec5} are dedicated to proving the opposite direction, i.e., given any $g$-twisted restricted conformal block $\varphi$, one can reconstruct a system of $g$-twisted correlation functions $S_\varphi$, such that $\varphi_{S_\varphi}=\varphi$ and $S_{\varphi_S}=S$.
\end{proof}
\subsection{Constructions of 3-point, 4-point, and 5-point functions}
In this subsection, we construct the $g$-twisted $3$-point, $4$-point, and $5$-point functions based on a given restricted conformal block $\varphi\in \Cfb[U^3, M^1, U^2]$. The general $(n+3)$-point functions can be inductively defined using the recursive formulas. 
\begin{construction}\label{def:3-point}
  Define the \emph{$3$-point function} $S_M\colon U^3\otimes M^{1}\otimes U^2\to \C[w^{\pm\frac{1}{T}}]$ by the formula:   
    \begin{equation}\label{eq:def:3-point}
    u_3\otimes v\otimes u_2 \longmapsto S_{M}\bracket*{u_3}{(v,\qq)}{u_2}:=\braket*{\varphi}{u_3\otimes w^{h_1-L(0)}v\otimes u_2}.
    \end{equation}
\end{construction}

Next, we can define the $4$-point functions $S_{VM}^{L}: U^3\otimes V\otimes M^{1}\otimes U^2\rightarrow \O(\infty\Delta_1)$ and $S_{MV}^{R}: U^3\otimes M^{1}\otimes V\otimes U^2\rightarrow \O(\infty\Delta_1)$ as follows:
\begin{construction}\label{def:4-point}
    For homogeneous $a\in V^r$, $v\in M^1$, $u_3\in U^3$ and $u_2\in U^2$, define the 4-point functions by
    \begin{align*}
   S_{VM}^{L}\bracket*{u_3}{(a,\pp)(v,\qq)}{u_2}:&=S_{M}\bracket*{u_3\cdot [a]}{(v,\qq)}{u_2}z^{-\wt a} \\
   &+\sum_{i\ge 0}F_{\wt a-1+\delta(r)+\tfrac{r}{T}, i} (\pp,\qq)S_{M}\bracket*{u_3}{(\vo{a}{i}v,\qq)}{u_2}, \numberthis\label{4.7}\\
   S_{MV}^{R}\bracket*{u_3}{(v,\qq)(a,\pp)}{u_2}:&=S_{M}\bracket*{u_3}{(v,\qq)}{[a]\cdot u_2}z^{-\wt a}\\
   &+\sum_{i\ge 0}F_{\wt a-1+\tfrac{r}{T}, i} (\pp,\qq)S_{M}\bracket*{u_3}{(\vo{a}{i}v,\qq)}{u_2}.\numberthis\label{4.8}
\end{align*}
\end{construction}

Before we move on to construct $5$-point functions,  we first prove the following lemma which states that the $4$-point functions we constructed satisfy the \emph{locality}:
\begin{lemma}\label{lem:locality:4-point}
    $S_{VM}^{L}\bracket*{u_3}{(a,\pp)(v,\qq)}{u_2}=S_{MV}^{R}\bracket*{u_3}{(v,\qq)(a,\pp)}{u_2}$.
\end{lemma}
\begin{proof}
    If $r\neq 0$, then $[a]=0$ and $\delta(r)=0$. Hence $S_{VM}^{L}\bracket*{u_3}{(a,\pp)(v,\qq)}{u_2}$ agrees with $S_{MV}^{R}\bracket*{u_3}{(a,\pp)(v,\qq)}{u_2}$. 
    Now suppose $a\in V^0$. Recalling \cref{lem:IndFpq}, 
    \[
      F_{\wt a-1,i}(\pp,\qq) - F_{\wt a,i}(\pp,\qq) = \binom{\wt a-1}{i}z^{-\wt a}w^{\wt a-i-1},
    \]
    we thus have
    \begin{align*}
        \MoveEqLeft
        S_{VM}^{L}\bracket*{u_3}{(a,\pp)(v,\qq)}{u_2}-S_{MV}^{R}\bracket*{u_3}{(v,\qq)(a,\pp)}{u_2}\\
        ={}&S_{M}\bracket*{u_3\cdot[a]}{(v,\qq)}{u_2}z^{-\wt a}-S_{M}\bracket*{u_3}{(v,\qq)}{[a]\cdot u_2}z^{-\wt a}\\
        &+\sum_{i\ge 0}\left(F_{\wt a-1, i} (\pp,\qq)-F_{\wt a, i} (\pp,\qq)\right)S_{M}\bracket*{u_3}{(\vo{a}{i}v,\qq)}{u_2}\\
        ={}&\braket*{\varphi}{u_3\cdot[a]\otimes v\otimes u_2}z^{-\wt a}w^{-\deg v}-\braket*{\varphi}{u_3\otimes v\otimes [a]\cdot u_2}z^{-\wt a}w^{-\deg v}\\
        &+\sum_{i \ge 0}\binom{\wt a-1}{i}z^{-\wt a}w^{\deg v}\braket*{\varphi}{u_3\otimes \vo{a}{i}v\otimes u_2}\\
        ={}&0.
        \end{align*}
    The last equality follows from \labelcref{f-relation2} and \labelcref{f-relation3}.
\end{proof}

\begin{construction}\label{def:5-point}
    For $5$-point functions, we define $S_{VVM}^{L}$, $S_{VMV}^{L}$ by 
\begin{align*}
    \MoveEqLeft
    S_{VVM}^{L}\bracket*{u_3}{(a^1,\pp_1)(a^2,\pp_2)(v,\qq)}{u_2}=S_{VMV}^{L}\bracket*{u_3}{(a^1,\pp_1)(v,\qq)(a^2, \pp_2)}{u_2}\\
    &:=S\bracket*{u_3\cdot [a^1]}{(a^2, \pp_2)(v,\qq)}{u_2}z^{-\wt a^1} \\
    &+\sum_{i\ge 0}F_{\wt a^1-1+\delta(r)+\tfrac{r}{T},i} (\pp_1,\pp_2)S\bracket*{u_3}{(\vo{a^1}{i}a^2,\pp_2)(v,\qq)}{u_2}\numberthis\label{ExpansionFromLeft}\\
    &+\sum_{i\ge 0}F_{\wt a^1-1+\delta(r)+\tfrac{r}{T},i} (\pp_1,\qq)S\bracket*{u_3}{(a^2, \pp_2)(\vo{a^1}{i}v,\qq)}{u_2}, 
\end{align*}
and $S_{MVV}^{R}$, $S_{VMV}^{R}$ by
\begin{align*}
   \MoveEqLeft
    S_{VMV}^{R}\bracket*{u_3}{(a^2,\pp_2)(v,\qq)(a^1, \pp_1)}{u_2}=S_{MVV}^{R}\bracket*{u_3}{(v,\qq)(a^2,\pp_2)(a^1,\pp_1)}{u_2}\\
    &:=S\bracket*{u_3}{(a^2, \pp_2)(v,\qq)}{[a^1]\cdot u_2}z^{-\wt a^1} \\
    &+\sum_{i\ge 0}F_{\wt a^1-1+\tfrac{r}{T},i} (\pp_1,\pp_2)S\bracket*{u_3}{(\vo{a^1}{i}a^2,\pp_2)(v,\qq)}{u_2}\numberthis\label{ExpansionFromRight}\\
    &+\sum_{i\ge 0}F_{\wt a^1-1+\tfrac{r}{T},i} (\pp_1,\qq)S\bracket*{u_3}{(a^2, \pp_2)(\vo{a^1}{i}v,\qq)}{u_2}, 
\end{align*}
where $a^1\in V^r$, $a^2\in V^s$ are homogeneous, $v\in M^1$, $u_3\in U^3$, $u_2\in U^2$, and $S$ is the $4$-point function defined in \cref{def:4-point}.
\end{construction}
For the well-definedness of the 5-point functions, we need to show that 
\begin{equation}\label{Well-definedness}
    S_{VMV}^{L}\bracket*{u_3}{(a^1,\pp_1)(v,\qq)(a^2, \pp_2)}{u_2}=S_{VMV}^{R}\bracket*{u_3}{(a^1,\pp_1)(v,\qq)(a^2, \pp_2)}{u_2},
\end{equation}
and for the proof of \emph{locality} of the 5-point functions, we need to show that 
\begin{align}
    &S_{VMV}^{L}\bracket*{u_3}{(a^1,\pp_1)(v,\qq)(a^2, \pp_2)}{u_2}=S_{VMV}^{R}\bracket*{u_3}{(a^2,\pp_2)(v,\qq)(a^1, \pp_1)}{u_2},\numberthis\label{com1}\\
    &S_{VVM}^{L}\bracket*{u_3}{(a^1,\pp_1)(a^2, \pp_2)(v,\qq)}{u_2}=S_{VVM}^{L}\bracket*{u_3}{(a^2,\pp_2)(a^1, \pp_1)(v,\qq)}{u_2},\numberthis\label{com2}\\
    &S_{MVV}^{R}\bracket*{u_3}{(v,\qq)(a^2,\pp_2)(a^1, \pp_1)}{u_2}=S_{MVV}^{R}\bracket*{u_3}{(v,\qq)(a^1,\pp_1)(a^2, \pp_2)}{u_2}\numberthis\label{com3}.
\end{align}
\begin{proposition}\label{well-definedness-com}
    Assume \labelcref{com1} and \labelcref{com2} hold, then \labelcref{Well-definedness} holds. Assume \labelcref{com1} and \labelcref{Well-definedness} hold, then \labelcref{com3} holds.
\end{proposition}
\begin{proof}
    The proof of the first part is similar to \cite[Proposition 4.12]{Liu}.
    Now assume \labelcref{com1} and \labelcref{Well-definedness} hold, then 
    \begin{align*}
       \MoveEqLeft[-1.4] S_{MVV}^{R}\bracket*{u_3}{(v,\qq)(a^1,\pp_1)(a^2, \pp_2)}{u_2}=S_{VMV}^{R}\bracket*{u_3}{(a^1,\pp_1)(v,\qq)(a^2, \pp_2)}{u_2}\\
       &=S_{VMV}^{L}\bracket*{u_3}{(a^2,\pp_2)(v,\qq)(a^1, \pp_1)}{u_2}=S_{VMV}^{R}\bracket*{u_3}{(a^2,\pp_2)(v,\qq)(a^1, \pp_1)}{u_2}\\
       &=S_{MVV}^{R}\bracket*{u_3}{(v,\qq)(a^2,\pp_2)(a^1, \pp_1)}{u_2},
    \end{align*}
    where the first and last equality follow from \labelcref{ExpansionFromRight}, the second equality follows from \labelcref{com1}, and the third equality follows from \labelcref{Well-definedness}. Thus \labelcref{com3} holds.
\end{proof}
By \cref{well-definedness-com}, to show $S^L_{VVM}, S^L_{VMV},S^R_{MVV},$ and $S^R_{VMV}$ above satisfies \emph{locality}, it suffices to show \labelcref{com1} and \labelcref{com2} hold. 
\begin{lemma}\label{lem:com}
    For homogeneous $a^1\in V^r$, $a^2\in V^s$, $v\in M^1$, $u_2\in U^2$, and $u_3\in U^3$, \labelcref{com1} and \labelcref{com2} hold.
\end{lemma}
\begin{proof}
    The proof will be given at the end of this subsection.
\end{proof}


Next we show the \emph{$\vo{L}{-1}$-derivative property} of the twisted $3$-point, $4$-point and $5$-point functions. 
\begin{proposition}\label{Pro:L(-1)-derivative}
  For $a^1, a^2\in V$, we have: 
  \begin{align}
  S\bracket*{u_3}{(\vo{L}{-1}v,\qq)}{u_2}w^{-h}&=\pdv {w} \left(S\bracket*{u_3}{(v,\qq)}{u_2}w^{-h}\right)\label{3point-L(-1)v}\\
  S\bracket*{u_3}{(\vo{L}{-1}a^1,\pp_1)(v,\qq)}{u_2}&=\pdv{z_1}S\bracket*{u_3}{(a^1,\pp_1)(v,\qq)}{u_2}, \label{4point-L(-1)a}\\
    S\bracket*{u_3}{(a^1,\pp_1)(\vo{L}{-1}v,\qq)}{u_2}w^{-h}&=\pdv {w} \left(S\bracket*{u_3}{(a^1,\pp_1)(v,\qq)}{u_2}w^{-h}\right) ,\label{4point-L(-1)v}\\
  S\bracket*{u_3}{(\vo{L}{-1}a^1,\pp_1)(a^2,\pp_2)(v,\qq)}{u_2}&=\pdv{z_1} S\bracket*{u_3}{(a^1,\pp_1)(a^2,\pp_2)(v,\qq)}{u_2},\label{5point-L(-1)a}\\ 
  S\bracket*{u_3}{(a^1,\pp_1)(a^2,\pp_2)(\vo{L}{-1}v,\qq)}{u_2}w^{-h}&=\pdv{w} \left(S\bracket*{u_3}{(a^1,\pp_1)(a^2,\pp_2)(v,\qq)}{u_2}w^{-h} \right). \label{5point-L(-1)v}
  \end{align} 
  \end{proposition}
\begin{proof}
We first show \labelcref{5point-L(-1)a}. By \labelcref{Fpq}, It is straightforward to show that
\begin{equation}\label{derF}
\pdv{w}F_{n, i}(\pp_1,\qq)=(i+1) F_{n,i+1}(\pp_1,\qq), \quad \pdv{z_1} F_{n,i}(\pp_1,\qq)= -(i+1)F_{n+1,i+1}(\pp_1,\qq). 
\end{equation}
Note that $[\vo{L}{-1}a+\vo{L}{0}a]=0$, and $\vo{(\vo{L}{-1}a)}{i}=-i\vo{a}{i-1}$ for $a\in V$ or $a\in M^1$. Suppose $a^1\in V^r$, we have
\begin{align*}
    \MoveEqLeft
    S\bracket*{u_3}{(\vo{L}{-1}a^1,\pp_1)(a^2,\pp_2)(v,\qq)}{u_2}\\
    ={}&S\bracket*{u_3}{(a^2,\pp_2)(v,\qq)}{[\vo{L}{-1}a^1]\cdot u_2}z_1^{-\wt a^1-1}\\
    &+\sum_{i\ge 0}F_{\wt a^1+\tfrac{r}{T},i}(\pp_1, \qq)S\bracket*{u_3}{(\vo{(\vo{L}{-1}a^1)}{i}a^2, \pp_2)(v, \qq)}{u_2}\\
    &+\sum_{i\ge 0}F_{\wt a^1+\tfrac{r}{T},i}(\pp_1, \qq)S\bracket*{u_3}{(a^2, \pp_2)(\vo{(\vo{L}{-1}a^1)}{i}v, \qq)}{u_2}\\
    ={}&-\wt a^1S\bracket*{u_3}{(a^2,\pp_2)(v,\qq)}{[a^1]\cdot u_2}z_1^{-\wt a^1-1}\\
    &-\sum_{i\ge 0}iF_{\wt a^1+\tfrac{r}{T},i}(\pp_1, \qq)S\bracket*{u_3}{(\vo{a^1}{i-1}a^2, \pp_2)(v, \qq)}{u_2}\\
    &-\sum_{i\ge 0}iF_{\wt a^1+\tfrac{r}{T},i}(\pp_1, \qq)S\bracket*{u_3}{(a^2, \pp_2)(\vo{a^1}{i-1}v, \qq)}{u_2}\\
    ={}&\pdv{z_1}S\bracket*{u_3}{(a^1,\pp_1)(a^2,\pp_2)(v,\qq)}{u_2}.
\end{align*}
Thus \labelcref{5point-L(-1)a} holds. It's easy to see that the $4$-point and $5$-point functions we defined satisfy the \emph{vacuum property} in \cref{def:genus-zero-bottom}. Letting $a^2=\vac$ in \labelcref{5point-L(-1)a}, we get \labelcref{4point-L(-1)a}.

For \labelcref{3point-L(-1)v}, note that $\braket*{\varphi}{u_3\otimes \vo{L}{-1}v\otimes u_2}=-\braket*{\varphi}{u_3\otimes (\deg v+h)v\otimes u_2},$ hence
\begin{align*}
  &  S\bracket*{u_3}{(\vo{L}{-1}v,\qq)}{u_2}w^{-h}=\braket*{\varphi}{u_3\otimes \vo{L}{-1}v\otimes u_2}w^{-\deg v-1-h}\\
    &=-\braket*{\varphi}{u_3\otimes (\deg v+h)v\otimes u_2}w^{-\deg v-1-h}=\pdv{w}\left(S\bracket*{u_3}{(v,\qq)}{u_2}w^{-h}\right).
\end{align*}

For \labelcref{4point-L(-1)v}, note that $\vo{a^1}{i}\vo{L}{-1}v=\vo{L}{-1}\vo{a^1}{i}v+i\vo{a^1}{i-1}v.$ Therefore, 
\begin{align*}
    \MoveEqLeft
    S\bracket*{u_3}{(a^1,\pp_1)(\vo{L}{-1}v,\qq)}{u_2}w^{-h}=S\bracket*{u_3}{(\vo{L}{-1}v,\qq)}{[a^1]\cdot u_2}w^{-h}\\
    &+\sum_{i\ge 0}F_{\wt a^1-1+\tfrac{r}{T},i}(\pp_1, \qq)S\bracket*{u_3}{(\vo{a^1}{i}\vo{L}{-1}v, \qq)}{u_2}w^{-h}\\
    ={}&\pdv{w}\left(S\bracket*{u_3}{(a^1,\pp_1)(v,\qq)}{u_2}w^{-h}\right)\\
    &+\sum_{i\ge 0}F_{\wt a^1-1+\tfrac{r}{T},i}(\pp_1, \qq)\pdv{w}\left(S\bracket*{u_3}{(\vo{a^1}{i}v, \qq)}{u_2}w^{-h}\right)\\
    &+\sum_{i\ge 0}\pdv{w}\left(F_{\wt a^1-1+\tfrac{r}{T},i}(\pp_1, \qq)\right)S\bracket*{u_3}{(\vo{a^1}{i}v, \qq)}{u_2}w^{-h}\\
    ={}&\pdv{w}\left(S\bracket*{u_3}{(a^1,\pp_1)(v,\qq)}{u_2}w^{-h}\right).
\end{align*}
\labelcref{5point-L(-1)v} can be proved in a similar way.
\end{proof}

We conclude this subsection by giving a proof of \cref{lem:com}. 
\begin{proof}[Proof of \cref{lem:com}]
    We first show \labelcref{com1}. Suppose $a^1\in V^r$ and $ a^2\in V^s$. If $r=s=0$, the proof follows the same suit as the argument in Section 4.2 in \cite{Liu}. Note that the only property of $\varphi$ we need to use to prove \labelcref{com1} is the equality
    \[\braket*{\varphi}{u_3\cdot[a]\o v\o u_2}-\braket*{\varphi}{u_3\o v\o [a]\cdot u_2}=\sum_{j\ge 0} \binom{\wt a-1}{j} \braket*{\varphi}{u_3\o \vo{a}{j-1}v\o u_2},\]
    which, in our case, follows from \labelcref{f-relation2}
 and \labelcref{f-relation3}.
 
    If $r\neq 0$, then \labelcref{com1} holds by \labelcref{ExpansionFromLeft} and \labelcref{ExpansionFromRight}. So it suffices to deal with the case where $r=0$ and $s\neq 0$. 
    Similar to Lemma 4.13 in \cite{Liu}, we have the following formula on module $M^1$: 
    \begin{equation}\label{4.21}
        \begin{aligned}
        &\sum_{i,j\ge 0} \binom{\wt a^1-1}{j} \binom{\tfrac{s}{T}+\wt a^2-1+n}{i} \left(\vo{a^1}{j}\vo{a^2}{i}-\vo{a^2}{i}\vo{a^1}{j}\right)v\\
        &= \sum_{i,j\ge 0} \binom{\wt a^1-1}{j} \binom{\tfrac{s}{T}+\wt a^1-j-2+\wt a^2+n}{i} \vo{(\vo{a^1}{j}a^2)}{i}v, 
        \end{aligned}
    \end{equation}
    where $a^1\in V^0,a^2\in V^s$, and $n\in \N$. For $u_3\in U^3$, $u_2\in U^2$, and $v\in M^1$, we write 
    \begin{align}
        A&:=S\bracket*{u_3\cdot [a^{1}]}{(v,\qq)(a^{2},\pp_{2})}{u_2}z_{1}^{-\wt a^{1}}-S\bracket*{u_3}{(a^{2},\pp_{2})(v,\qq)}{[a^{1}]\cdot u_2}z_{1}^{-\wt a^{1}},\label{A}\\
        B&:=\sum_{j\ge 0}(F_{\wt a^{1},j}(\pp_1,\qq)-F_{\wt a^{1}-1,j}(\pp_{1},\qq))S\bracket{u_3}{(\vo{a^{1}}{j}v,\qq)(a^{2},\pp_{2})}{u_2},\label{B}\\
        C&:=\sum_{j\ge 0}(F_{\wt a^{1},j}(\pp_1,\pp_2)-F_{\wt a^{1}-1,j}(\pp_{1},\pp_2))S\bracket{u_3}{(v,\qq)(\vo{a^{1}}{j}a^{2},\pp_{2})}{u_2}. \label{C}
    \end{align} 
    By \cref{lem:IndFpq}, \labelcref{f-relation2}, \labelcref{f-relation3}, \labelcref{eq:def:3-point}, \labelcref{ExpansionFromLeft}, and \labelcref{ExpansionFromRight}, we can derive the following expressions for $A$, $B$, and $C$:
   \begin{align*}
        A&=\sum_{i,j\ge 0}\binom{\wt a^1-1}{j}F_{\wt a^2-1+\tfrac{s}{T}, i}(\pp_2,\qq) \braket*{\varphi}{u_3\otimes \vo{a^1}{j}\vo{a^2}{i}v\otimes u_2} z_1^{-\wt a^1} w^{-\wt a^2+i+1-\deg v},\\
        B&=-\sum_{i,j\ge 0}\binom{\wt a^1-1}{j} F_{\wt a^2-1+\tfrac{s}{T},i}(\pp_2,\qq) 
        \braket*{\varphi}{u_3\otimes \vo{a^2}{i}\vo{a^1}{j}v\otimes u_2} z_1^{-\wt a^1} w^{-\wt a^2+i+1-\deg v},\\
        C&=-\sum_{j,i\ge 0}\binom{\wt a^1-1}{j} F_{\wt (\vo{a^1}{j}a^2)-1+\tfrac{s}{T},i}(\pp_2,\qq) 
        \braket*{\varphi}{u_3\otimes \vo{(\vo{a^1}{j}a^2)}{i}v\otimes u_2} z_1^{-\wt a^1}z_2^{\wt a^1-j-1}\\
        &\phantom{=}  \cdot w^{-\wt a^2-\wt a^1+j+1+i+1-\deg v}.
 \end{align*} 
  Then, by \labelcref{4.21}
  \begin{align*}
      \MoveEqLeft
      \iota_{\pp_2=\infty}(A+B+C)\\
      ={}&\sum_{i,j\ge 0}\sum_{n\ge 0} \binom{\wt a^1-1}{j} \binom{\tfrac{s}{T}+\wt a^2-1+n}{i} \braket*{\varphi}{u_3\otimes (\vo{a^1}{j}\vo{a^2}{i}v-\vo{a^2}{i}\vo{a^1}{j}v)\otimes u_2} \\
      &\phantom{=}  \cdot z_1^{-\wt a^1} z_2^{-\wt a^2-\tfrac{s}{T}-n} w^{\tfrac{s}{T}+n-\deg v}\\
     &-\sum_{i,j\ge 0}\sum_{n\ge 0} \binom{\wt a^1-1}{j} \binom{\tfrac{s}{T}+\wt a^1-j-2+\wt a^2+n}{i} \braket*{\varphi}{u_3\otimes \vo{(\vo{a^1}{j}a^2)}{i}v\otimes u_2} \\
     &\phantom{=}  \cdot z_1^{-\wt a^1} z_2^{-\wt a^2-\tfrac{s}{T}-n} w^{\tfrac{s}{T}+n-\deg v}\\
     ={}&0.
  \end{align*}
Thus $A+B+C=0$, and \labelcref{com1} holds.  

Now we show \labelcref{com2}. Again, the case $r=s=0$ has been dealt with in \cite{Liu}. It suffices to show that \labelcref{com2} holds for the following three cases: (1) $r=0, s\neq 0$, (2) $r\neq 0, s\neq 0,$ and $ r+s\neq T$, or (3) $r\neq 0, s\neq 0,$ and $ r+s=T$. 

Similar to (2.2.10) in \cite{Z}, we can rewrite the recursive formula \labelcref{4.7} as follows: 
   \begin{align*}
        &S\bracket{u_3}{(a,\pp)(v,\qq)}{u_2}=S\bracket*{u_3\cdot[a]}{(v,\qq)}u_2)z^{-\wt a}\\
        &+\Res_{x}\left(\frac{z^{-\wt a+1-\delta(r)-\frac{r}{T}}(w+x)^{\wt a-1+\delta(r)+\frac{r}{T}}}{z-w-x} S\bracket*{u_3}{(Y_{M^1}(a,x)v,\qq)}{u_2}\right), \numberthis\label{recursive-other-form}
   \end{align*} 
   for $a\in V^r$ homogeneous. 
   
   For the case where $r=0$ and $ s\neq 0$, by \cref{lem:locality:4-point} and \labelcref{recursive-other-form}, we can express the left-hand side of \labelcref{com2} as follows: 
\begin{align*}
    \MoveEqLeft
    S^L_{VVM}\bracket*{u_3}{(a^1,\pp_1)(a^2,\pp_2)(v,\qq)}{u_2}=\underbrace{S\bracket*{u_3\cdot[a^1]}{(v,\qq)(a^2,\pp_2)}{u_2}z_1^{-\wt a^1}}_{(D1)}\\
	&+\underbrace{\Res_{x_1} \left(\frac{z_1^{-\wt a^1}(w+x_1)^{\wt a^1}}{z_1-w-x_1}\right) S\bracket*{u_3}{(Y_{M^1}(a^1,x_1)v,\qq) (a^2,\pp_2)}{u_2)}}_{(D2)}\\
	&+\underbrace{\Res_{x_0} \left(\frac{z_1^{-\wt a^1}(z_2+x_0)^{\wt a^1}}{z_1-z_2-x_0}\right) S\bracket*{u_3}{(v,\qq) (Y(a^1,x_0)a^2,\pp_2)}{u_2}}_{(D3)}\\
	={}&(D1)+(D2)+(D3). 
\end{align*}
By \labelcref{4.8}, \labelcref{recursive-other-form}, and the Jacobi identity of $Y_{M^1}$, we can rewrite $(D1),(D2),$ and $(D3)$ as follows: 
\begin{align*}
    (D1)={}&\Res_{x_2} \left(\frac{z_2^{-\wt a^2+1-\frac{s}{T}}(w+x_2)^{\wt a^2+\frac{s}{T}-1}}{z_2-w-x_2}\right) S\bracket{u_3\cdot[a^1]}{(Y_M(a^2,x_2)v,\qq)}{u_2} z_1^{-\wt a^1}\\
    (D2)={}&\Res_{x_1}\Res_{x_2} \left(\frac{z_1^{-\wt a^1}(w+x_1)^{\wt a^1}}{z_1-w-x_1} \cdot \frac{z_2^{-\wt a^2+1-\frac{s}{T}}(w+x_2)^{\wt a^2+\frac{s}{T}-1}}{z_2-w-x_2}\right) \\
    &\qquad \cdot S\bracket*{u_3}{(Y_{M^1}(a^2,x_2)Y_{M^1}(a^1,x_1)v,\qq)}{u_2}\\
    (D3)={}&\Res_{x_0}\Res_{x_2} \sum_{n\in \Z} \frac{z_1^{-\wt a^1}(z_2+x_0)^{\wt a^1}}{z_1-z_2-x_0}\cdot \frac{(w+x_2)^{\wt a^1-n-1+\wt a^2+\frac{s}{T}-1}}{z_2-w-x_2}\\
    &\qquad \cdot z_2^{-\wt a^1+n+1-\wt a^2+1-\frac{s}{T}}x_0^{-n-1}S\bracket*{u_3}{(Y_{M^1}(\vo{a^1}{n}a^2,x_2)v,\qq)}{u_2}\\
    ={}&\Res_{x_1,x_2} \frac{z_2^{-\wt a^2+1-\frac{s}{T}+1}z_1^{-\wt a^1}(w+x_2)^{\wt a^2+\frac{s}{T}-1}(w+x_1)^{\wt a^1}}{(z_2-w-x_2)(z_1(w+x_2)-z_2(w+x_1))}\\
    &\qquad \cdot S\bracket*{u_3}{(Y_{M^1}(a^1,x_1)Y_{M^1}(a^2,x_2)v,\qq)}{u_2}\\
    -&\Res_{x_1,x_2} \frac{z_2^{-\wt a^2+1-\frac{s}{T}+1}z_1^{-\wt a^1}(w+x_2)^{\wt a^2+\frac{s}{T}-1}(w+x_1)^{\wt a^1}}{(z_2-w-x_2)(z_1(w+x_2)-z_2(w+x_1))}\\
    &\qquad \cdot S\bracket*{u_3}{(Y_{M^1}(a^2,x_2)Y_{M^1}(a^1,x_1)v,\qq)}{u_2}.
\end{align*}
On the other hand, by \labelcref{ExpansionFromLeft}, we can write the right-hand side of \labelcref{com2} as
\begin{align*}
    \MoveEqLeft
    S_{VVM}^{L}\bracket*{u_3}{(a^2,\pp_1)(a^1,\pp_2)(v,\qq)}{u_2}\\
    ={}&\underbrace{\Res_{x_2} \left(\frac{z_2^{-\wt a^2+1-\frac{s}{T}}(w+x_2)^{\wt a^2+\frac{s}{T}-1}}{z_2-w-x_2}\right) S\bracket*{u_3}{(Y_{M^1}(a^2,x_2)v,\qq)(a^1,\pp_1)}{u_2}}_{(E1)}\\
    &+\underbrace{\Res_{x_0} \left(\frac{z_2^{-\wt a^2+1-\frac{s}{T}}(z_1+x_0)^{\wt a^2+\frac{s}{T}-1}}{z_2-z_1-x_0}\right) S\bracket*{u_3}{(v,\qq)(Y(a^2,x_0)a^1,\pp_1)}{u_2}}_{(E2)}\\
    ={}&(E1)+(E2). 
\end{align*}
By \labelcref{4.7}, \labelcref{recursive-other-form}, and the Jacobi identity, we can rewrite $(E1)$ and $(E2)$ as
\begin{align*}
    (E1)={}&\Res_{x_2} \left(\frac{z_1^{-\wt a^1}z_2^{-\wt a^2+1-\frac{s}{T}}(w+x_2)^{\wt a^2+\frac{s}{T}-1}}{z_2-w-x_2}\right) S\bracket*{u_3\cdot[a^1]}{(Y_{M^1}(a^2,x_2)v,\qq)}{u_2}\\
    &+\Res_{x_2}\Res_{x_1} \left(\frac{z_2^{-\wt a^2+1-\frac{s}{T}}(w+x_2)^{\wt a^2+\frac{s}{T}-1}}{z_2-w-x_2} \cdot\frac{z_1^{-\wt a^1}(w+x_1)^{\wt a^1}}{z_1-w-x_1} \right)\\
    &\qquad \cdot S\bracket*{u_3}{(Y_{M^1}(a^1,x_1)Y_{M^1}(a^2,x_2)v,\qq)}{u_2},\\
    (E2)={}&\Res_{x_0} \Res_{x_2} \sum_{n\in \Z} \frac{z_2^{-\wt a^2+1-\frac{s}{T}}(z_1+x_0)^{\wt a^2+\frac{s}{T}-1}}{z_2-z_1-x_0}\cdot \frac{(w+x_2)^{\wt a^1-n-1+\wt a^2+\frac{s}{T}-1}}{z_1-w-x_2}\\
    &\qquad \cdot z_1^{-\wt a^1+n+1-\wt a^2+1-\frac{s}{T}}S\bracket*{u_3}{(Y_{M^1}(\vo{a^2}{n}a^1,x_2)v,w)}{u_2}x_0^{-n-1}\\
    ={}&\Res_{x_2,x_1}\frac{z_1^{-\wt a^1+1}z_2^{-\wt a^2+1-\frac{s}{T}}(w+x_1)^{\wt a^1}(w+x_2)^{\wt a^2+\frac{s}{T}-1}}{(z_1-w-x_1)(z_2(w+x_1)-z_1(w+x_2))}\\
    &\qquad \cdot S\bracket*{u_3}{(Y_{M^1}(a^2,x_2)Y_{M^1}(a^1,x_1)v,\qq)}{u_2}\\
    &-\Res_{x_2,x_1}  \frac{z_1^{-\wt a^1+1}z_2^{-\wt a^2+1-\frac{s}{T}}(w+x_1)^{\wt a^1}(w+x_2)^{\wt a^2+\frac{s}{T}-1}}{(z_1-w-x_1)(z_2(w+x_1)-z_1(w+x_2))}\\
    &\qquad \cdot S\bracket*{u_3}{(Y_{M^1}(a^1,x_1)Y_{M^1}(a^2,x_2)v,\qq)}{u_2}.
\end{align*}
Thus we have 	
\begin{align*}
    &(D1)+(D2)+(D3)-(E1)-(E2)\\
    ={}&\Res_{x_1,x_2} \frac{z_1^{-\wt a^1}z_2^{-\wt a^2+1-\frac{s}{T}}(w+x_1)^{\wt a^1}(w+x_2)^{\wt a^2+\frac{s}{T}-1}}{z_1(w+x_2)-z_2(w+x_1)}\\
    &\qquad\cdot \left(\frac{z_2}{z_2-w-x_2}-\frac{z_1}{z_1-w-x_1}\right)S\bracket*{u_3}{(Y_{M^1}(a^1,x_1)Y_{M^1}(a^2,x_2)v,\qq)}{u_2}\\
    -&\Res_{x_1,x_2} \frac{z_1^{-\wt a^1}z_2^{-\wt a^2+1-\frac{s}{T}}(w+x_1)^{\wt a^1}(w+x_2)^{\wt a^2+\frac{s}{T}-1}}{z_1(w+x_2)-z_2(w+x_1)}\\
    &\qquad \cdot \left(\frac{z_2}{z_2-w-x_2}-\frac{z_1}{z_1-w-x_1}\right)
    S\bracket*{u_3}{(Y_{M^1}(a^2,x_2)Y_{M^1}(a^1,x_1)v,\qq)}{u_2}\\
    +&\Res_{x_1}\Res_{x_2} \left(\frac{z_1^{-\wt a^1}(w+x_1)^{\wt a^1}}{z_1-w-x_1} \cdot \frac{z_2^{-\wt a^2+1-\frac{s}{T}}(w+x_2)^{\wt a^2+\frac{s}{T}-1}}{z_2-w-x_2}\right)\\
    &\qquad \cdot S\bracket*{u_3}{(Y_{M^1}(a^2,x_2)Y_{M^1}(a^1,x_1)v,\qq)}{u_2}\\
    -&\Res_{x_2}\Res_{x_1} \left(\frac{z_2^{-\wt a^2+1-\frac{s}{T}}(w+x_2)^{\wt a^2+\frac{s}{T}-1}}{z_2-w-x_2} \cdot\frac{z_1^{-\wt a^1}(w+x_1)^{\wt a^1}}{z_1-w-x_1} \right)\\
    &\qquad \cdot S\bracket*{u_3}{(Y_{M^1}(a^1,x_1)Y_{M^1}(a^2,x_2)v,\qq)}{u_2}\\
    ={}&0.
\end{align*}
The proof of \labelcref{com2} for the case when $r\neq 0, s\neq 0$, and $r+s\neq T$ is similar to the case when $r=0$ and $ s\neq 0$, we omit it. 

Now, we show \labelcref{com2} for the case $r\neq 0, s\neq 0$, and $ r+s=T$. By adopting a similar computation as above, using the fact that $r/T=1-s/T$, we can express the left-hand side of \labelcref{com2} as follows: 
\begin{align*}
    \MoveEqLeft
    S_{VVM}^{L}\bracket*{u_3}{(a^1,\pp_1)(a^2,\pp_2)(v,\qq)}{u_2}\\
    ={}&\Res_{x_1} \left(\frac{z_1^{-\wt a^1+1-\frac{r}{T}}z_2^{-\wt a^2-\frac{s}{T}+1}(1+x_1)^{\wt a^1+\frac{r}{T}-1}}{z_1-z_2-z_2x_1}\right)S\bracket*{u_3\cdot[Y(a^1,x_1)a^2]}{(v,\qq)}{u_2}\\
    &\quad +\Res_{x_1,x_2} \frac{z_1^{-\wt a^1+1-\frac{r}{T}}(w+x_1)^{\wt a^1+\frac{r}{T}-1} z_2^{-\wt a^2-\frac{s}{T}+1}(w+x_2)^{\frac{s}{T}+\wt a^2} }{((w+x_2)z_1-(w+x_1)z_2)(z_2-w-x_2)}\\ &\qquad \cdot S\bracket*{u_3}{(Y_{M^1}(a^1,x_1)Y_{M^1}(a^2,x_2)v,\qq)}{u_2}\\
    &\quad -\Res_{x_1,x_2} \frac{z_1^{-\wt a^1+1-\frac{r}{T}}(w+x_1)^{\wt a^1+\frac{r}{T}-1} z_2^{-\wt a^2-\frac{s}{T}+1}(w+x_2)^{\frac{s}{T}+\wt a^2} }{((w+x_2)z_1-(w+x_1)z_2)(z_2-w-x_2)}\\  &\qquad \cdot S\bracket*{u_3}{(Y_{M^1}(a^2,x_2)Y_{M^1}(a^1,x_1)v,\qq)}{u_2}\\
    &\quad +\Res_{x_2}\Res_{x_1} \left(\frac{z_1^{-\wt a^1+1-\frac{r}{T}}(w+x_1)^{\wt a^1+\frac{r}{T}-1}}{z_1-w-x_1}\cdot \frac{z_2^{-\wt a^2+1-\frac{s}{T}}(w+x_2)^{\wt a^2+\frac{s}{T}-1}}{z_2-w-x_2}\right) \\
    &\qquad \cdot S\bracket*{u_3}{(Y_{M^1}(a^2,x_2)Y_{M^1}(a^1,x_1)v,\qq)}{u_2}\\
    ={}&(F1)+(F2)+(F3)+(F4),
\end{align*}		
where $(F1)-(F4)$ are the corresponding terms on the right-hand side. On the other hand, we can express the right-hand side of \labelcref{com2} as follows: 
\begin{align*}
    \MoveEqLeft
    S_{VVM}^{L}\bracket*{u_3}{(a^2,\pp_2)(a^1,\pp_1)(v,\qq)}{u_2}\\
    ={}&\Res_{x_2} \left(\frac{z_2^{-\wt a^2+1-\frac{s}{T}}z_1^{-\wt a^1-\frac{r}{T}+1}(1+x_2)^{\wt a^2+\frac{s}{T}-1}}{z_2-z_1-z_1x_2}\right) S\bracket*{u_3\cdot[Y(a^2,x_2)a^1]}{(v,\qq)}{u_2}\\
    &+ \Res_{x_2,x_1} \frac{z_2^{-\wt a^2+1-\frac{s}{T}}(w+x_1)^{\wt a^1+\frac{r}{T}} z_1^{-\wt a^1-\frac{r}{T}+1}(w+x_2)^{\wt a^2+\frac{s}{T}-1} }{((w+x_1)z_2-(w+x_2)z_1)(z_1-w-x_1)}\\
    &\qquad \cdot S\bracket*{u_3}{(Y_{M^1}(a^2,x_2)Y_{M^1}(a^1,x_1)v,\qq)}{u_2}\\
    &-\Res_{x_2,x_1} \frac{z_2^{-\wt a^2+1-\frac{s}{T}}(w+x_1)^{\wt a^1+\frac{r}{T}} z_1^{-\wt a^1-\frac{r}{T}+1}(w+x_2)^{\wt a^2+\frac{s}{T}-1} }{((w+x_1)z_2-(w+x_2)z_1)(z_1-w-x_1)} \\
    &\qquad \cdot S\bracket*{u_3}{(Y_{M^1}(a^1,x_1)Y_{M^1}(a^2,x_2)v,\qq)}{u_2}\\
    &+ \Res_{x_1}\Res_{x_2}\left(\frac{z_2^{-\wt a^2+1-\frac{s}{T}}(w+x_2)^{\wt a^2+\frac{s}{T}-1}}{z_2-w-x_2}\cdot \frac{z_1^{-\wt a^1+1-\frac{r}{T}}(w+x_1)^{\wt a^1+\frac{r}{T}-1}}{z_1-w-x_1}\right) \\
    &\qquad \cdot S\bracket*{u_3}{(Y_{M^1}(a^1,x_1)Y_{M^1}(a^2,x_2)v,\qq)}{u_2}\\
    ={}&(G1)+(G2)+(G3)+(G4). 
\end{align*}
It is easy to see that 
\begin{align*}
    \MoveEqLeft
    (F2)-(G3)\\
    ={}&\Res_{x_1,x_2} \frac{z_1^{-\wt a^1+1-\frac{r}{T}}(w+x_1)^{\wt a^1+\frac{r}{T}-1} z_2^{-\wt a^2-\frac{s}{T}+1}(w+x_2)^{\wt a^2+\frac{s}{T}-1} }{((w+x_2)z_1-(w+x_1)z_2)}\\
    &\qquad \cdot \left(\frac{w+x_2}{z_2-w-x_2}-\frac{w+x_1}{z_1-w-x_1}\right)S\bracket*{u_3}{(Y_{M^1}(a^1,x_1)Y_{M^1}(a^2,x_2)v,\qq)}{u_2}\\
    ={}&\Res_{x_1,x_2} \frac{z_1^{-\wt a^1+1-\frac{r}{T}}(w+x_1)^{\wt a^1+\frac{r}{T}-1} z_2^{-\wt a^2-\frac{s}{T}+1}(w+x_2)^{\wt a^2+\frac{s}{T}-1} }{(z_2-w-x_2)(z_1-w-x_1)}\\
    &\qquad \cdot S\bracket*{u_3}{(Y_{M^1}(a^1,x_1)Y_{M^1}(a^2,x_2)v,\qq)}{u_2}\\
    ={}&(G4).
\end{align*}
On the other hand, 
\begin{align*}
    \MoveEqLeft
    (F3)-(G2)\\
    ={}&\Res_{x_1,x_2} \frac{z_1^{-\wt a^1+1-\frac{r}{T}}(w+x_1)^{\wt a^1+\frac{r}{T}-1} z_2^{-\wt a^2-\frac{s}{T}+1}(w+x_2)^{\frac{s}{T}+\wt a^2-1} }{((w+x_1)z_2-(w+x_2)z_1)(z_2-w-x_2)}\\
    &\qquad \cdot \left(\frac{w+x_2}{z_2-w-x_2}-\frac{w+x_1}{z_1-w-x_1}\right)S\bracket*{u_3}{(Y_{M^1}(a^2,x_2)Y_{M^1}(a^1,x_1)v,\qq)}{u_2}\\
    ={}&-\Res_{x_1,x_2} \frac{z_1^{-\wt a^1+1-\frac{r}{T}}(w+x_1)^{\wt a^1+\frac{r}{T}-1} z_2^{-\wt a^2-\frac{s}{T}+1}(w+x_2)^{\frac{s}{T}+\wt a^2-1} }{(z_1-w-x_1)(z_2-w-x_2)}\\
    &\qquad \cdot S\bracket*{u_3}{(Y_{M^1}(a^2,x_2)Y_{M^1}(a^1,x_1)v,\qq)}{u_2}\\
    ={}&-(F4).
\end{align*}
Thus, $(F2)+(F3)+(F4)-(G2)-(G3)-(G4)=0$. Finally, recall that $o(\vo{L}{-1}a+\vo{L}{0}a)=0$, which implies that
$o(Y(a,x)b)=o((1+x)^{-\wt a-\wt b}Y\left(b,-x/(1+x)\right)a)$ for $a, b\in V$. Then, we have
\begin{align*}
    \MoveEqLeft
    (F1)=
    \Res_{x_1} \left(\frac{z_1^{-\wt a^1+1-\frac{r}{T}}z_2^{-\wt a^2-\frac{s}{T}+1}(1+x_1)^{\wt a^1+\frac{r}{T}-1}}{z_1-z_2-z_2x_1}\right)\\
    &\qquad \cdot S\bracket*{u_3\cdot[(1+x_1)^{-\wt a^1-\wt a^2}Y(a^2,\frac{-x_1}{1+x_1})a^1]}{(v,\qq)}{u_2}\\
    ={}&\Res_{x_2} \frac{z_1^{-\wt a^1+1-\frac{r}{T}}z_2^{-\wt a^2-\frac{s}{T}+1}\left(\frac{1}{1+x_2}\right)^{\wt a^1+\frac{r}{T}-1}}{z_1-z_2+z_2\left(\frac{x_2}{1+x_2}\right)}\cdot \frac{-1}{(1+x_2)^2}\cdot \left(\frac{1}{1+x_2}\right)^{-\wt a^1-\wt a^2}\\
    &\qquad S\bracket*{u_3\cdot[Y(a^2,x_2)a^1]}{(v,\qq)}{u_2}\\
    ={}&\Res_{x_2} \left(\frac{z_1^{-\wt a^1+1-\frac{r}{T}}z_2^{-\wt a^2-\frac{s}{T}+1}(1+x_2)^{\wt a^2-\frac{r}{T}}}{z_2-z_1-z_1x_2} \right)S\bracket*{u_3\cdot[Y(a^2,x_2)a^1]}{(v,\qq)}{u_2}\\
    ={}&(G1).
\end{align*}
Therefore, $(F1)+(F2)+(F3)+(F4)-(G1)-(G2)-(G3)-(G4)=0$. 

The proof of \cref{lem:com} is completed.
\end{proof}

\subsection{The \texorpdfstring{$(n+3)$}{(n+3)}-point correlation functions}

We can use a similar induction argument as in \cite[Section 4.3]{Liu} to construct the $(n+3)$-point function $S$, with the well-defineness and locality of the $3$-point, $4$-point, and $5$-point functions from the previous subsection as the base case. Note that the only property involving the $A(V)$-modules $U^2$ and $U^3$ we used for the induction process in \cite{Liu} was
\begin{align*}
    \MoveEqLeft
    S\bracket*{u_3\cdot[a^1][a^2]}{(a^3,\pp_3)\cdots (a^n,\pp_n)(v,\qq)}{u_2}\\
    &-S\bracket*{u_3\cdot[a^2][a^1]}{(a^3,\pp_3)\cdots (a^n,\pp_n)(v,\qq)}{u_2}\\
    ={}&\sum_{j\ge 0}\binom{\wt a^1-1}{j} S\bracket{u_3\cdot[\vo{a^1}{j}a^2]}{(a^3,\pp_3)\cdots (a^n,\pp_n)(v,\qq)}{u_2},
\end{align*}
which is also true when $U^2$ and $U^3$ are modules over $A_g(V)$. We omit the rest of the details for the induction. 

Thus, we have a well-defined system of $(n+3)$-point functions for $n\ge 3$.
\begin{equation}
\begin{aligned}
S_{V\cdots M\cdots V}\colon U^3\otimes V\otimes\cdots\otimes M^{1}&\otimes\cdots V\otimes U^2\rightarrow \O(\infty\Delta_n)\\
u_3\otimes a^{1}\otimes \cdots\otimes v\cdots\otimes a^{n}\otimes u_2&\mapsto S\bracket*{u_3}{(a_{1},\pp_{1})\cdots (v,\qq)\cdots(a^{n},\pp_{n})}{u_2},
\end{aligned}
\end{equation}
where $u_3\in U^3$, $a^1,\cdots a^n\in V$, $v\in M^1$, and $u_2\in U^2$. Note that $S$ satisfies the recursive formulas \labelcref{eq:recursive_U2,eq:recursive_U3}, and the \emph{locality} in \cref{def:genus-zero-bottom}.

\section{Associativity of the reconstructed correlation functions}\label{sec5}
In this section, we show that the system of $(n+3)$-point functions $S$ we constructed in \cref{sec4} is contained in $\Cor[\Sigma_{1}(U^3, M^1, U^2)]$.    

By our construction in \cref{sec4}, it remains to show the \emph{associativity} in \cref{def:genus-zero-bottom}. Since the recursive formulas for the correlation functions in \cref{def:genus-zero-bottom} are different from the ones in \cite{Liu}, there are $5$ new cases arise in our case. Recall there are two formulas for the associativity: for any $k\in \Z$,
\begin{align}
    \Res_{\pp_1=\qq}S\bracket*{u_3}{(a^{1},\pp_1)\cdots(v,\qq)}{u_2}(z_{1}-w)^{k}\d{z_1} &
    = S\bracket*{u_3}{\cdots(\vo{a^{1}}{k}v,\qq)}{u_2},\numberthis\label{5.1} \\
    \Res_{\pp_1=\pp_2}S\bracket*{u_3}{(a^{1},\pp_1)(a^{2},\pp_2)\cdots}{u_2}(z_{1}-z_{2})^{k}\d{z_1} &
    = S\bracket*{u_3}{(\vo{a^{1}}{k}a^{2},\pp_{2})\cdots}{u_2}. \numberthis\label{5.2}
\end{align}
\subsection{Associativity for one algebra element and one module element}

We first prove \labelcref{5.1} for the $4$-point functions. The general $(n+3)$-point functions case can be proved in a similar way, so we omit the details. 
\begin{proposition}\label{pro:associativity-4-point} 
  For $a^1\in V$, $u_3\in U^3$, $u_2\in U^2$, $v\in M^1$ homogeneous, and $k\in \Z$, we have 
\begin{equation}
  \Res_{\pp_1=\qq}S\bracket*{u_3}{(a^{1},\pp_{1})(v,\qq)}{u_2}(z_{1}-w)^{k} \d{z_{1}}=S\bracket*{u_3}{(\vo{a^1}{k}v,\qq)}{u_2}.
  \end{equation}
\end{proposition}
\begin{proof}  
    Suppose $a^1\in V^r$. When $k\ge 0$, by \cref{lem:ExpOfFpq,def:3-point,def:4-point}, 
    \begin{align*}
        \MoveEqLeft
        \Res_{\pp_1=\qq} S\bracket*{u_3}{(a^1,\pp_1)(v,\qq)}{u_2}(z_1-w)^k \d{z_1}\\
        ={}&\Res_{\pp_1=\qq} \braket*{\varphi}{u_3 \otimes {v}\otimes [a^1]\cdot u_2} w^{-\deg v} z_1^{-\wt a^1} (z_1-w)^k \d{z_1} \\
        &+\Res_{\pp_1=\qq} \sum_{i\ge 0}F_{\wt a^1-1+\frac{r}{T},i}(\pp_1,\qq) \braket*{\varphi}{u_3\otimes {\vo{a^1}{i}v}\otimes u_2} w^{-(\wt a^1-i-1+\deg v)} (z_1-w)^k \d{z_1}\\
        ={}&\Res_{\pp_1=\qq} \sum_{j\ge 0} \binom{-\wt a^1}{j} w^{-\deg v-\wt a^1-j} (z_1-w)^{j+k}\braket*{\varphi}{u_3 \otimes {v}\otimes [a^1]\cdot u_2} \\
        &+\Res_{\pp_1=\qq} \sum_{i\ge 0} \sum_{l=0}^i \sum_{p\ge 0} \binom{\wt a^1-1+\frac{r}{T}}{i-l}\binom{-\wt a^1+1-\frac{r}{T}}{p} \frac{w^{l-p-\wt a^1+1-\deg v}}{(z_1-w)^{l+1-p-k}} \\
        &\qquad \cdot \braket*{\varphi}{u_3\otimes {\vo{a^1}{i}v}\otimes  u_2}\\
        ={}&\sum_{i\ge 0}\left( \sum_{l=k}^i \binom{\wt a^1-1+\frac{r}{T}}{i-l}\binom{-\wt a^1+1-\frac{r}{T}}{l-k}\right)w^{k-\wt a^1 +1-\deg v} \braket*{\varphi}{u_3\otimes {\vo{a^1}{i}v}\otimes u_2}\\
        ={}&\sum_{i\ge 0}\left(\sum_{s=0}^{i-k} \binom{\wt a^1-1+\frac{r}{T}}{i-k-s} \binom{-\wt a^1+1-\frac{r}{T}}{s}\right) w^{k-\wt a^1+1-\deg v} \braket*{\varphi}{u_3\otimes {\vo{a^1}{i}v}\otimes u_2}\\
        ={}&\braket*{\varphi}{u_3\otimes {\vo{a^1}{k}v}\otimes u_2} w^{k-\wt a^1+1-\deg v}\\
        ={}&S\bracket*{u_3}{(\vo{a^1}{k}v,\qq)}{u_2},
    \end{align*}
    where we used the fact that $\sum_{s=0}^{i-k} \binom{\wt a^1-1+\tfrac{r}{T}}{i-k-s} \binom{-\wt a^1+1-\tfrac{r}{T}}{s}$ is the coefficient of the term $x^{i-k}$ in $(1+x)^{\wt a^1-1+\frac{r}{T}}(1+x)^{-\wt a^1+1-\frac{r}{T}}=1$. 
  
    When $k=-1$, we have
    \begin{align*}
        &\Res_{\pp_1=\qq} S\bracket*{u_3}{(a^1,\pp_1)(v,\qq)}{u_2}(z_1-w)^{-1} \d{z_1}\\
        ={}&\braket*{\varphi}{u_3\otimes {v}\otimes [a^1]\cdot u_2} w^{-\deg v-\wt a^1}\numberthis\label{5.4}\\
        &+\sum_{i\ge 0} \left(\sum_{l=0}^i \binom{\wt a^1-1+\frac{r}{T}}{i-l}\binom{-\wt a^1+1-\frac{r}{T}}{l+1}\right)  w^{-\wt a^1-\deg v} \braket*{\varphi}{u_3\otimes {\vo{a^1}{i}v}\otimes u_2}\\
        ={}&\braket*{\varphi}{u_3\otimes {v}\otimes [a^1]\cdot u_2} w^{-\deg v-\wt a^1}\\
        &-\sum_{i\ge 0} \binom{\wt a^1-1+\frac{r}{T}}{i+1}w^{-\wt a^1-\deg v}  \braket*{\varphi}{u_3\otimes {\vo{a^1}{i}v}\otimes u_2}.
    \end{align*}
    If $r=0$, by \labelcref{f-relation3} we have \[\labelcref{5.4}=\binom{\wt a^1-1}{0}\braket*{\varphi}{u_3\otimes {\vo{a^1}{-1}v}\otimes u_2}w^{-\deg v-\wt a^1}=S\bracket*{v'_3}{(\vo{a^1}{-1}v,\qq)}{u_2}.\] 
    If $r\neq 0$, since $[a^1]=0$, by \labelcref{f-relation4} we have
    \[\labelcref{5.4}=\binom{\wt a^1-1+\frac{r}{T}}{0}w^{-\wt a^1-\deg v}  \braket*{\varphi}{u_3\otimes {\vo{a^1}{-1}v}\otimes u_2}=S\bracket*{v'_3}{(\vo{a^1}{-1}v,\qq)}{u_2}.\]
    When $k<-1$, the proof is similar to the proof of \cite[(2.2.9)]{Z}, using the \emph{$\vo{L}{-1}$-derivative property} in \cref{Pro:L(-1)-derivative}, we omit the details. 
\end{proof}

\subsection{Associativity for two algebra elements}
In this subsection, we prove \labelcref{5.2} for the $5$-point functions. The general $(n+3)$-point functions case can be proved similarly.
First, we show that the kernel $J$ of conformal blocks in \cref{def:resCfb} also contains some information about a generalized version of $O(M^1)$ in \cite{FZ, Liu}. We give the following definition with the notations in \cite{DLM1,JJ}:

\begin{definition}\label{def:bimodule}
Let $(M,Y_M)$ be an untwisted module, and $\lambda$ be a complex number. Introduce a bilinear operator $\circ_g:V\otimes M\longrightarrow M$ by letting 
    \begin{align}
        a\circ_g v:&=\Res_x \frac{(1+x)^{\wt a-1+\delta(r)+\frac{r}{T}}}{x^{1+\delta(r)}} Y_M(a,x)v, \label{def:circle-g}
    \end{align}
    where $0\le r\le T-1$, $a\in V^r$ homogeneous, and $v\in M$. Let
    \begin{equation}\label{def:O-quotient}
        O_{g,\lambda}(M):=\spn\Set*{a\circ_g u,\ \vo{L}{-1}u+(\vo{L}{0}+\lambda)u: a\in V, u\in M},
    \end{equation}
    and $B_{g,\lambda}(M):=M/O_{g,\lambda}(M)$. 
\end{definition}
We will show that $B_{g,\lambda}(M)$ is a bimodule over the $g$-twisted Zhu's algebra $A_g(V)$ in the next Section. The following property is necessary for the proof of associativity \labelcref{5.2}.
\begin{lemma}\label{lem:contracted-space-and-bimodule}
  Let $J$ be the subspace spanned by elements of the form \labelcref{f-relation1}-\labelcref{f-relation4} in \cref{def:resCfb}. Then, $U^3\otimes O_{g,h_2-h_3}(M^1)\otimes U^2\subseteq J$. 
\end{lemma}
\begin{proof}
    By \labelcref{f-relation1}, \labelcref{f-relation4}, and \labelcref{def:circle-g}, it is clear that $u_3\otimes (\vo{L}{-1}u+(\vo{L}{0}+h_2-h_3)u)\otimes u_2\in J$, and $u_3\o(b\circ_g u) \otimes u_2\in J$, where $b\in V^r$ with $1\le t\le T-1$. Now let $a\in V^0$. Since $[\vo{L}{-1}a+\vo{L}{0}a]=0$, it follows from \labelcref{f-relation2} that
    \begin{align*}
        0 &= u_3\cdot[\vo{L}{-1}+\vo{L}{0}a]\otimes u\otimes u_2\\
        &\equiv -\sum_{j\ge 0} \binom{\wt a+1}{j} u_3 \otimes \vo{(\vo{L}{-1}a)}{j-1}u\otimes u_2-\sum_{j\ge 0} \binom{\wt a}{j} u_3 \otimes \vo{(\vo{L}{0}a)}{j-1}u\otimes u_2 \\
        &\equiv -u_3 \otimes \left((\vo{L}{-1}a+\vo{L}{0}a)\ast u\right)\otimes u_2      \equiv u_3 \otimes (a\circ_g u) \otimes u_2 \pmod{J}.  
    \end{align*}
  Thus $u_3\otimes O_{g,h_2-h_3}(M^1)\otimes u_2\subseteq J$, in view of \labelcref{def:O-quotient}. 
  \end{proof}
By \cref{lem:contracted-space-and-bimodule}, together with the \emph{$\vo{L}{-1}$-derivative property} of module $M^1$, it is easy to show the following fact (see \cite[Lemma 2.1.2]{Z}):
\begin{equation}\label{eq:O-relation}
u_3\otimes \left(\Res_x \frac{(1+x)^{\wt a-1+\delta(r)+\frac{r}{T}+i}}{x^{j+1+\delta(r)}}Y_{M^1}(a,x)v\right)\otimes u_2\in J,\quad j\ge i\ge 0, i, j\in \N.
\end{equation}

\begin{lemma}\label{lem:S-O-relation}
Let $S$ be the $3$-point function in \cref{def:3-point}. For $a\in V^r$ homogeneous, $v\in M^1$, $u_2\in U^2$, $u_3\in U^3$, and $j\in \N$, we have
    \begin{equation}\label{eq:S-O-relation}
        \Res_x \frac{(w+x)^{\wt a-1+\delta(r)+\frac{r}{T}}}{x^{j+\delta(r)}} S\bracket*{u_3}{(Y_{M^1}(a,x)v,\qq)}{u_2}=
        \begin{dcases*}
        S\bracket*{u_3\cdot[a]}{(v,\qq)}{u_2} & if $r,j=0$,\\
        0& if $j\ge 1$. 
        \end{dcases*}
    \end{equation}
\end{lemma}
\begin{proof}
By \cref{def:3-point} and the change of variable formula, we have
    \begin{align*}
        \MoveEqLeft
        \Res_x \frac{(w+x)^{\wt a-1+\delta(r)+\frac{r}{T}}}{x^{j+\delta(r)}} S\bracket*{u_3}{(Y_{M^1}(a,x)v,\qq)}{u_2}\\
        &=\Res_{x} \frac{(w+x)^{\wt a-1+\delta(r)+\frac{r}{T}}}{x^{j+\delta(r)}} \sum_{n\in \Z} \braket*{\varphi}{u_3\otimes \vo{a}{n+\frac{r}{T}}v\otimes u_2} x^{-n-\frac{r}{T}-1}w^{-\wt a+n+\frac{r}{T}+1-\deg v}\\
        &=\Res_x \frac{1}{w}\frac{(1+x/w)^{\wt a-1+\delta(r)+\frac{r}{T}}}{(x/w)^{j+\delta(r)}} \braket*{\varphi}{u_3\otimes Y_{M^1}(a,x/w)v\otimes u_2} w^{-j-\deg v-\frac{r}{T}}\\
        &=\Res_z \frac{(1+z)^{\wt a-1+\delta(r)+\frac{r}{T}}}{z^{j+\delta(r)}} \braket*{\varphi}{u_3\otimes Y_{M^1}(a,z)v\otimes u_2} w^{-j-\deg v-\frac{r}{T}}.
    \end{align*}
By \labelcref{eq:O-relation}, the last term is $0$ if $j\ge 1$. On the other hand, if $r, j=0$, by \labelcref{f-relation2} we have 
    \begin{align*}
        \MoveEqLeft
        \Res_z \frac{(1+z)^{\wt a}}{z} \braket*{\varphi}{u_3\otimes Y_{M^1}(a,z)v\otimes u_2} w^{-\deg v}\\
        &=\braket*{\varphi}{u_3 \otimes \sum_{i\ge 0} \binom{\wt a}{i} \vo{a}{i-1}v\otimes u_2} w^{-\deg v}\\
        &=\braket*{\varphi}{u_3\cdot [a]\otimes v\otimes u_2} w^{-\deg v}=S\bracket*{u_3\cdot [a]}{(v,\qq)}{u_2}.
    \end{align*}
This proves \labelcref{eq:S-O-relation}. 
\end{proof}

\begin{proposition}\label{pro:associativity-5-point}
  For any $u_3\in U^3$, $u_2\in U^2$, $v\in M^1$, $a^1,a^2\in V$, and $k\in \Z$, we have
  \begin{equation}\label{eq:associativity-5-point}
  \Res_{\pp_1=\pp_2}S\bracket*{u_3}{(a^{1},\pp_{1})(a^{2},\pp_{2})(v,\qq)}{u_2}(z_{1}-z_{2})^{k} \d{z_{1}}=S\bracket*{u_3}{(\vo{a^{1}}{k}a^{2},\pp_{2})(v,\qq)}{u_2}.
  \end{equation} 
\end{proposition}
\begin{proof}
It suffices to prove \cref{pro:associativity-5-point} for homogeneous $a^1\in V^r$ and $ a^2\in V^{s}$, where $0\le r, s<T$. Note that $\vo{a^1}{n}a^2\in V^{\overline{r+s}}$ for any $n\in \Z$, where $\overline{r+s}$ denotes the residue of $r+s$ modulo $T$.
  
When $k\ge 0$, by \labelcref{ExpansionFromLeft}, we have
    \begin{align*}
        \MoveEqLeft
        \Res_{\pp_1=\pp_2} S\bracket*{u_3}{(a^1,\pp_1)(a^2,\pp_2)(v,\qq)}{u_2}(z_1-z_2)^k \d{z_1}\\
        ={}&\Res_{\pp_1=\pp_2} S\bracket*{u_3\cdot [a^1]}{(a^2,\pp_2)(v,\qq)}{u_2}z_1^{-\wt a^1} (z_1-z_2)^k \d{z_1}\\
        &+ \Res_{\pp_1=\pp_2} \sum_{i\ge 0}F_{\wt a^1-1+\delta(r)+\frac{r}{T}}(\pp_1,\pp_2) S\bracket*{u_3}{(\vo{a^1}{i}a^2,\pp_2)(v,\qq)}{u_2}(z_1-z_2)^k  \d{z_1}\\
        &+\Res_{\pp_1=\pp_2} \sum_{i\ge 0} F_{\wt a^1-1+\delta(r)+\frac{r}{T}} (\pp_1,\qq) S\bracket*{u_3}{(a^2,\pp_2)(\vo{a^1}{i}v,\qq)}{u_2}(z_1-z_2)^k \d{z_1}\\
        ={}&0+\Res_{\pp_1=\pp_2} \sum_{i\ge 0}\sum_{l=0}^i \sum_{p\ge 0}      \binom{\wt a^1-1+\delta(r)+\frac{r}{T}}{i-l}\binom{-\wt a^1+1-         \delta(r)-\frac{r}{T}}{p} \\
        &\cdot \frac{z_2^{-i+l-p}}{(z_1-z_2)^{l+1-p-k}}S\bracket*{u_3}{(\vo{a^1}{i}a^2,\pp_2)(v,\qq)}{u_2}+0\\
        &= \sum_{i\ge 0} \left(\sum_{l=k}^i\binom{\wt a^1-1+\delta(r)+\frac{r}{T}}{i-l} \binom{-\wt a^1+1-\delta(r)-\frac{r}{T}}{l-k} \right)\\
        &\cdot S\bracket*{u_3}{(\vo{a^1}{i}a^2,\pp_2)(v,\qq)}{u_2}z_2^{-i+k} \\
        &= S\bracket*{u_3}{(\vo{a^1}{k}a^2,\pp_2)(v,\qq)}{u_2}.
    \end{align*}
  
Now consider the case where $k=-1$. Similar to the case $k\ge 0$, we have
    \begin{align*}
        &\Res_{\pp_1=\pp_2} S\bracket*{u_3}{(a^1,\pp_1)(a^2,\pp_2)(v,\qq)}{u_2}(z_1-z_2)^{-1} \d{z_1}-S\bracket*{u_3}{(\vo{a^1}{-1}a^2,\pp_2)(v,\qq)}{u_2}\\
        ={}&\Res_{\pp_1=\pp_2} S\bracket*{u_3\cdot [a^1]}{(a^2,\pp_2)(v,\qq)}{u_2}z_1^{-\wt a^1} (z_1-z_2)^{-1}\d{z_1}\\
        &+\Res_{\pp_1=\pp_2} \sum_{i\ge 0} F_{\wt a^1-1+\delta(r)+\frac{r}{T}, i}(\pp_1,\pp_2) S\bracket*{u_3}{\vo{(a^1}{i}a^2,\pp_2)(v,\qq)}{u_2}(z_1-z_2)^{-1}\d{z_1}\\
        &+\Res_{\pp_1=\pp_2}\Res_{x_1} \frac{z_1^{-\wt a^1+1-\delta(r)-\frac{r}{T}}(w+x_1)^{\wt a^1-1+\delta(r)+\frac{r}{T}}}{z_1-w-x_1} (z_1-z_2)^{-1}\d{z_1}\\
        &\cdot S\bracket*{u_3}{(a^2,\pp_2)(Y_{M^1}(a^1, x_1)v,\qq)}{u_2}\\
        &-S\bracket*{u_3}{(\vo{a^1}{-1}a^2,\pp_2)(v,\qq)}{u_2}\\
        ={}&S((u_3\cdot [a^1])\cdot[a^2], (v, \qq)u_2) z_2^{-\wt a^2-\wt a^1}\\
        &+\Res_{\pp_1=\pp_2} \sum_{i\ge 0} F_{\wt a^2-1+\delta(s)+\frac{s}{T},i}(\pp_2,\qq) S\bracket*{u_3\cdot [a^1]}{(\vo{a^2}{i}v,\qq)}{u_2}z_1^{-\wt a^1}(z_1-z_2)^{-1}\d{z_1}\\
        &-\sum_{i\ge 0}\binom{\wt a^1-1+\delta(r)+\frac{r}{T}}{i+1} S\bracket*{u_3}{(\vo{a^1}{i}a^2, z_2)(v, \qq)}{u_2} z_2^{-i-1}\\
        &+\Res_{x_1} \frac{z_2^{-\wt a^1+1-\delta(r)-\frac{r}{T}}(w+x_1)^{\wt a^1-1+\delta(r)+\frac{r}{T}}}{z_2-w-x_1} S\bracket*{u_3}{(a^2,\pp_2)(Y_{M^1}(a^1, x_1)v,\qq)}{u_2}\\
        &-S\bracket*{u_3}{(\vo{a^1}{-1}a^2,\pp_2)(v, \qq)}{u_2}\\
        ={}&\underbrace{S\bracket*{(u_3\cdot [a^1])\cdot[a^2]}{(v, \qq)}{u_2}z_2^{-\wt a^2-\wt a^1}}_{(A1)}\\
        &+\underbrace{\Res_{x_2}\frac{z_2^{-\wt a^2+1-\delta(s)-\frac{s}{T}-\wt a^1}(w+x_2)^{\wt a^2-1+\delta(s)+\frac{s}{T}}}{z_2-w-x_2} S\bracket*{u_3\cdot [a^1]}{(Y_{M^1}(a^2,x_2)v,\qq)}{u_2}}_{(B1)} \\
        &\underbrace{-\sum_{i\ge 0} \binom{\wt a^1-1+\delta(r)+\frac{r}{T}}{i+1} S\bracket*{u_3\cdot[\vo{a^1}{i}a^2]}{(v, \qq)}{u_2}z_2^{-\wt a^1-\wt a^2}}_{(A2)}\\
        &\underbrace{-\sum_{i\ge 0} \binom{\wt a^1-1+\delta(r)+\frac{r}{T}}{i+1} \Res_{x_2}\frac{(w+x_2)^{\wt a^1-i-1+\wt a^2-1+\delta(\overline{r+s})+\frac{\overline{r+s}}{T}}}{z_2-w-x_2}}_{(B_2)}\\
        &\underbrace{\quad \cdot z_2^{-\wt a^1-\wt a^2+1-\delta(\overline{r+s})-\frac{\overline{r+s}}{T}}S\bracket*{u_3}{(Y_{M^1}(\vo{a^1}{i}a^2,x_2)v,\qq)}{u_2}}_{(B2)}\\
        &+\underbrace{\Res_{x_1} \frac{z_2^{-\wt a^1+1-\delta(r)-\frac{r}{T}}(w+x_1)^{\wt a^1-1+\delta(r)+\frac{r}{T}}}{z_2-w-x_1} S\bracket*{u_3\cdot[a^2]}{(Y_{M^1}(a^1,x_1)v,\qq)}{u_2}z_2^{-\wt a^2}}_{(B3)}\\
        &+\underbrace{\Res_{x_1}\Res_{x_2}\frac{(w+x_1)^{\wt a^1-1+\delta(r)+\frac{r}{T}}}{z_2-w-x_1}\cdot \frac{(w+x_2)^{\wt a^2-1+\delta(s)+\frac{s}{T}}}{z_2-w-x_2}}_{(B4)}\\
        &\phantom{=}  \cdot \underbrace{z_2^{-\wt a^1-\wt a^2+2-\delta(r)-\delta(s)-\frac{s+r}{T}}S\bracket*{u_3}{(Y_{M^1}(a^2,x_2)Y_{M^1}(a^1,x_1)v,\qq)}{u_2}}_{(B4)}\\
        &\underbrace{-S\bracket*{u_3\cdot[\vo{a^1}{-1}a^2]}{(v, \qq)}{u_2}z_2^{-\wt a^1-\wt a^2}}_{(A3)}\\
        &\underbrace{-\Res_{x_2}\frac{z_2^{-\wt a^1-\wt a^2+1-\delta(\overline{r+s})-\frac{\overline{r+s}}{T}}(w+x_2)^{\wt a^1+\wt a^2-1+\delta(\overline{r+s})+\frac{\overline{r+s}}{T}}}{z_2-w-x_2}}_{(B5)}\\
        &\phantom{=}  \cdot \underbrace{S\bracket*{u_3}{(Y_{M^1}(\vo{a^1}{-1}a^2,x_2)v,\qq)}{u_2}}_{(B5)}\\
        ={}&(A1)+(A2)+(A3)+(B1)+(B2)+(B3)+(B4)+(B5). 
    \end{align*}
By \labelcref{eq:def:3-point} and \labelcref{f-relation2}, we have 
    \begin{align*}
        \MoveEqLeft
        (A1)+(A2)+(A3)\\
        &=S\bracket*{(u_3\cdot [a^1])\cdot[a^2]}{(v, \qq)}{u_2}z_2^{-\wt a^2-\wt a^1}\\
        &-\sum_{i\ge 0} \binom{\wt a^1-1+\delta(r)+\frac{r}{T}}{i+1} S\bracket*{u_3\cdot[\vo{a^1}{i}a^2}{(v, \qq)}{u_2}z_2^{-\wt a^1-\wt a^2}\\
        &-S\bracket*{u_3\cdot[\vo{a^1}{-1}a^2]}{(v, \qq)}{u_2}z_2^{-\wt a^1-\wt a^2}\\
        ={}&\braket*{\varphi}{(u_3\cdot [a^1])\cdot[a^2]\otimes v\otimes u_2- \sum_{j\ge 0}\binom{\wt a^1-1+\delta(r)+\frac{r}{T}}{j}u_3\cdot[\vo{a^1}{j-1}a^2]\otimes v\otimes u_2}\\
        &\phantom{=}  \cdot z_{2}^{\wt a^1-\wt a^2} w^{-\deg v}\\   
        ={}&0.
    \end{align*}
The last equality follows from the fact that $U^3$ is a right module over $A_g(V)$. More precisely, if $r\neq 0$, we have $[a^1]=0$, and $\sum_{j\ge 0} \binom{\wt a^1-1+\frac{r}{T}}{j}[\vo{a^1}{j-1}a^2]=0$; if $r=0$ and $s\neq 0$, the last equality holds since $[a^2]=[\vo{a^1}{j-1}a^2]=0$ for all $j\ge 0$; if $r=s=0 $, the last equality holds since $(u_3\cdot [a^1])\cdot[a^2]=u_3([a^1]\ast_g [a^2])$. 

On the other hand, by the Jacobi identity, we can express $(B2)+(B5)$ as follows: 
    \begin{align*}
        \MoveEqLeft
        (B2)+(B5)\\
        ={}&-\sum_{j\ge 0}\Res_{x_2}\Res_{x_1-x_2} (x_1-x_2)^{j-1} \binom{\wt a^1-1+\delta(r)+\frac{r}{T}}{j} (x_1+x_2)^{-j}\\ 
        &\cdot \frac{z_2^{-\wt a^1-\wt a^2+1-\delta(\overline{r+s})-\frac{\overline{r+s}}{T}}(w+x_2)^{\wt a^1+\wt a^2-1+\delta(\overline{r+s})+\frac{\overline{r+s}}{T}}}{z_2-w-x_2}\\
        &\cdot S\bracket*{u_3}{(Y_{M^1}(Y(a^1,x_1-x_2)a^2,x_2)v,\qq)}{u_2}\\
        ={}&-\Res_{x_2}\Res_{x_1-x_2} \frac{1}{x_1-x_2}\left(1+\frac{x_1-x_2}{w+x_2}\right)^{\wt a^1-1+\delta(r)+\frac{r}{T}}z_2^{-\wt a^1-\wt a^2+1-\delta(\overline{r+s})-\frac{\overline{r+s}}{T}} \\
        &\cdot \frac{(w+x_2)^{\wt a^1+\wt a^2-1+\delta(\overline{r+s})+\frac{\overline{r+s}}{T}}}{z_2-w-x_2}
        S\bracket{u_3}{(Y_{M^1}(Y(a^1,x_1-x_2)a^2,x_2)v,\qq)}{u_2}\\
        ={}& \underbrace{-\Res_{x_1,x_2}\left(\frac{1}{x_1-x_2}\right) \frac{ (w+x_1)^{\wt a^1-1+\delta(r)+\frac{r}{T}}(w+x_2)^{\wt a^2+\delta(\overline{r+s})+\frac{\overline{r+s}}{T}-\delta(r)-\frac{r}{T}}}{z_2-w-x_2}}_{(C1)}\\
        &\phantom{=}  \underbrace{\cdot z_2^{-\wt a^1-\wt a^2+1-\delta(\overline{r+s})-\frac{\overline{r+s}}{T}}S\bracket*{u_3}{(Y_{M^1}(a^1,x_1)Y_{M^1}(a^2,x_2)v,\qq)}{u_2}}_{(C1)}\\
        &+\underbrace{\Res_{x_1,x_2}\left(\frac{1}{-x_2+x_1}\right) \frac{ (w+x_1)^{\wt a^1-1+\delta(r)+\frac{r}{T}}(w+x_2)^{\wt a^2+\delta(\overline{r+s})+\frac{\overline{r+s}}{T}-\delta(r)-\frac{r}{T}}}{z_2-w-x_2}}_{(C2)}\\
        &\phantom{=}  \underbrace{\cdot z_2^{-\wt a^1-\wt a^2+1-\delta(\overline{r+s})-\frac{\overline{r+s}}{T}} S\bracket*{u_3}{(Y_{M^1}(a^2,x_2)Y_{M^1}(a^1,x_1)v,\qq)}{u_2}}_{(C2)}\\
        ={}&(C1)+(C2).
  \end{align*}
By \cref{lem:S-O-relation}, we have
    \begin{align*}
        (C1)&=-\sum_{j\ge 0} \Res_{x_2}\frac{z_2^{-\wt a^1-\wt a^2+1-\delta(\overline{r+s})-\frac{\overline{r+s}}{T}}(w+x_2)^{\wt a^2+\delta(\overline{r+s})+\frac{\overline{r+s}}{T}-\delta(r)-\frac{r}{T}}x_2^j}{z_2-w-x_2}\\
        &\phantom{=}  \cdot S\bracket*{u_3}{(\Res_{x_1} \frac{(w+x_1)^{\wt a^1-1+\delta(r)+\frac{r}{T}}}{x_1^{j+1}}Y_{M^1}(a^1,x_1)Y_{M^1}(a^2,x_2)v,\qq)}{u_2} \\
        &=-\Res_{x_2} \frac{z_2^{-\wt a^1-\wt a^2+1-\delta(\overline{r+s})-\frac{\overline{r+s}}{T}}(w+x_2)^{\wt a^2+\delta(\overline{r+s})+\frac{\overline{r+s}}{T}-\delta(r)-\frac{r}{T}}x_2^j}{z_2-w-x_2}\\
        &\phantom{=}  \cdot S\bracket*{u_3\cdot [a^1]}{(Y_{M^1}(a^2,x_2)v,\qq)}{u_2}\\
        &=-(B1). 
  \end{align*}
The last equality holds since both $(B1)$ and $(C1)$ are equal to $0$ when $r\neq0$.
 
For $(C2)+(B4)$, using \cref{lem:S-O-relation} again, we have 
    \begin{align*}
        &(C2)+(B4)\\
        ={}&\Res_{x_1,x_2}\left(\frac{1}{-x_2+x_1}\right) \frac{ (w+x_1)^{\wt a^1-1+\delta(r)+\frac{r}{T}}(w+x_2)^{\wt a^2+\delta(\overline{r+s})+\frac{\overline{r+s}}{T}-\delta(r)-\frac{r}{T}}}{z_2-w-x_2}\\
        &\phantom{=}  \cdot  z_2^{-\wt a^1-\wt a^2+1-\delta(\overline{r+s})-\frac{\overline{r+s}}{T}} S\bracket*{u_3}{(Y_{M^1}(a^2,x_2)Y_{M^1}(a^1,x_1)v,\qq)}{u_2} \\
        &\phantom{=}  +\Res_{x_1}\Res_{x_2} \frac{(w+x_1)^{\wt a^1-1+\delta(r)+\frac{r}{T}}}{z_2-w-x_1}\cdot \frac{(w+x_2)^{\wt a^2-1+\delta(s)+\frac{s}{T}}}{z_2-w-x_2}\\
        &\phantom{=}  \cdot z_2^{-\wt a^1-\wt a^2+2-\delta(r)-\delta(s)-\frac{r+s}{T}}S\bracket*{u_3}{(Y_{M^1}(a^2,x_2)Y_{M^1}(a^1,x_1)v,\qq)}{u_2}.
    \end{align*}
Note that $(C2)+(B4)$ varies when $r$ and $ s$ take different values. There are 6 cases in total: (1) $r=s=0$, (2) $r=0$ and $ s\neq 0$, (3) $r\neq 0$ and $ s=0$, (4) $r+s=T$, (5) $r, s\neq 0$, and $r+s<T$, and (6) $r+s>T$. We only present the proof for the case $r=s=0$ and the case $r+s=T$. The proof of other cases are similar, we omit the details. 
 
 Case (1): $r=s=0$. In this case, 
 \begin{align*}
     &(C2)+(B4)\\
     ={}&\Res_{x_1}\Res_{x_2} \frac{ (w+x_1)^{\wt a^1}(w+x_2)^{\wt a^2}}{z_2-w-x_2}\left(\frac{1}{-x_2+x_1}+\frac{1}{z_2-w-x_1}\right)z_2^{-\wt a^1-\wt a^2}\\
     &\phantom{=}  \cdot S\bracket*{u_3}{(Y_{M^1}(a^2,x_2)Y_{M^1}(a^1,x_1)v,\qq)}{u_2}\\
     ={}&\Res_{x_1}\Res_{x_2} \frac{ (w+x_1)^{\wt a^1}(w+x_2)^{\wt a^2}}{z_2-w-x_2}\left(\frac{z_2-w-x_2}{(-x_2+x_1)(z_2-w-x_1)}\right)z_2^{-\wt a^1-\wt a^2}\\
     &\phantom{=}  \cdot S\bracket*{u_3}{(Y_{M^1}(a^2,x_2)Y_{M^1}(a^1,x_1)v,\qq)}{u_2}\\
     =-&\Res_{x_1} \frac{ (w+x_1)^{\wt a^1}}{z_2-w-x_1}\Res_{x_2}\sum_{j\ge 0}\left(\frac{(w+x_2)^{\wt a^2}}{x_2^{1+j}}\right)x_1^jz_2^{-\wt a^1-\wt a^2}\\
     &\phantom{=}  \cdot S\bracket*{u_3}{(Y_{M^1}(a^2,x_2)Y_{M^1}(a^1,x_1)v,\qq)}{u_2}\\
     ={}&-(B3),
 \end{align*}
 where the last equality follows from \cref{lem:S-O-relation}.

 Case (4): $r+s=T$. Since $r,s\neq 0$, we have
 \begin{align*}
     &(C2)+(B4)\\
     ={}&\Res_{x_1}\Res_{x_2} \frac{ (w+x_1)^{\wt a^1-1+\frac{r}{T}}(w+x_2)^{\wt a^2+\frac{s}{T}}}{z_2-w-x_2}\left(\frac{1}{-x_2+x_1}\right)z_2^{-\wt a^1-\wt a^2}\\
     &\phantom{=}  \cdot S\bracket*{u_3}{(Y_{M^1}(a^2,x_2)Y_{M^1}(a^1,x_1)v,\qq)}{u_2}\\
     &+\Res_{x_1}\Res_{x_2} \frac{ (w+x_1)^{\wt a^1-1+\frac{r}{T}}(w+x_2)^{\wt a^2-1+\frac{s}{T}}}{z_2-w-x_2}\left(\frac{1}{z_2-w-x_1}\right)z_2^{-\wt a^1-\wt a^2+1}\\
     &\phantom{=}  \cdot S\bracket*{u_3}{(Y_{M^1}(a^2,x_2)Y_{M^1}(a^1,x_1)v,\qq)}{u_2}\\
     ={}&\Res_{x_1}\Res_{x_2} \frac{ (w+x_1)^{\wt a^1-1+\frac{r}{T}}(w+x_2)^{\wt a^2-1+\frac{s}{T}}}{z_2-w-x_2}\left(\frac{w+x_2}{-x_2+x_1}+\frac{z_2}{z_2-w-x_1}\right)\\
     &\phantom{=}  \cdot z_2^{-\wt a^1-\wt a^2}S\bracket{u_3}{(Y_{M^1}(a^2,x_2)Y_{M^1}(a^1,x_1)v,\qq)}{u_2}\\
     ={}&\Res_{x_1}\Res_{x_2} \frac{ (w+x_1)^{\wt a^1+\frac{r}{T}}(w+x_2)^{\wt a^2-1+\frac{s}{T}}}{z_2-w-x_1}\left(\frac{1}{-x_2+x_1}\right)z_2^{-\wt a^1-\wt a^2}\\
     &\phantom{=}  \cdot S\bracket{u_3}{(Y_{M^1}(a^2,x_2)Y_{M^1}(a^1,x_1)v,\qq)}{u_2}\\
     ={}&-\Res_{x_1}\frac{ (w+x_1)^{\wt a^1+\frac{r}{T}}}{z_2-w-x_1}\sum_{j\ge0}\Res_{x_2}\left(\frac{(w+x_2)^{\wt a^2-1+\frac{s}{T}}}{x_2^{1+j}}\right)x_1^jz_2^{-\wt a^1-\wt a^2}\\
     &\phantom{=}  \cdot S\bracket{u_3}{(Y_{M^1}(a^2,x_2)Y_{M^1}(a^1,x_1)v,\qq)}{u_2}\\
     ={}&0=-(B3).
 \end{align*}
Therefore, we have $(B1)+(B2)+(B3)+(B4)+(B5)=((B1)+(C1))+((C2)+(B4)+(B3))=0$, and so
\[\Res_{\pp_1=\pp_2} S\bracket*{u_3}{(a^1, \pp_1)(a^2,\pp_2)(v, \qq)}{u_2} (z_1-z_2)^{-1} \d{z_1} =S\bracket*{u_3}{(\vo{a^1}{-1}a^2,\pp_2)(v, \qq)}{u_2}.\]
This proves \labelcref{5.2} for $k=-1$. When $k<-1$, it can be proved inductively using the \emph{$\vo{L}{-1}$-derivative property}, we omit the details. 
\end{proof}

\begin{proposition}\label{prop:Cbf-to-Cor}
The system of $(n+3)$-point functions $S=\Set*{S_{V\cdots M\cdots V}}$ constructed from the given $\varphi\in\Cfb[U^3, M^1, U^2]$ in \cref{sec4} lies in $\Cor[\Sigma_{1}(U^3, M^1, U^2)]$. 
\end{proposition}
\begin{proof}
    It follows from the construction of $S$, \cref{Pro:L(-1)-derivative,pro:associativity-4-point,pro:associativity-5-point}.
\end{proof}
Now the proof of \cref{thm:iso-restrictcfb-bottomcorrelation} is complete. 

\section{Fusion rules characterized by \texorpdfstring{$A_g(V)$}{Ag(V)}-bimodules}\label{sec6}
In \cref{def:bimodule}, we constructed a quotient space $B_{g,\lambda}(M)=M/O_{g,\lambda}(M)$ associated to an untwisted module $M$ and a complex number $\lambda$. 

In this section, we will show that $B_{g,\lambda}(M)=M/O_{g,\lambda}(M)$ is in fact a bimodule over $A_g(V)$, and the space of coinvariants $(U^3\o M^1\o U^2)/J$ in \cref{def:resCfb}  is isomorphic to the tensor product space $U^3\otimes_{A_g(V)}B_{g,\lambda}(M^1)\otimes_{A_g(V)}U^2$, where $\lambda=h_2-h_3$. Moreover, we will show that the tensor products $U^3\otimes_{A_g(V)} B_{g,\lambda}(M^1)\otimes _{A_g(V)} U^2$ is isomorphic to the tensor product $U^3\otimes_{A_g(V)} A_g(M^1)\otimes _{A_g(V)} U^2$, which gives us two ways to compute fusion rules using $A_g(V)$-bimodules. Finally, we will show that the $g$-twisted fusion rules are all finite when the VOA $V$ is $g$-rational and $C_2$-cofinite. 

\subsection{The \texorpdfstring{$A_g(V)$}{Ag(V)}-bimodules \texorpdfstring{$B_{g,\lambda}(M^1)$}{Bglambda(M1)}}
\begin{lemma}\label{lem:bimodule}
Let $M$ be an untwisted module of conformal weight $h_1$, and let $\lambda$ be a complex number. For $a\in V^r$ and $u\in M$, define 
    \begin{equation}\label{eq:left-right-actions}
        \begin{aligned}
        a\ast_g u:= &\begin{cases}\Res_z Y_M(a,z)u\frac{(1+z)^{\wt a}}{z}& \mathrm{if}\ r=0\\
        0& \mathrm{if}\ r\neq 0  \end{cases},\\
        u\ast_g a:= &\begin{cases}\Res_z Y_M(a,z)u\frac{(1+z)^{\wt a-1}}{z}& \mathrm{if}\ r=0\\
        0& \mathrm{if}\ r\neq 0 \end{cases}.
        \end{aligned}
    \end{equation}
Then, we have $b\ast_g O_{g,\lambda}(M)\subseteq O_{g,\lambda}(M)$ and $O_{g,\lambda}(M)\ast _g b\subseteq O_{g,\lambda}(M)$ for any $b\in V$. Moreover, $B_{g,\lambda}(M)=M/O_{g,\lambda}(M)$ is a bimodule over $A_g(V)$, with respect to the products $a\ast_g u$ and $u\ast _g a$ in \labelcref{eq:left-right-actions}.
\end{lemma}
\begin{proof}
In \cite{JJ}, Jiang and Jiao introduced a quotient space: 
\begin{equation}\label{6.2'}
A_g(M)=M/ O_g(M),
\end{equation}
where $O_g(M)$ is spanned by $a\circ_g u$ for all $a\in V, u\in M$, and $a\circ _g u$ is given by \labelcref{def:circle-g}. They proved that $A_g(M)$ is an $A_g(V)$-bimodule with left and right actions given by \labelcref{eq:left-right-actions}. In particular, $b\ast_g O_g(M)\subseteq O_g(M)$ and $O_g(M)\ast_g b\subseteq O_g(M)$ for all $b\in V$. Since $a\circ_g u=a\circ u$ for $a\in V^0$, we introduce an intermediate subspace 
\[
  O^0_{g,\lambda}(M):= \spn\Set*{ a\circ_g u,\vo{L}{-1}u+(\vo{L}{0}+\lambda)u:a\in V^0, u\in M }\subseteq O_{g,\lambda}(M).
\]
By \cite[Lemmas 4.3, 4.4 and 4.5]{Liu}, together with the fact that $b\ast_g u=u\ast _g b=0$ for $b\in V^r$ with $r>0$, we have $b\ast_g O^0_{g,\lambda}(M)\subseteq O^0_{g,\lambda}(M)$ and $O^0_{g,\lambda}(M)\ast_g b\subseteq O^0_{g,\lambda}(M)$ for all $b\in V$. In particular, for any $u\in M$, we have
\[b\ast_g (\vo{L}{-1}u+(\vo{L}{0}+\lambda)u), \ (\vo{L}{-1}u+(\vo{L}{0}+\lambda)u)\ast_g b\in O^0_{g,\lambda}(M)\subseteq O_{g,\lambda}(M).\]
Then the conclusion follows from 
\[
  O_{g,\lambda}(M)=O_g(M)+
  \spn\Set*{ \vo{L}{-1}u+(\vo{L}{0}+\lambda)u:u\in M},
\] in view of \labelcref{def:O-quotient}.
By \cref{lem:bimodule} and \cite[Theorem 3.4]{JJ}, $B_{g,\lambda}(M)=M/O_{g,\lambda}(M)$ is a bimodule over $A_g(V)$ with respect to the products $a\ast_g u$ and $u\ast _g a$ in \labelcref{eq:left-right-actions}. 
\end{proof}

Since there is an epimorphism of associative algebras $A(V^0)\longrightarrow A_g(V)$ (\cite{DLM1}), $B_{g,\lambda}(M)$ and $A_g(M)$ are also bimodules over $A(V^0)$ with actions $[a]\cdot[u]=[a\ast _g u]=[a\ast u]$ and $[u]\cdot[a]=[u\ast_g a]=[u\ast a]$, where $a\in V^0$ and $[u]\in B_{g,\lambda}(M)$ or $A_g(M)$, in view of \labelcref{eq:left-right-actions}. 

\begin{proposition}\label{pro:iso-between-tensor}
Let $M^1$ be an untwisted module of conformal weight $h_1$, $U^2$ (resp. $U^3$) be a left (resp. right) irreducible $A_g(V)$-modules on which $[\upomega]$ acts as $h_2\id$ (resp. $h_3\id$). Then we have an isomorphism of vector spaces
    \begin{equation}\label{eq:iso-between-tensors}
        \begin{aligned}
            (U^3\otimes M^1\otimes U^2)/J\cong U^3\otimes_{A_g(V)}B_{g,\lambda}(M^1)\otimes_{A_g(V)}U^2\cong U^3\o_{A(V^0)} B_{g,\lambda}(M^1)\o_{A(V^0)}U^2,
        \end{aligned}
    \end{equation}
    where $J$ is given by \labelcref{f-relation1}--\labelcref{f-relation4}, and $\lambda=h_2-h_3$. 
\end{proposition}
\begin{proof}
    Define a linear map 
    \begin{align*}
        \phi: U^3\otimes M^1 \otimes U^2&\longrightarrow U^3\otimes_{A_g(V)}B_{g,\lambda}(M^1)\otimes_{A_g(V)}U^2,\\
        \phi(u_3\o v\o u_2):&=u_3\o [v]\o u_2,\quad u_3\in U^3, v\in M^1, u_2\in U^2,
    \end{align*}
    where $[v]$ is the image of $v\in M^1$ in $B_{g,\lambda}(M^1)$. By \cref{def:bimodule}, it is straightforward to see that $\phi$ factors through $(U^3\otimes M^1\otimes U^2)/J$. Denote the induced map by $\overline{\phi}$.
  
    Conversely, we consider the following linear map: 
    \begin{equation*}
      \psi: U^3\o_{A_g(V)}B_{g,\lambda}(M^1)\o_{A_g(V)}U^2\longrightarrow (U^3\o_\C M^1\o_\C U^2)/J,\ \psi(u_3\o [v]\o u_2):=u_3 \o v\o u_2+J.
    \end{equation*}
	Indeed, by \cref{lem:contracted-space-and-bimodule}, we have $\psi(u_3\o [O_{g,\lambda}(M^1)]\o u_2)=u_3 \o O_{g,\lambda}(M^1)\o u_2+J=0$. Furthermore, recall that $A_g(V)$ is a quotient of $A(V^0)$ \cite{DLM1}, and so $U^2$ and $U^3$ are also left and right modules over $A(V^0)$, respectively. Let $a\in V^0$, by \labelcref{f-relation2}, \labelcref{f-relation3}, and \labelcref{eq:left-right-actions}, we have 
	\begin{align*}
	\psi(u_3[a]\o [v]\o u_2)=\psi(u_3\o [a\ast_g v]\o u_2),\ \psi(u_3\o [v]\o [a] u_2)=\psi(u_3\o [v\ast_g a]\o u_2).
	\end{align*}
	Hence $\psi$ is well-defined. It is clear that $\psi$ is an inverse of $\bar{\phi}$. Observe that $[a\ast _g v]=[a\ast v]$ and $[v\ast_g a]=[v\ast a]$ for $a\in V^0$, then by adopting a similar argument, we can also show that $(U^3\o_\C M^1\o_\C U^2)/J\cong U^3\o_{A(V^0)} B_{g,\lambda}(M^1)\o_{A(V^0)}U^2$. 
\end{proof}

Although the $A_g(V)$-bimodules $B_{g,h_2-h_3}(M^1)$ and $A_g(M^1)$ are not isomorphic in general, the tensor products $U^3\otimes_{A_g(V)}B_{g,h_2-h_3}(M^1)\o _{A_g(V)}U^2$ and $U^3\otimes_{A_g(V)}A_{g}(M^1)\o _{A_g(V)}U^2$ are in fact isomorphic. This isomorphism was proved in \cite{Liu1} for the untwisted case under the assumption that $A(V)$ is semi-simple. Now we drop the semi-simplicity condition. 

\begin{proposition}\label{pro:iso-between-different-bimodule-tensors}
With the aforementioned assumptions, we have a linear isomorphism:  
\begin{equation}\label{eq:iso-between-different-bimodule-tensors}
    \begin{aligned}
	&U^3\otimes_{A_g(V)} B_{g,h_2-h_3}(M^1)\otimes _{A_g(V)} U^2\cong U^3\otimes_{A_g(V)} A_g(M^1)\otimes _{A_g(V)} U^2\\
	&\cong U^3\otimes_{A(V^0)} B_{g,h_2-h_3}(M^1)\otimes _{A(V^0)} U^2\cong U^3\otimes_{A(V^0)} A_g(M^1)\otimes _{A(V^0)} U^2.
	\end{aligned}
\end{equation}
\end{proposition}
\begin{proof}
  By \labelcref{def:O-quotient} and \cref{lem:bimodule}, $B_{g,h_2-h_3}(M^1)= A_g(M^1)/I$, where 
  \[I=\spn\Set*{[(\vo{L}{-1}+\vo{L}{0}+h_2-h_{3})u]\given u\in M^1}=O_{g,h_2-h_3}(M^1)/O_{g}(M^1)\] is a sub-bimodule of $A_g(M^1)$.
 
  Observe that 
  $\upomega \ast _gu-u\ast_g \upomega+(h_2-h_3)u =(\vo{L}{-1}+\vo{L}{0}+h_2-h_3)u$ for any $u\in M^1$, where $\upomega \in V$ is the conformal vector. Thus,
  \begin{equation} 
  I=\spn\Set*{ [\upomega]\ast_g [u]-[u]\ast_g [\upomega]+(h_{2}-h_{3})[u]\given u\in M^{1}}\subset A_g(M^1).
  \end{equation}
  Recall that $[\upomega]\in A_g(V)$ is a central element \cite{DLM1,Z}, $I$ is a sub-bimodule of $A_g(M^1)$. (This gives an alternative proof of \cref{lem:bimodule}.) Denote the inclusion map $I\monomorphism A_g(M^1)$ by $\iota$, and $A_g(V)$ by $A$ for short, then by the right exactness of tensor functor, we have a right exact sequence: 
  \begin{equation}\label{eq:right-exact-sequence}
    U^3\otimes_{A}I \otimes_{A}U^2\overset{1\o \iota\o 1}{\longrightarrow} 
    U^3\otimes_{A} A_g(M^1) \otimes_{A}U^2 \to
    U^3\otimes_{A}(A_g(M^1)/I) \otimes_{A}U^2 \to 0.
  \end{equation}
  We claim that $(1\o \iota\o1)(U^3\otimes_{A} I\otimes_{A} U^2)=0$ in $U^3\otimes_{A} A_g(M^{1})\otimes_{A}U^2$. Indeed, for any $u_3\in U^3$, $u_2\in U^2$, and $u\in M^1$, we have
  \begin{align*}
  &(1\o \iota\o1)(u_3\otimes ([\upomega]\ast [u]-[u]\ast [\upomega]+(h_{2}-h_{3})[u])\otimes u_2)\\
  &=u_3([\upomega]-h_{3})\otimes [u]\otimes v_2-u_3\otimes [u]\otimes ([\upomega]-h_{2})u_2\\
  &=0.
  \end{align*}
in $U^3\otimes_{A} A_g(M^{1})\otimes_{A}U^2$. Then, the first isomorphism in \labelcref{eq:iso-between-different-bimodule-tensors} follows from \labelcref{eq:right-exact-sequence}. The last two isomorphism follows from \cref{pro:iso-between-tensor} and its proof.
\end{proof}

\begin{remark}  
In the last remark of \cite{Liu}, the author made a false claim that the isomorphism \labelcref{eq:iso-between-different-bimodule-tensors} does not hold in general for the untwisted case. This was due to a mistake in Example 4.22 in \cite{Liu}. We make a correction here:

In the isomorphism $B_{h}(M(c,h_{1}))\otimes_{A(M_{c})} M_{c}(0)\cong \C[t_{0}]\otimes _{\C[t]} M_{c}(0)\cong M_{c}(0)$, the left $A(M(c,0))=\C[t]$-action is given by 
\begin{align*}
t.(1\o v_{c,0})=(t_0+h_2)\o v_{c,0}=1.t\o v_{c,0}+h_2\o v_{c,0}=1\o \vo{L}{0}v_{c,0}+h_2\o v_{c,0}=h_2\o v_{c,0}.
\end{align*}
Thus the left $\C[t]$-module $\C[t_{0}]\otimes _{\C[t]} M_{c}(0)\cong M_{c}(0)$ is isomorphic to $M(c,h_2)(0)=\C v_{c,h_2}$, and so 
$((M(c,h_{2})(0))^{\ast}\otimes _{A(M_{c})}B_{h}(M(c,h_{1}))\otimes _{A(M_{c})}M_{c}(0))^{\ast}\cong \Hom_{A(M_{c})}( M_{c}(0),M(c,h_{2})(0))$ is $1$-dimensional, same as  $(M(c,h_{2})(0)^{\ast}\otimes _{A(M_{c})}A(M(c,h_{1}))\otimes _{A(M_{c})}M_{c}(0))^{\ast}$. 
\end{remark}

\begin{theorem}\label{thm:twisted-fusion-rules-theorem}
  Let $M^1$ be an untwisted lowest weight module of conformal weight $h_1$, $M^2$ (resp. $M^3$) be a lowest weight $g$-twisted module of conformal weight $h_2$ (resp. $h_3$) with bottom level $U_2$ (resp. $U_3$). Further, assume $M^2$ and $(M^3)'$ are generalized Verma module. Then, we have isomorphisms:
  \begin{equation}\label{eq:twisted-fusion-rules-theorem}
  \Fusion\cong \Cor[\Sigma_{1}((U^3)^\ast, M^1, U^2)]\cong ((U^3)^\ast\otimes_{A_g(V)}A_{g}(M^1)\otimes_{A_g(V)}U^2)^{\ast}.
  \end{equation}
  In particular, if $V$ is $g$-rational, then \labelcref{eq:twisted-fusion-rules-theorem} holds for any untwisted irreducible module $M^1$, and irreducible $g$-twisted modules $M^2$ and $M^3$. 
\end{theorem}
\begin{proof}
    It follows from \cref{thm:I=Cor,thm:iso-restrictcfb-bottomcorrelation,coro:isomorphism-between-correlations,pro:iso-between-different-bimodule-tensors}.
\end{proof}

In general, we have the following upper bound of the $g$-twisted fusion rules by \cref{prop:injectivity-from-cor-to-corbot}: 
\begin{equation}\label{eq:half-twisted-fusion-rules-theorem}
  \Nusion \le \dim \left((M^3(0))^\ast\otimes_{A_g(V)}A_{g}(M^1)\otimes_{A_g(V)}M^2(0)\right)^{\ast},
\end{equation}
where $M^2$ and $M^3$ are irreducible $g$-twisted modules with bottom levels $M^2(0)$ and $M^3(0)$, respectively, and $V$ is an arbitrary VOA. 

\subsection{The finiteness of twisted conformal blocks and fusion rules}
In this subsection, we assume that $V$ is of CFT-type, i.e. $V=V_0\oplus V_+$, with $V_0=\C \vac$, and $V_+=\bigoplus_{n=1}^\infty V_n$.

The finiteness of fusion rules is one of the standard assumptions for rational conformal field theory \cite{MS89}. Li proved the finiteness of fusion rules among three irreducible untwisted modules $M^1, M^2$, and $M^3$ when the module $M^1$ is $C_2$-cofinite \cite{Li2}. Using \cref{thm:twisted-fusion-rules-theorem} and \labelcref{eq:half-twisted-fusion-rules-theorem}, we can show the finiteness of $g$-twisted fusion rules under the same condition.

We observe that the filtrations on $A(V)$ and $A(M)$ studied in \cite{Liu} also have a $g$-twisted analog. Define a filtration: $0=A_g(V)_{-1}\subset A_g(V)_{0}\subset A_g(V)_{1}\subset \cdots$ on $A_g(V)$ by 
\begin{equation}\label{eq:filtration}
A_g(V)_n:=\left(\bigoplus_{i=0}^n V_i+ O_g(V)\right)/O_g(V),\quad n\in \Z.
\end{equation}
It is clear that $A_g(V)_m\ast_g A_g(V)_n\subset A_g(V)_{m+n}$ for any $m,n\in \N$. Denote the associated graded algebra by $\gr A_g(V)=\bigoplus_{n=0}^\infty (A_g(V)_n/A_g(V)_{n-1})$. The product $\cdot\ast_g \cdot$ on $\gr A_g(V)$ is given by 
\begin{equation}\label{eq:product-in-filtration}
([a]+A_g(V)_{m-1})\ast_g ([b]+A_g(V)_{n-1}):=[a]\ast_g [b]+A_g(V)_{m+n-1}, 
\end{equation}
where $[a]\in A_g(V)_{m}$ and $[b]\in A_g(V)_{n}$. It is easy to check this product is commutative. 

Recall that $R(V)=V/C_2(V)$ is a commutative associative algebra with product $(a+C_2(V))\cdot (b+C_2(V))=\vo{a}{-1}b+C_2(V)$, see \cite{Z}. The following Proposition is a $g$-twisted version of \cite[Theorem 2.6]{Liu} and a refinement of \cite[Proposition 3.6]{DLM00}: 
\begin{proposition}\label{pro:R(V)}
  The linear map $\phi: V\longrightarrow \gr A_g(V)$, $\phi(a)=[a]+A_g(V)_{m-1}$ for $a\in \oplus_{i=0}^m V_i$, factors through $C_2(V)$. It induces an epimorphism of commutative associative algebras: 
  \begin{equation}\label{def:R(V)-map}
  \phi: R(V)=V/C_2(V)\longrightarrow \gr A_g(V),\quad \phi(a+C_2(V))=[a]+A_g(V)_m,\ \mathrm{where}\ a\in \oplus_{i=0}^m V_i.
  \end{equation}
\end{proposition}
\begin{proof}
	First we show that $\phi(C_2(V))=0$. Let $a\in V^r_m$ and $b\in V_n$. Since $\wt (\vo{a}{-2}b)=m+n+1$, we have $\phi(\vo{a}{-2}b)=[\vo{a}{-2}b]+A_g(V)_{m+n}$. By Lemma 2.2 in \cite{DLM1}, we have
	\begin{equation}\label{6.13''}
	\Res_z Y(a,z)b\frac{(1+z)^{m-1+\delta(r)+\frac{r}{T}}}{z^{1+\delta(r)+k}}\in O_g(V),\quad k\ge 0.
	\end{equation}
	We may choose $k=0$ if $\delta(r)=1$, and $k=1$ if $\delta(r)=0$. Then, 
	\[[\vo{a}{-2}b]+A_g(V)_{m+n}=-\sum_{j\ge 0} \binom{m-1+\delta(r)+\frac{r}{T}}{j+1} [\vo{a}{j-1}b]+A_g(V)_{m+n}=[0]+A_g(V)_{m+n} \]
	since $\wt \vo{a}{j-1}b=m+n-j\le m+n $ for any $j\ge 0$. Hence $\phi$ in \labelcref{def:R(V)-map} is well-defined. 
	
	Moreover, $\phi((a+C_2(v))\cdot (b+C_2(V)))=\phi(\vo{a}{-1}b+C_2(V))=[\vo{a}{-1}b]+A_g(V)_{m+n-1}$. If $0<r<T$, by \labelcref{6.13''} with $k=0$, \labelcref{def:g-star-product} and \labelcref{eq:product-in-filtration}, we have 
	\begin{align*}
	[\vo{a}{-1}b]+A_g(V)_{m+n-1}&=-\sum_{j\ge 0} \binom{m-1+\frac{r}{T}}{j+1}[\vo{a}{j}b]+A_g(V)_{m+n-1}=[0]+A_{g}(V)_{m+n-1}\\
	&=[a]\ast_g [b]+A_g(V)_{m+n-1}=\phi(a+C_2(V))\ast_g \phi(b+C_2(V))
	\end{align*}
	since $\wt \vo{a}{j}b=m+n-j-1\le m+n-1$ for any $j\ge 0$. Finally, if $r=0$, by \labelcref{def:g-star-product} again, 
	\begin{align*}
	[a]\ast_g [b]+A_g(V)_{m+n-1}&=\sum_{j\ge 0}\binom{m}{j}[\vo{a}{j-1}b]+A_g(V)_{m+n-1}=[\vo{a}{-1}b]+A_g(V)_{m+n-1}.
	\end{align*}
	Thus, $\phi$ is an epimorphism of graded commutative associative algebras. 	
\end{proof}
Recall that a CFT-type VOA $V$ is called $C_1$-cofinite, if $\dim V/C_1(V)<\infty$, where $C_1(V)=\spn\Set*{L_{(-1)}a:a\in V}\cup \Set*{a_{(-1)}b\given a,b\in V_+}$. Karel and Li proved in \cite{KL98} that $R(V)$ is finitely generated when $V$ is $C_1$-cofinite. 
As an immediate Corollary, we have the following fact about $A_g(V)$: 
\begin{corollary}
  $A_g(V)$ is an noetherian algebra if $V$ is $C_1$-cofinite. 
\end{corollary}

Let $M=\bigoplus_{n=0}^\infty M_{\lambda+n}$ be an ordinary untwisted module of conformal weight $\lambda$. We introduce a similar filtration $0=A_g(M)_{-1}\subset A_g(M)_{0}\subset A_g(M)_1\subset \cdots$ by 
\[A_g(M)_n:=\left(\bigoplus_{i=0}^n M_{\lambda+i}+ O_g(M)\right)/O_g(M),\quad n\in \Z.\]
It is clear that $A_g(M)$ becomes a filtered $A_g(V)$-bimodule with this filtration and filtration \labelcref{eq:filtration} of $A_g(V)$, and $\gr A_g(M)$ is a graded $\gr A_g(V)$-module. The following fact is similar to \cref{pro:R(V)}, and the proof is also similar: 
\begin{proposition}
  Let $M$ be an untwisted module. There exists a epimomorphism of $R(V)$-modules 
  \begin{equation}\label{eq:epimophism-R(V)}
  \psi: M/C_2(M)\longrightarrow \gr A_g(M), \quad \psi(u+C_2(M))=[u]+A_g(M)_{m},\ \mathrm{where}\ u\in \oplus_{i=0}^m M_{\lambda+i}. 
  \end{equation}
  \end{proposition} 

\begin{corollary}
  If $V$ is $C_2$-cofinite, the 
  the twisted fusion rules among irreducible untwisted module $M^1$, and $g$-twisted modules $M^2$ and $M^3$ are finite.
  \end{corollary}
\begin{proof}
  It was proved by Buhl in \cite{Bu02} that an irreducible untwisted module $M$ is $C_2$-cofinite if $V$ is $C_2$-cofinite. In particular, we have $\dim A_g(M^1)=\dim \gr A_g(M^1)<\infty$ in view of \labelcref{eq:epimophism-R(V)}. Then, the conclusion follows from the estimate \labelcref{eq:half-twisted-fusion-rules-theorem}.
\end{proof}
\subsection{Twisted fusion rules and fusion rules for cyclic orbifold VOAs}
Let $V$ be a strongly rational VOA, and $g\in \Aut(V)$ be an automorphism of finite order $T$. According to \cite{M15,CM16}, the cyclic orbifold subVOA $V^{0}=V^{\langle g\rangle}$ is also strongly rational. 

We again let $M^1$ be an irreducible untwisted module, and $M^2$ and $M^3$ be irreducible $g$-twisted modules. Note that $M^1$, $M^2$, and $M^3$ are also ordinary $V^0$-modules. 
With \cref{thm:twisted-fusion-rules-theorem} and \cite[Theorem 2.11]{Li} or \cite[Theorem 4.19]{Liu}, we can find a concrete relation between the space of twisted intertwining operators $\Fusion$, and the space of ordinary intertwining operators of $V^0$-modules, which we denote by  $\mathfrak{I}_{V^0}\fusion$. 

By Theorem 3.1 in \cite{DRX17}, the twisted modules $M^1, M^2$, and $M^3$ decompose into direct sum of irreducible ordinary $V^0$-modules: 
$M^1=\bigoplus_{i=0}^{m_1} M^{1,i},$ $ M^2=\bigoplus_{j=0}^{m_2} M^{2,j},$ and $ M^3=\bigoplus_{k=0}^{m_3} M^{3,k},$
where the direct summands could appear multiple times. Denote the $A(V^0)$-bimodule $M^{1,i}/O_{V^0}(M^{1,i})$ by $A_{V^0}(M^{1,i})$ for all $i$. Since $a\circ M^{1,i}\subset M^{1,i}$ for all $i$ and $a\in V^0$, we have
 \[A_{V^0}(M^1)=M^1/O_{V^0}(M^1)\cong \bigoplus_{i=0}^{m_1} M^{1,i}/O_{V^0}(M^{1,i})=\bigoplus_{i=0}^{m_1} A_{V^0}(M^{1,i}).\] 
Moreover, since $V^0$ is also $C_2$-cofinite \cite{M15}, we have $\dim  \mathfrak{I}_{V^0}\fusion[M^{1,i}][M^{2,j}][M^{3,k}]<\infty$ for all $i,j,$ and $k$. By taking restrictions and projections onto the direct summands, using \cite[Theorem 4.19]{Liu} and the fact that $V^0$ is rational \cite{CM16}, we have the following identification of $\mathfrak{I}_{V^0}\fusion$: 
\begin{align*}
\mathfrak{I}_{V^0}\fusion&=\bigoplus_{i,j,k} \mathfrak{I}_{V^0}\fusion[M^{1,i}][M^{2,j}][M^{3,k}]
\cong \left( \bigoplus_{i,j,k} M^{3,k}(0)^{\ast}\o _{A(V^0)} A_{V^0}(M^{1,i})\otimes_{A(V^0)} M^{2,j}(0)\right)^{\ast} \\
&\cong (M^3(0)^{\ast}\o _{A(V^0)} M^1/O_{V^0}(M^1)\otimes_{A(V^0)} M^2(0))^{\ast}.
\end{align*}
On the other hand, we have $\Fusion\cong (M^3(0)^{\ast}\o _{A(V^0)} M^1/O_{g}(M^1)\otimes_{A(V^0)} M^2(0))^{\ast},$ in view of \cref{thm:twisted-fusion-rules-theorem} and \labelcref{eq:iso-between-different-bimodule-tensors}. By the proof of \cref{lem:bimodule}, $O_{g}(M^1)/O_{V^0}(M^1)$ is an $A(V^0)$-sub-bimodule of $A_g(M^1)=M^1/O_g(M^1)$. Thus, we have the following 
\begin{proposition}
	With the settings as above, we have the following relation between the fusion rules of $g$-twisted modules and fusion rules of ordinary $V^0$-modules: 
\begin{equation}\label{n6.16}
\begin{aligned}
    &\sum_{i,j,k}\Nusion[M^{1,i}][M^{2,j}][M^{3,k}]\\
    ={}&\Nusion+\dim  (M^3(0)^{\ast}\o _{A(V^0)} \frac{O_g(M^1)}{O_{V^0}(M^1)}\otimes_{A(V^0)} M^2(0)).
\end{aligned}
\end{equation} 
\end{proposition}

\section{Fusion rules among \texorpdfstring{$\theta$}{theta}-twisted modules over the Heisenberg and lattice VOAs}\label{sec7}
In this Section, we apply \cref{thm:twisted-fusion-rules-theorem} to compute the fusion rules among $\theta$-twisted modules over the Heisenberg VOA and rank one lattice VOA. We refer to \cite{FLM, D93, DL} for the detailed constructions of the Heisenberg VOAs and lattice VOAs. 

Here we recall the definition of involution $\theta$ in \cite{FLM}. Let $L$ be a positive definite even lattice of rank $d>0$, and $\theta:L\longrightarrow L$ be an involution of $L$ defined by $\theta(\alpha)=-\alpha$, for any $\alpha\in L$. Then, $\theta$ lifts to an involution of the Heisenberg VOA $M(1)$ associated to $\mathfrak{h}=\C\otimes_{\Z} L$ and the lattice VOA $V_L=M(1)\otimes \C^\epsilon[L]$ as follows: 
\begin{align}
&\theta: M(1)\longrightarrow M(1),\quad \theta(\alpha^1(-n_1)\cdots \alpha^k(-n_k)\vac):=(-1)^k \alpha^1(-n_1)\cdots \alpha^k(-n_k)\vac,\label{6.8}\\
&\theta: V_L\longrightarrow V_L,\quad \theta(\alpha^1(-n_1)\cdots \alpha^k(-n_k) e^\alpha):= (-1)^k\alpha^1(-n_1)\cdots \alpha^k(-n_k)e^{-\alpha},\label{6.9}
\end{align}
where $\alpha^1,\cdots, \alpha^k\in \mathfrak{h}$, $n_1,\cdots, n_k\ge 1$, and $\alpha\in L$. Clearly, $\theta^2=1$. The $\theta$-eigenspaces of eigenvalue $1$ in $M(1)$ and $V_L$ are denoted by $M(1)^+$ and $V_L^+$, respectively, while the $-1$-eigenspaces are denoted by $M(1)^-$ and $V_L^-$. 

Since $M(1)$ and $V_L$ are both simple VOAs, by \cite[Theorem 3.4]{X95}, the only possible nonzero intertwining operators among $\theta$-twisted modules are of the type $\fusion[1][\theta][\theta]$, $\fusion[\theta][1][\theta]$, or $\fusion[\theta][\theta][1]$. Here type $\fusion[g_1][g_2][g_3]$ means type $\fusion$, where $M^i$ is a $g_i$-twisted module for $i=1, 2 ,3$. On the other hand, by \cite[Corollary 5.2 and Corollary 6.2]{H18}, we have 
\[
\Fusion[\theta][\theta][1]\cong\Fusion[\theta][1][\theta^{-1}]\cong\Fusion[1][\theta][\theta].
\]
Therefore, in order to determine the fusion rules among $\theta$-twisted modules, where $V=M(1)$ or $V_L$, we only need to determine the space of $\Fusion[1][\theta][\theta]$-twisted intertwining operators. 

\subsection{The Heisenberg VOA case} Let $\mathfrak{h}=\C\otimes_\Z L$ or any $d$-dimensional $\C$-vector space, equipped with a non-degenerate bilinear form $(\cdot|\cdot)$. Recall the twisted affine algebra  $\hat{\mathfrak{h}}_{\Z+\frac{1}{2}}=\mathfrak{h}\o t^{1/2}\C[t,t^{-1}]\oplus \C K$, with Lie bracket given by 
\[
[a(m),b(n)]=\delta_{m+n,0}m(a|b)K,\quad  [K,\hat{\mathfrak{h}}_{\Z+\frac{1}{2}}]=0, \quad a,b\in \mathfrak{h}, m,n\in \Z+\frac{1}{2}.
\]
Let $M(1)_{\Z+\frac{1}{2}}$ be the induced module $U(\hat{\mathfrak{h}}_{\Z+\frac{1}{2}})\o _{U(\hat{\mathfrak{h}}_{\Z+1/2}^+\oplus \C K)}\C \vac$, where $\hat{\mathfrak{h}}_{\Z+1/2}^+\oplus \C K$ acts trivially on $\C \vac$ \cite{FLM,DN99}. By Corollary 3.9 in \cite{DN99}, the $\theta$-twisted Zhu's algebra $A_\theta(M(1))$ is isomorphic to $\C$, and $M(1)_{\Z+\frac{1}{2}}$ is the unique $\theta$-twisted $M(1)$-module.  

It is clear that $M(1)_{\Z+\frac{1}{2}}$ is an irreducible module over twisted affine Lie algebra $\hat{\mathfrak{h}}_{\Z+\frac{1}{2}}$. Since $M(1)_{\Z+\frac{1}{2}}$ is universal in the sense that any $\theta$-twisted $M(1)$-module with bottom level $\C \vac$ is a quotient module of $M(1)_{\Z+\frac{1}{2}}$, it is also the $\theta$-twisted generalized Verma module over $M(1)$. 

\begin{lemma}\label{lem: Heisenberg}
  Let $M(1,\la)=M(1)\otimes \C e^\la$ be the irreducible (untwisted) module over $M(1)$ associated to $\la\in \h$. Then, $A_{\theta}(M(1,\la))=\C [e^\la]$. 
  \end{lemma}
\begin{proof}
  The proof is similar to the proof of Lemma 3.7 and 3.8 in \cite{DN99}. We briefly sketch it. For any $\alpha\in \h$, since $\alpha(-1)\vac \in M(1)^-$, by \labelcref{def:O-quotient}, we have
  \begin{equation} \label{6.11}
  \alpha(-m-1)v\equiv -\sum_{j\ge 0} \binom{1/2}{j+1} \alpha(j-m)v\pmod{O_\theta(M(1,\la))},
  \end{equation}
  for any $m\ge 0$, $v\in M(1,\la)$. We use induction on the degree $n_1+\cdots +n_k=n$ of a spanning element $v=\alpha^1(-n_1)\cdots \alpha^k(-n_k)e^\la$ of $M(1,\la)$ to show that $v\equiv c e^\la\pmod{O_\theta}(M(1,\la))$, where $c\in \C$. The base case $n=0$ is clear. For $n>0$, by \labelcref{6.11} we have
    \begin{align*}
        \MoveEqLeft
        \alpha^1(-n_1)\cdots \alpha^k(-n_k)e^\la\\
        &\equiv -\sum_{j\ge  0}\binom{1/2}{j+1}\alpha^1(j-n_1+1)\alpha^2(-n_2)\cdots \alpha^k(-n_k)e^\la\pmod{O_\theta(M(1,\la))},
    \end{align*}
where $\deg (\alpha^1(j-n_1+1)\alpha^2(-n_2)\cdots \alpha^k(-n_k)e^\la)=n_1-j-1+n_2+\cdots +n_k=n-j-1<n$. Then, by the induction hypothesis, $v\equiv c e^\la \pmod{O_\theta(M(1,\la))}$. 
\end{proof}
 
We note that \cref{lem: Heisenberg} only shows $A_{\theta}(M(1,\la))$ is at most one-dimensional. It is not obvious that $[e^\la]\neq 0$. Using \cref{lem: Heisenberg}, \cref{thm:twisted-fusion-rules-theorem}, the Hom-tensor duality, and the fact $A_\theta(M(1))\cong \C$, we have 
\begin{equation}\label{6.12'}
 \Nusion[M(1,\la)][M_{\Z+\frac{1}{2}}(1)][M_{\Z+\frac{1}{2}}(1)]= \dim \Hom_{\C}(\C [e^\la]\otimes \C \vac, \C\vac)\le 1.
\end{equation}
 On the other hand,  Abe, Dong, and Li constructed a nonzero $\fusion[1][\theta][\theta]$-twisted intertwining operator in \cite{ADL05} based on twisted vertex operators from \cite{FLM}:
 \begin{equation}\label{6.13'}
 \begin{aligned}
& \mathcal{Y}_{\la}^{\mathrm{tw}}(\cdot,w): M(1,\la)\longrightarrow \End(M_{\Z+\frac{1}{2}}(1))\{w\},\\
&  \mathcal{Y}_{\la}^{\mathrm{tw}}(e^\la,w)=e^{-|\la|^2\log 2}z^{-\frac{|\la|^2}{2}}\exp(\sum_{n\in -\frac{1}{2}-\N} \frac{-\la(n)}{n}z^{-n}) \exp(\sum_{n\in \frac{1}{2}+\N} \frac{-\la(n)}{n}z^{-n}),
 \end{aligned}
 \end{equation}
see equation (4.7) in \cite{ADL05}.  
Using \labelcref{6.12'}, we have the following result about the fusion rules among $\theta$-twisted module over $M(1)$: 
\begin{proposition}
 Let $\la\in \h^{\ast}$. Then, $\Nusion[M(1,\la)][M_{\Z+\frac{1}{2}}(1)][M_{\Z+\frac{1}{2}}(1)]=1$. 
\end{proposition}
 
 \subsection{The rank one lattice VOA case} 
 In this subsection,  we assume that $L=\Z \alpha$, with $(\alpha|\alpha)=2$ and $\epsilon (\alpha,\alpha)=1$. i.e., $\epsilon: L\times L\longrightarrow <\pm 1>$ is trivial. Then, $L$ is the root lattice of type $A_1$, with dual lattice $L^\circ=\frac{1}{2}L=L\sqcup (L+\frac{1}{2}\alpha)$. We first recall some general results about twisted representations of $V_L$ in \cite{FLM,D93,D94,DN99}. 
 
 By \cite[Theorem 3.1]{D93}, $V_L$ has two untwisted irreducible modules $V_L$ and $V_{L+\frac{1}{2}\alpha}$. On the other hand, according to \cite[Theorem 3.5.1, Remark 3.5.3, and Remark 7.4.14]{FLM}, together with \cite[Theorem 3.1]{D94}, $V_L$ has two irreducible $\theta$-twisted modules $V_L^{T_\chi}$ and $V_L^{T_{-\chi}}$, with bottom levels $T_\chi=\C v_\chi$ and $T_{-\chi}=\C v_{-\chi}$, respectively, where $\chi\in \C^\times$, and
\[V_L^{T_\chi}=M_{\Z+\frac{1}{2}}(1)\otimes \C v_{\chi},\quad V_L^{T_{-\chi}}=M_{\Z+\frac{1}{2}}(1)\otimes \C v_{-\chi}.\]
Moreover, consider the Lie algebra $sl_2=\C e^\alpha+\C \alpha(-1)\vac+\C e^{-\alpha}=(V_L)_1$, let $E=e^\alpha+e^{-\alpha}$ and $F=e^{\alpha}-e^{-\alpha}$, and let 
\[\hat{sl_2}[\theta_2]=(\C E\o \C[t,t^{-1}])\oplus ((\C\alpha(-1)\vac+\C F)\o t^{1/2}\C[t,t^{-1}])\oplus \C K\] be the twisted affine Lie algebra associated to $sl_2$ and the involution
\[\theta_2: sl_2\longrightarrow sl_2,\quad e^\alpha\mapsto e^{-\alpha},\quad  \alpha(-1)\vac\mapsto -\alpha(-1)\vac,\quad e^{-\alpha}\mapsto e^{\alpha},\]
which is $\theta$ in \labelcref{6.9} restricted to $(V_L)_1$. See Chapters 2 and 3 of \cite{FLM} for more details. Then, $V_L^{T_\chi}$ and $V_L^{T_{-\chi}}$ are non-isomorphic irreducible modules over the twisted affine Lie algebra $\hat{sl_2}[\theta_2]$. Let $x(n):=x\o t^n\in \hat{sl_2}[\theta_2]$, for any $x\in sl_2$ and $n\in \Z$ or $\Z+\frac{1}{2}$. By the construction of twisted modules, the action of $\vo{E}{0}$ on $T_\chi$ and $T_{-\chi}$ are given by 
\begin{equation}\label{6.12}
\vo{E}{0}v_\chi=\frac{1}{2} v_\chi,\quad \vo{E}{0} v_{-\chi}=-\frac{1}{2} v_{-\chi}.
\end{equation}
See Section 5.1 in \cite{DN99(1)}. 

By Lemma 3.10 and Proposition 3.12 in \cite{DN99}, $A_\theta(V_L)$ has the following characterization: 
\begin{lemma}
  $A_\theta(V_L)\cong \C[\vac]\oplus\C[e^\alpha]$, with $[e^\alpha]=[e^{-\alpha}]$ and $[e^\alpha]\ast_\theta [e^\alpha]=4^{-(\alpha|\alpha)}[\vac]$. 
  \end{lemma}
Moreover, by Theorem 3.13 in \cite{DN99}, $V_L$ is $\theta$-rational. Thus we can apply the twisted fusion rules formula \labelcref{eq:twisted-fusion-rules-theorem} for $V_L$. 
\begin{proposition}
  Let $M^1$ be the untwisted $V_L$-module $V_L$, we have
  \begin{equation}\label{6.13}
  \Nusion[V_L][V_L^{T_\chi}][V_L^{T_{\chi}}]=\Nusion[V_L][V_L^{T_{-\chi}}][V_L^{T_{-\chi}}]=1,\quad \Nusion[V_L][V_L^{T_\chi}][V_L^{T_{-\chi}}]=\Nusion[V_L][V_L^{T_{-\chi}}][V_L^{T_{\chi}}]=0.
  \end{equation}
  \end{proposition}
\begin{proof}
  By \cref{thm:twisted-fusion-rules-theorem} and Hom-tensor duality, we have 
  \[\Fusion[V_L][V_L^{T_{\pm \chi}}][V_L^{T_{\pm\chi}}]\cong \Hom_{A_\theta(V_L)}(T_{\pm \chi},T_{\pm \chi}).\]
  Now \labelcref{6.13} follows from \labelcref{6.12} since $[E]=2[e^\alpha]\in A_{\theta}(V_L)$ acts on $T_{\pm \chi}=\C v_{\pm \chi}$ by $o(E)v_{\pm \chi}=\vo{E}{0}v_{\pm \chi}=\pm (1/2) v_{\pm \chi}$, and an element $f$ in  $\Hom_{A_\theta(V_L)}(T_{\pm \chi},T_{\pm \chi})$ preserves $\vo{E}{0}$. 
\end{proof}
It remains to consider the case when $M^1$ is the untwisted irreducible $V_L$-module $V_{L+\frac{1}{2}\alpha}$. We first give a spanning set of the $A_\theta(V_L)$-bimodule $A_\theta(V_{L+\frac{1}{2}\alpha})$. 
\begin{lemma}\label{lem:latice}
  $A_{\theta}(V_{L+\frac{1}{2}\alpha})=\C [e^{\frac{1}{2} \alpha}]+\C [e^{-\frac{1}{2}\alpha}]$, with 
  \begin{equation}\label{6.14}
  [E]\ast_\theta [e^{\frac{1}{2}\alpha}]-[e^{\frac{1}{2}\alpha}]\ast_\theta[E]=[e^{-\frac{1}{2}\alpha}],\quad   [E]\ast_\theta [e^{-\frac{1}{2}\alpha}]-[e^{-\frac{1}{2}\alpha}]\ast_\theta[E]=[e^{\frac{1}{2}\alpha}]. 
  \end{equation}
\end{lemma}
\begin{proof}
  Since $\alpha(-1)\vac \in V_L^-$, we have a congruence formula similar to \labelcref{6.11}:
$$
    \alpha(-m-1)v\equiv -\sum_{j\ge 0} \binom{1/2}{j+1} \alpha(j-m)v\pmod{O_\theta (V_{L+\frac{1}{2}\alpha})}, 
$$
  where $v\in V_{L+\frac{1}{2}\alpha}$ and $m\ge 0$. In particular, let $m=1$, we have 
    \begin{equation}\label{6.15}
  \alpha(-1)e^{\frac{1}{2}\alpha} \equiv -(1/2) e^{\frac{1}{2}\alpha}\quad \mathrm{and}\quad   \alpha(-1)e^{-\frac{1}{2}\alpha} \equiv (1/2) e^{-\frac{1}{2}\alpha} \pmod{O_\theta (V_{L+\frac{1}{2}\alpha})}. 
  \end{equation}

  Moreover, given $r\in \Z$ and $u=\alpha^1(-n_1)\cdots \alpha^k(-n_k)e^{\frac{2r+1}{2}\alpha}\in M(1, \frac{2r+1}{2}\alpha)\subset V_{L+\frac{1}{2}\alpha}$, using a similar induction process as Lemma~\ref{lem: Heisenberg} on the degree $n_1+\cdots +n_k$ of $u$, we can show that 
    \begin{equation}\label{6.16}
u=\alpha^1(-n_1)\cdots \alpha^k(-n_k)e^{\frac{2r+1}{2}\alpha}\equiv b_{u} e^{\frac{2r+1}{2}\alpha}\pmod{O_\theta (V_{L+\frac{1}{2}\alpha})},
\end{equation}
for some constant $b_u\in \C$.   Now we use induction on $r\in \N$ to show that 
  \begin{equation}\label{6.17}
u=\alpha^1(-n_1)\cdots \alpha^k(-n_k)e^{\frac{2r+1}{2}\alpha}\equiv c_{u} e^{\frac{1}{2}\alpha}\ \  \mathrm{or}\ \ d_{u} e^{-\frac{1}{2}\alpha}\pmod{O_\theta (V_{L+\frac{1}{2}\alpha})},
  \end{equation} 
for some constants $c_{u}, d_u\in \C$, where $k,r\ge 0$, $n_1\ge \cdots \ge n_k\ge 1$, and $\alpha^1,\cdots,\alpha^k\in \h$.
  
When $r=0$, \labelcref{6.17} follows from \labelcref{6.16}. Consider the case where $r=1$. Note that $E=e^\alpha+e^{-\alpha}\in V_L^+$ and $\wt E=1$. It follows from \labelcref{def:circle-g} and \labelcref{6.2'} that
\begin{equation}\label{6.18}
\vo{E}{-m-2}v+\vo{E}{-m-1}v\equiv 0\pmod{O_\theta (V_{L+\frac{1}{2}\alpha})}, \quad m\ge 0, v\in V_{L+\frac{1}{2}\alpha}.
\end{equation} 
By the definition of lattice vertex operators in \cite{FLM} and \labelcref{6.16}, we have 
\begin{align*}
\vo{(e^\alpha)}{-2}e^{\frac{1}{2}\alpha}\equiv& e^{\frac{3}{2}\alpha},\quad \vo{(e^\alpha)}{-1} e^{\frac{1}{2}\alpha}\equiv 0\pmod{O_\theta (V_{L+\frac{1}{2}\alpha})},\\
\vo{(e^{-\alpha})}{-2}e^{\frac{1}{2}\alpha}={}&\Res_z E^-(\alpha, z) z^{-3} e^{-\frac{1}{2}\alpha}=-\frac{\alpha(-2)}{2}e^{-\frac{1}{2}\alpha}+\frac{1}{2}\alpha(-1)^2e^{-\frac{1}{2}\alpha}\\
\equiv& \la e^{-\frac{1}{2}\alpha}\pmod{O_\theta (V_{L+\frac{1}{2}\alpha})},\\
\vo{(e^{-\alpha})}{-1} e^{\frac{1}{2}\alpha}={}&\Res_z E^-(\alpha, z) z^{-2}e^{-\frac{1}{2}\alpha}= -\alpha(-1) e^{-\frac{1}{2}\alpha}\equiv -(1/2) e^{-\frac{1}{2}\alpha}\pmod{O_\theta (V_{L+\frac{1}{2}\alpha})}. 
\end{align*}
Choose $m=0$ in \labelcref{6.18} we have: 
\[\vo{E}{-2}e^{\frac{1}{2}\alpha}+\vo{E}{-1}e^{\frac{1}{2}\alpha}\equiv e^{\frac{3}{2}\alpha}+ \la e^{-\frac{1}{2}\alpha}-(1/2)e^{-\frac{1}{2}\alpha}\equiv 0\pmod{O_\theta (V_{L+\frac{1}{2}\alpha})}. \]
Hence $\alpha^1(-n_1)\cdots \alpha^k(-n_k)e^{\frac{3}{2}\alpha}\equiv b_u e^{\frac{3}{2}\alpha}\equiv  b_u(1/2-\la)e^{-\frac{1}{2}\alpha}\pmod{O_\theta (V_{L+\frac{1}{2}\alpha})}$, in view of \labelcref{6.16}. This proves \labelcref{6.17} when $r=1$ since $b_u(1/2-\la)$ is a constant. 

Now suppose $r>1$, and the conclusion holds for smaller $r$.  Let $m=2r$ in \labelcref{6.18}, by the induction hypothesis, we have
  \begin{align*}
  \vo{(e^\alpha)}{-2r-2} e^{\frac{2r+1}{2}\alpha}={}&\Res_z z^{-2r-2} E^{-}(-\alpha, z)e^{\frac{2r+3}{2}\alpha} z^{2r+1}= e^{\frac{2r+3}{2}\alpha}\\
  \vo{(e^{-\alpha})}{-2r-2} e^{\frac{2r+1}{2}\alpha}={}&\Res_z E^{-}(\alpha,z)e^{\frac{2r-1}{2}\alpha}  z^{-4r-1}\in M(1,(2r-1)\alpha/2)\\ \equiv& c_{r-1} e^{\pm\frac{1}{2}\alpha} \pmod{O_\theta (V_{L+\frac{1}{2}\alpha})},  \\
  \vo{(e^{\alpha})}{-2r-1} e^{\frac{2r+1}{2}\alpha} ={}&\Res_z z^{-2r-1} E^-(-\alpha,z) e^{\frac{2r+3}{2}\alpha} z^{2r+1} =0,\\
  \vo{(e^{-\alpha})}{-2r-1} e^{\frac{2r+1}{2}\alpha} ={}&\Res_z  E^-(\alpha,z) e^{\frac{2r-1}{2}\alpha} z^{-4r-2} \in M(1, (2r-1)/2)\\
  \equiv& c'_{r-1}e^{\pm\frac{1}{2}\alpha}\pmod{O_\theta (V_{L+\frac{1}{2}\alpha})},
  \end{align*}
  where $e^{\pm \frac{1}{2}\alpha}$ attains the same sign in the second and fourth congruence equations.
By \labelcref{6.18}, 
  \begin{align*}
  e^{\frac{2r+3}{2}\alpha}&\equiv \vo{E}{-2r-2}e^{\frac{2r+1}{2}\alpha}-c_{r-1}e^{\pm\frac{1}{2}\alpha}\equiv -\vo{E}{-2r-1} e^{\frac{2r+1}{2}\alpha}-c_{r-1}e^{\pm \frac{1}{2}\alpha}\\ &\equiv -c'_{r-1}e^{\pm\frac{1}{2}\alpha}- c_{r-1}e^{\pm\frac{1}{2}\alpha}
  \equiv p_{r} e^{\pm\frac{1}{2}\alpha} \pmod{O_\theta (V_{L+\frac{1}{2}\alpha})}. 
  \end{align*}
  where $p_{r}=-c'_{r-1}-c_{r-1}$. Now it follows from \labelcref{6.16} that
  \[\alpha^1(-n_1)\cdots \alpha^k(-n_k)e^{\frac{2r+3}{2}\alpha}\equiv q_{r} e^{\frac{2r+3}{2}\alpha}\equiv p_{r}q_{r} e^{\pm\frac{1}{2}\alpha}=c_{r} e^{\pm \frac{1}{2}\alpha}\pmod{O_\theta (V_{L+\frac{1}{2}\alpha})}.\]
  This finishes the induction step and proves \labelcref{6.17} for any $r\ge 0$. By adopting a similar induction argument, we can also prove \labelcref{6.17}for $r\in \Z_{<0}$. Since $V_{L+\frac{1}{2}\alpha}=\bigoplus_{r\in \Z} M(1, \frac{2r+1}{2}\alpha)$, by \labelcref{6.17} we have $A_{\theta}(V_{L+\frac{1}{2}\alpha})=\C [e^{\frac{1}{2} \alpha}]+\C [e^{-\frac{1}{2}\alpha}]$. Finally, by \labelcref{eq:left-right-actions}, we have 
\begin{align*}
    &[E\ast_\theta e^{\frac{1}{2}\alpha}-e^{\frac{1}{2}\alpha}\ast_\theta E]\\
    ={}&\Res_z [Y(E,z)e^{\frac{1}{2}\alpha}(1+z)^0]=[\vo{E}{0}e^{\frac{1}{2}\alpha}]=[\vo{(e^\alpha)}{0}e^{\frac{1}{2}\alpha}+\vo{(e^{-\alpha})}{0} e^{\frac{1}{2}\alpha}]=[e^{-\frac{1}{2}\alpha}].
\end{align*}
  Similarly, $[E\ast_\theta e^{-\frac{1}{2}\alpha}-e^{-\frac{1}{2}\alpha}\ast_\theta E]=[e^{\frac{1}{2}\alpha}]$. This proves \labelcref{6.14}. 
  \end{proof}

\begin{proposition}
  Let $M^1$ be the untwisted $V_L$-module $V_{L+\frac{1}{2}\alpha}$, we have 
  \begin{align}
    &\Nusion[V_{L+\frac{1}{2}\alpha}][V_L^{T_\chi}][V_L^{T_{\chi}}]=\Nusion[V_{L+\frac{1}{2}\alpha}][V_L^{T_{-\chi}}][V_L^{T_{-\chi}}]=0,\label{6.19}\\
      &\Nusion[V_{L+\frac{1}{2}\alpha}][V_L^{T_{\chi}}][V_L^{T_{-\chi}}]=\Nusion[V_{L+\frac{1}{2}\alpha}][V_L^{T_{-\chi}}][V_L^{T_{\chi}}]=1.\label{6.21}
    \end{align}
  \end{proposition}
\begin{proof}
  We show \labelcref{6.19} first. By \cref{lem:latice} we have 
    \[A_{\theta}(V_{L+\frac{1}{2}\alpha})\o _{A_\theta (V_L)} T_{\chi}=\C ([e^{\frac{1}{2}\alpha}]\o v_\chi)+  \C ([e^{-\frac{1}{2}\alpha}]\o v_{\chi}).\]
    Given $f\in \Hom_{A_\theta(V_L)}(A_{\theta}(V_{L+\frac{1}{2}\alpha})\o _{A_\theta (V_L)} T_{\chi}, T_\chi)$, we assume that 
    \begin{equation}\label{6.22}
    f([e^{\frac{1}{2}\alpha}]\o v_\chi)=\la v_\chi,\quad f([e^{-\frac{1}{2}\alpha}]\o v_\chi)=\mu v_\chi,\quad \la,\mu \in \C.
    \end{equation}
    Recall that $o(E)v_\chi=E(0)v_\chi=(1/2) v_\chi$, see \labelcref{6.12}. By \labelcref{6.14} and \labelcref{6.22} we have: 
    \begin{align*}
    f([E]\ast([e^{\frac{1}{2}\alpha}]\o v_\chi))&=f([E\ast_\theta e^{\frac{1}{2}\alpha}-e^{\frac{1}{2}\alpha}\ast_\theta E]\o v_\chi)+ f([e^{\frac{1}{2}\alpha}]\o o(E)v_\chi)\\
    &=f([e^{-\frac{1}{2}\alpha}]\o v_\chi)+ \frac{1}{2} f([e^{\frac{1}{2}\alpha}]\o v_\chi)=(\mu+\frac{\la}{2}) v_\chi,\\
      f([E]\ast([e^{-\frac{1}{2}\alpha}]\o v_\chi))&=f([E\ast_\theta e^{-\frac{1}{2}\alpha}-e^{-\frac{1}{2}\alpha}\ast_\theta E]\o v_\chi)+ f([e^{-\frac{1}{2}\alpha}]\o o(E)v_\chi)\\
    &=f([e^{\frac{1}{2}\alpha}]\o v_\chi)+ \frac{1}{2} f([e^{-\frac{1}{2}\alpha}]\o v_\chi)=(\la+\frac{\mu}{2}) v_\chi.
    \end{align*}
    On the other hand, we have $  [E]. f([e^{\frac{1}{2}\alpha}]\o v_\chi)=o(E)\la v_\chi=(\la /2) v_\chi$ and $  [E]. f([e^{-\frac{1}{2}\alpha}]\o v_\chi)=o(E)\mu v_\chi=(\mu/2) v_\chi$. Since $f$ is an $A_{\theta}(V_L)$-homomorphism, we have 
    $(\mu+(\la/2))v_\chi=(\la/2) v_\chi$ and $(\la+(\mu/2))v_\chi=(\mu/2) v_\chi$. It follows that $\la=\mu=0$, and $f=0$. By \cref{thm:twisted-fusion-rules-theorem}, we have 
    \[\Nusion[V_{L+\frac{1}{2}\alpha}][V_L^{T_\chi}][V_L^{T_{\chi}}]=\dim \Hom_{A_\theta(V_L)}(A_{\theta}(V_{L+\frac{1}{2}\alpha})\o _{A_\theta (V_L)} T_{\chi}, T_\chi)=0.\]
    Replacing $\chi$ by $-\chi$ in the argument above, it is easy to see that $\Fusion[V_{L+\frac{1}{2}\alpha}][V_L^{T_{-\chi}}][V_L^{T_{-\chi}}]=0.$ This proves \labelcref{6.19}. Next, we show \labelcref{6.21}. Given $f\in \Hom_{A_\theta(V_L)}(A_{\theta}(V_{L+\frac{1}{2}\alpha})\o _{A_\theta (V_L)} T_{\chi}, T_{-\chi}),$ assume
    \begin{equation}
      f([e^{\frac{1}{2}\alpha}]\o v_\chi)=\la v_{-\chi},\quad f([e^{-\frac{1}{2}\alpha}]\o v_\chi)=\mu v_{-\chi},\quad \la,\mu \in \C.
    \end{equation}
    With a similar argument as above, we have $(\mu+(\la/2))v_{-\chi}=-(\la/2) v_{-\chi},$ and $(\la+(\mu/2))v_{-\chi}=-(\mu/2) v_{-\chi}$. Thus, $\mu=-\la$, and $\dim \Hom_{A_\theta(V_L)}(A_{\theta}(V_{L+\frac{1}{2}\alpha})\o _{A_\theta (V_L)} T_{\chi}, T_{-\chi})\le 1$.  On the other hand, by Proposition 5.10 in \cite{ADL05}, there exists a nonzero twisted intertwining operator 
    \begin{equation}\label{6.25}
    \tilde{\mathcal{Y}}^{\mathrm{tw}}_{\alpha/2}(\cdot, w): V_{L+\frac{1}{2}\alpha}\longrightarrow \Hom(V_{L}^{T_\chi}, V_L^{T_{\chi}^{(\alpha/2)}}),\quad \tilde{\mathcal{Y}}^{\mathrm{tw}}_{\alpha/2}(u ,w)=\mathcal{Y}_{\alpha/2}^{\mathrm{tw}}(u,w)\o \eta_{(\alpha/2)+\beta},
    \end{equation}
    where $u\in M(1, (\alpha/2)+\beta)$, $\mathcal{Y}_{\alpha/2}^{\mathrm{tw}}$ is given by \labelcref{6.13'}, $\eta_{(\alpha/2)+\beta}:T_{\chi}\longrightarrow T_{\chi}^{(\alpha/2)}$ is a linear isomorphism, and $T_{\chi}^{(\alpha/2)}=T_{-\chi}$ by \labelcref{6.12} and the construction in Section 5.3 in \cite{ADL05}. Thus we have 
       \[\Nusion[V_{L+\frac{1}{2}\alpha}][V_L^{T_\chi}][V_L^{T_{-\chi}}]=\dim \Hom_{A_\theta(V_L)}(A_{\theta}(V_{L+\frac{1}{2}\alpha})\o _{A_\theta (V_L)} T_{\chi}, T_{-\chi})=1,\]
       and the second equality in \labelcref{6.21} can be proved by a similar method.
  \end{proof}

\begin{remark}
We need to use the nonzero twisted intertwining operator $  \tilde{\mathcal{Y}}^{\mathrm{tw}}_{\alpha/2}$ in \cite{ADL05} for the proof of \labelcref{6.21} since it is not clear from \cref{lem:latice} that $A_{\theta}(V_{L+\frac{1}{2}\alpha})=\C [e^{\frac{1}{2} \alpha}]+\C [e^{-\frac{1}{2}\alpha}]$ is nonzero. Although we cannot achieve here, we believe there is an intrinsic proof of the facts that $A_{\theta}(M(1,\la))=\C [e^\la]$ is nonzero for any $\la\in \h$, and that $A_{\theta}(V_{L+\frac{1}{2}\alpha})=\C[e^{\frac{1}{2} \alpha}]\oplus \C [e^{-\frac{1}{2}\alpha}]$ is a two-dimensional vector space.
\end{remark}

\backmatter





\bmhead{Acknowledgments}

We thank Professors Yi-Zhi Huang, James Lepowsky, and Angela Gibney for their valuable discussions and suggestions.

\section*{Declarations}

No funding was received for conducting this study.
The authors have no competing interests to declare that are relevant to the content of this article.

\addcontentsline{toc}{section}{References}
\bibliography{bib}

\end{document}